\documentclass[11pt]{amsart}
\usepackage[centering,text={15.5cm,22cm},
		marginparwidth=20mm]{geometry}

\usepackage{amsmath}
\usepackage{amsfonts}
\usepackage{amssymb}
\usepackage{graphicx}
\usepackage{mathrsfs}
\usepackage{pb-diagram}
\usepackage{amscd}
\usepackage{hyperref}
\usepackage{url}
\usepackage{caption}
\usepackage{subcaption}
\usepackage{color}
\usepackage[all]{xy}
\usepackage{tikz, multicol}
\usepackage{tikz-cd}
\usepackage{float}
\usepackage{comment}
\usepackage[abbrev,non-sorted-cites]{amsrefs}

\newtheorem{theorem}{Theorem}[section]
\newtheorem{lemma}[theorem]{Lemma}
\newtheorem{corollary}[theorem]{Corollary}

\newtheorem{prop}[theorem]{Proposition}
\newtheorem{defn}[theorem]{Definition}
\newtheorem{example}[theorem]{Example}
\newtheorem{remark}[theorem]{Remark}

\numberwithin{equation}{section}

\newcommand{\cX}{\mathcal{X}}
\newcommand{\seed}{\textbf{s}}

\newcommand{\fd}{\mathfrak{d}}
\newcommand{\fD}{\mathfrak{D}}

\newcommand{\TV}{\mathrm{TV}}

\newcommand{\Spec}{\mathrm{Spec}}

\newcommand{\bA}{\mathbb{A}}
\newcommand{\PP}{\mathbb{P}}

\newcommand{\Q}{\mathbb{Q}}

\newcommand{\fb}{\mathfrak{b}}
\newcommand{\fm}{\mathfrak{m}}

\newcommand{\Z}{\mathbb{Z}}
\newcommand{\R}{\mathbb{R}}
\newcommand{\C}{\mathbb{C}}

\newcommand{\Hom}{\mathrm{Hom}}
\newcommand{\rank}{\mathrm{rank}}

\newcommand{\WT}[1]{\widetilde{#1}}

\newcommand{\cA}{\mathcal{A}}

\newcommand{\AI}{A_\infty}

\DeclareFontFamily{U}{mathx}{\hyphenchar\font45}
\DeclareFontShape{U}{mathx}{m}{n}{
      <5> <6> <7> <8> <9> <10>
      <10.95> <12> <14.4> <17.28> <20.74> <24.88>
      mathx10
      }{}
\DeclareSymbolFont{mathx}{U}{mathx}{m}{n}
\DeclareFontSubstitution{U}{mathx}{m}{n}
\DeclareMathAccent{\widecheck}{0}{mathx}{"71}
\DeclareMathAccent{\wideparen}{0}{mathx}{"75}

\begin{document}

\author[Bardwell-Evans]{Sam Bardwell-Evans}
\address{Department of Mathematics and Statistics\\ Boston University}
\email{sambe@bu.edu}

\author[Cheung]{Man-Wai Mandy Cheung}
\address{Kavli IPMU}
\email{manwai.cheung@ipmu.jp}

\author[Hong]{Hansol Hong}
\address{Department of Mathematics \\ Yonsei University}
\email{hansolhong@yonsei.ac.kr}

\author[Lin]{Yu-Shen Lin}
\address{Department of Mathematics and Statistics\\ Boston University}
\email{yslin@bu.edu}

\begin{abstract}
We construct special Lagrangian fibrations for log Calabi-Yau surfaces, and scattering diagrams from Lagrangian Floer theory of the fibres. Then we prove that the scattering diagrams recover the scattering diagrams of Gross-Pandharipande-Siebert \cite{GPS} and the canonical scattering diagrams of Gross-Hacking-Keel \cite{GHK}. With an additional assumption on the non-negativity of boundary divisors, we compute the disc potentials of the Lagrangian torus fibres via a holomorphic/tropical correspondence. As an application, we provide a version of mirror symmetry for rank two cluster varieties. 
\end{abstract}

\title[Scattering diagrams from holomorphic discs]{
Scattering diagrams from holomorphic discs in log Calabi-Yau surfaces}

\maketitle

\setcounter{tocdepth}{1}
\tableofcontents

\section{Introduction}

A \emph{Looijenga pair} $(Y,D)$ is 
 a projective rational surface $Y$ with an anticanonical cycle $D$. Gross-Hacking-Keel \cite{GHK} provided a pure algebraic method to construct the mirror family for $(Y,D)$ following the SYZ heuristic. 
First, a certain affine manifold $B$ with a singularity was constructed from $(Y,D)$ as an alternative to the base of an SYZ fibration, capturing the expected asymptotic behavior of a putative SYZ fibration near $D$.\footnote{In the case of a rational elliptic surface, this is indeed verified in \cite{CJL2}.} One can intuitively view it as the limiting affine structure obtained by squeezing all the singular SYZ fibres together. Then Gross-Hacking-Keel used the relative Gromov-Witten invariants counting $\mathbb{A}^1$-curves, rational curves in $Y$ intersecting $D$ exactly at a single point, to serve as substitutes for quantum correction from holomorphic discs of Maslov index zero in $X:=Y\setminus D$ with boundaries on the SYZ fibres.
 The enumeration of the $\mathbb{A}^1$-curves decodes the so-called \emph{canonical scattering diagram} on $B$. The canonical scattering diagram intuitively tells how the local charts glue together to form a subset of the mirror family, which can be viewed as the mirror family for the non-compact Calabi-Yau surface $X$. In practice, the canonical scattering diagram is a limit of scattering diagram constructed by Gross-Siebert-Pandharipande \cite{GPS} with certain modification. 
 
The symplectic counterpart of the canonical scattering diagram has not been fully understood  previously. The major difficulty lies in analyzing the moduli spaces of holomorphic discs, which can be highly obstructed, especially when their Maslov indices are zero, due to possible contributions of multiple covers. In \cite{L8}, the fourth author established a general scheme for constructing a scattering diagram out of an SYZ fibration in dimension $2$ using \emph{Fukaya's trick} \cite{F1}, which will be summarized in Theorem \ref{LF scattering diagram} (see also \cite{Y5} for a comprehensive exposition). The scattering diagram sits in the base of the SYZ fibration, and the rays (or walls) consist of torus fibres on which boundaries of Maslov zero discs lie.
Throughout, it will be denoted by $\mathfrak{D}^{LF}$ where the superscript ``\emph{LF}" highlights the role of Lagrangian Floer theory (of SYZ fibres) in its construction. 

Following the SYZ spirit, we construct a special Lagrangian fibration on $X:=Y \setminus D$ (see  Lemma \ref{Lagrangian fibration}), but this does not immediately reproduce the canonical scattering diagram in \cite{GHK}. The main obstacle here is that the moduli space of Maslov index zero stable discs may lose its compactness after removing $D$. A priori, it is possible that the limiting configurations involve negative Chern number spheres lying in $D$. Note that $X$ itself is neither geometrically bounded nor complete.
Understanding such contributions is already a challenging task that usually requires intricate virtual techniques.
To remedy the situation, we make use of the fact that $(Y,D)$ admits a toric model $(\bar{Y},\bar{D})$, i.e. that there exists a non-toric blowup $Y\rightarrow \bar{Y}$ up to a birational modification of the boundary divisor $D$. 
Let us call SYZ fibres in $Y \setminus D$ pulled-back from toric fibres (and sitting away from small neighborhoods of singular fibres) \emph{admissible SYZ fibres}. Their Floer theory is somewhat easier to deal with, as Maslov index zero discs with boundaries on an admissible fibre correspond injectively to holomorphic discs with boundary on a toric fibre in $\bar{Y}$. The latter are essentially determined by their intersection profiles with the toric divisor $\bar{D}$, which are prescribed by the singular fibres that the Maslov index zero discs pass through. Then we argue that the limiting configuration of such discs with a given intersection profile cannot have a sphere bubble, as otherwise we would have a disc with components of higher genus. See Section \ref{sec: compactness} for more details.

Having demonstrated compactness, we can now implement the argument in \cite{L8} to produce a scattering diagram on the portion of the SYZ base parametrizing admissible fibres. Our first task is to match our scattering diagram with the scattering diagram constructed in Gross-Pandharipande-Siebert \cite{GPS}.
The source of \emph{initial} rays are singular fibres associated with each point in the blowup center for $Y \to \bar{Y}$. This can be seen from the local model of a non-toric blowup in \cite{A5}, which we incorporate into our construction of the SYZ fibration on $Y \setminus D$. One can easily see that these initial rays coincide with the initial data of the scattering diagram of \cite{GPS} (or \cite{GHK}), but
another subtlety arises at this point. While ideally one would expect that the initial rays described above are the only rays coming out from a neighborhood of a singular fibre, as is the case in \cite{GPS}, this is not automatically true in our geometric situation, 
since we do not have control over the holomorphic discs with boundaries on non-admissible SYZ fibres. In other words, there might be a Maslov index zero disc with boundary on an admissible fibre that crosses a neighborhood of a singular fibre, in which case one would find an additional ray seemingly sourced from this region.
The resolution of this subtlety comes from the ideas of tropical geometry. 
Having in mind the one-to-one correspondence between discs with boundary on admissible fibres in $Y$ and those with boundary on toric fibres in $\bar{Y}$, we assign a tropical disc to each holomorphic disc in $\bar{Y}$ that has precisely the same intersection profile with $\bar{D}$. The existence of such a tropical disc, in turn, provides a strong obstruction to the existence of holomorphic discs with boundaries on admissible SYZ fibres. With the help of a tropical analogue of Gromov compactness (see Section \ref{sec: LF/GPS}), we recover the Gross-Siebert-Pandharipande scattering diagram from Lagrangian Floer theory. Following the same procedure as in \cite{GHK}, from the Gross-Siebert-Pandharipande scattering diagrams to the canonical scattering diagram, we reach the following theorem:
\begin{theorem}[see Theorem \ref{correspondence: scattering diagram} and discussion in Section \ref{sec: can}]\label{thm:intromain}
One can recover the Gross-Hacking-Keel canonical scattering diagram from the Lagrangian Floer theory of the admissible SYZ fibres in $Y\setminus D$.	
\end{theorem}

The last step of the proof involves the construction of the canonical scattering diagram from the Gross-Pandharipande-Siebert diagram following \cite[Section 3.4]{GHK}. As a result, we obtain an algorithm entirely based on Lagrangian Floer theory that derives the canonical scattering diagram of Gross-Hacking-Keel. An immediate corollary of the above theorem is to establish the equivalence between open Gromov-Witten invariants and log Gromov-Witten invariants via tropical geometry.
\begin{corollary}[Corollary \ref{open-closed}]
For a Looijenga pair $(Y,D)$ and the SYZ fibration on $Y\setminus D$ in Theorem \ref{thm:intromain}, the sum of open Gromov-Witten invariants of a fixed relative class for admissible SYZ fibres is the log Gromov-Witten invariant for the corresponding curve class.
\end{corollary}
 Notice that the former invariants are defined via symplectic geometry, whereas the latter is purely algebro-geometric. The novelty of our approach of establishing a tropical/holomorphic correspondence is that we can bypass the direct comparison of virtual fundamental cycles for moduli spaces in symplectic geometry and algebraic geometry, which would cause a huge amount of technicalities. It is worth mentioning that the later invariants in Corollary 1.2 are equivalent to many other types of enumerative invariants from the work of Bousseau-Brini-Van Garrel \cite{BBV}, all defined via algebraic geometry.\\

The mirror of $(Y,D)$ is a Landau-Ginzburg superpotential $W$, which is a holomorphic function on some complex manifold\footnote{In symplectic setting, it is a function on the rigid analytic variety over $\Lambda$.}, expected to capture the geometry of $(Y,D)$. For instance, if $Y$ is a toric Fano manifold, then $W$ is simply the generating functions of Maslov index two discs with boundaries on the moment fibres \cite{CO}.  
Furthermore, \cite{G9} introduced the notion of broken lines, which are expected to be the tropicalization of Maslov index two discs with boundaries on SYZ fibres. Later, it is further developed by Carl-Pumperla-Siebert \cite{CPS}. Enumeration of broken lines results in the theta functions, which are well-defined global functions on the mirror family of $(Y,D)$, due to their compatibility with the wall-crossing behavior with respect to the canonical scattering diagram. In particular, the spectrum of the algebra generated by the theta functions provide the mirror family of $(Y,D)$. 
On the mirror side, Pascaleff \cite{P5} constructed a non-canonical correspondence between theta functions and certain generators of $SH^0(X)$ assuming $X=Y\setminus D$ is an affine surface (or an exact 4-dimensional symplectic manifold). Later, Ganatra-Pomerleano \cite{GP} generalized the result to higher dimensional cases. 

We will prove that $W$ can be calculated tropically, again by establishing the tropical/holomorphic correspondence for the Maslov index two discs. Namely, the tropical counterpart of the Maslov index two discs is shown to be exactly the broken lines in \cite{GHK}, which confirms the general expectation mentioned at the beginning. For instance, wall-crossing in the holomorphic setting is caused by the bubbling of a Maslov index two disc into a union of Maslov index two and zero discs, and bending of a broken line precisely reflects this phenomenon. Hence, our method provides an explicit (algorithmic) count of Maslov index two discs, which would be very difficult if directly looking at holomorphic discs themselves. The following is a consequence of Theorem \ref{thm: main} and Lemma \ref{blow down}:
\begin{theorem}\label{thm:winintro}
	Let $(Y,D)$ be a Looijenga pair such that 
	$D$ contains no negative Chern number spheres, and let $L$ be an admissible SYZ fibre. Then the number of Maslov index two discs with boundaries on $L$ in class $\beta\in H_2(Y,\mathbb{Z})$ equals to the weighted count of broken lines with stop at $u$ in class $\beta$ and suitable unbounded directions.
\end{theorem} 	
Theorem \ref{thm:winintro} implies that the theta functions associated with $D_i$ in \cite{GHK}, an irreducible component of $D$, agree with the generating functions of Maslov index two discs that intersect $D_i$. 
We remark that the computation of the mirror of toric semi-Fano surfaces will be significantly used as a stepping stone in the proof of Theorem \ref{thm:winintro}. It has been already done by \cite{CL}, but we will examine in Section \ref{subsec:spbintrop} a tropical interpretation of their calculation which is of its own interest. It is worthwhile to mention that for the symplectic form chosen here, $Y$ is not a monotone symplectic manifold nor do we impose the monotone assumption on the SYZ fibres. Therefore, the wall-crossing results of Seidel \cite{S11} or Pascaleff-Tonkonog \cite{PT} do not directly apply here. As an application, we compute the LG mirror of the degree $5$ del Pezzo surface, and verify closed-string mirror symmetry (Theorem \ref{thm:maindp5}).

Additionally, we explore a few more examples that do not exactly fit into the setting of Theorem \ref{thm:winintro}, including a non-Fano surface. More specifically, we find a tropical (and hence holomorphic, through our correspondence) argument that can cover classic examples of Auroux \cite{A, A5} and cubic surfaces \cite{GHKS}.
Mirror symmetry can further extend to non-Fano manifolds with similar pictures, although the superpotential is much harder to compute in the non-Fano cases due to the possible bubbling phenomena. Nevertheless, we expect that the method here generalizes beyond the semi-Fano case by understanding the degeneration of Looijenga pairs.

Last but not least, Theorem \ref{thm:intromain} has a nice application to cluster varieties of rank 2. 
There are two types of cluster varieties, namely the $\cA$ and the $\cX$ cases. 
Fock-Goncharov \cite{FGX} conjectured that the $\cA$ cluster varieties are mirror to the Langlands dual of the $\cX$ varieties, and vice versa. 
Roughly speaking, taking Langlands dual is taking negative transpose of the exchange matrix of the original cluster structure. We will discuss more about Lauglands dual in Section \ref{subsubsec:Langlandsdual}.
On the $\cA$ side, there is a natural torus action, denoted by $T_{K^{\circ}}$, on the $\cA$ cluster variety while there is a fibration of the $\cX$ to a torus $T_{K^*}$ on the $\cX$ side. We refer the readers to Section \ref{sec:clusterdata} on the definition for cluster varieties, $K^{\circ}$, and $K^*$.
By understanding the toric models of the rank 2 cases, we employ the scattering diagram construction in Gross-Pandharipande-Siebert \cite{GPS}.
Then we identify the Gross-Pandharipande-Siebert scattering diagrams with the fibers of $\cX$ cluster scattering diagrams. 
Hence we obtain the following theorem, and we refer the readers to Theorem \ref{thm:cluster} for the precise statement. 
\begin{theorem} 
  Consider the rank 2 cluster varieties. 
  The quotient of the $\cA$ cluster varieties by $T_{K^{\circ}}$ are mirror to the the Langlands dual ${}^L\cX_e$ cluster varieties, where ${}^L\cX_e$ is the fibre of the Langlands dual $\cX$ family at $e \in T_{K^*}$. 
\end{theorem}
 
The organization of the paper is as follows. After reviewing preliminary materials in Section \ref{sec:prelim1}, we revisit Floer theory of Lagrangian fibres in toric surfaces in Section \ref{sec:holodisctoric}, where we mostly focus on \emph{tropicalization} of holomorphic discs in a sense that will be clarified therein. Section \ref{sec:constsyznon-toric} is devoted to constructing an SYZ fibration on a non-toric blowup of a toric surface and analyzing the moduli of Maslov index zero discs, which leads to the scattering diagram $\mathfrak{D}^{LF}$. We then prove our main theorem  in Section \ref{sec:mainholotrop}, stating that $\mathfrak{D}^{LF}$ coincides with $\mathfrak{D}^{GPS}$. Section \ref{sec:apptocluster} provides a cluster duality interpretation of the mirror symmetry studied in this paper. Finally, as an application, we compute the LG mirror and show the closed string mirror symmetry for the del Pezzo surface of degree 5 in Section \ref{sec:msdp5cal}.

\subsection*{Acknowledgments}
We thank Mark Gross, Yoosik Kim, Siu-Cheong Lau, Travis Mandel, Grigory Mikahlkin, Johannes Rau, Tony Yue Yu for their valuable comments.
The second named author is partially supported by NSF grant DMS-1854512.
The third named author is supported by the National Research Foundation of Korea (NRF) grant
funded by the Korea government (MSIT) (No. 2020R1C1C1A01008261). The last named author is partly supported by a Simons Collaboration Grant for Mathematician.  

\subsection*{Notation}
\begin{enumerate}
\item[$\bullet$] Let $N\cong \mathbb{Z}^2$ be a lattice and $M:=\mbox{Hom}(N,\mathbb{Z})$ be the dual lattice. Let $N_{\mathbb{R}}:=N\otimes \mathbb{R}$ and $M_{\mathbb{R}}:=M\otimes \mathbb{R}$. Let $\Sigma\subseteq  M_{\mathbb{R}}$ be a toric fan and $\bar{Y}:=Y_{\Sigma}$ be the associated toric surface. We denote the toric boundary divisor by $\bar{D}$ and its irreducible component by $\bar{D}_i$. We will write $m_i$ for the primitive generator of the $1$-cone corresponding to $\bar{D}_i$ in $\Sigma$. 
\item[$\bullet$] Let $p \colon \bar{Y} \rightarrow P\subseteq N_{\mathbb{R}}$ be the moment map fibration, and let $P$ denote the moment polytope. We will write $\mbox{Log} \colon \bar{Y}\backslash \bar{D} \cong (\C^\ast)^2 \rightarrow \R^2, \ (X,Y) \mapsto (\log |X| , \log |Y|)$. Since fibres of $p$ are fibres of $\mbox{Log}$ and vice versa, we have $p\circ \mbox{Log}^{-1}:\mathbb{R}^2\rightarrow \mbox{Int}P$ is a diffeomorphism. 
For each $u\in \mbox{Int}P$, we denote $L_u$ the moment fibre over $u$.
\item[$\bullet$] Fixing a Lagrangian $L$, we denote by $\mathcal{M}_{k+1, l}(L, \beta, J)$ the moduli space of stable bordered $J$-holomorphic discs representing the class $\beta$ in $H_2(X, L)$ with $k+1$ boundary marked points and $l$ interior marked points with respect to the \textit{standard toric complex structure} $J$ on $X$ in \cite{Gu}.
\item[$\bullet$] $\Lambda_0$ denotes the Novikov ring over the real numbers $\R$,
\begin{equation}
\Lambda_0:=\left\{\sum_{i=1}^{\infty} a_i T^{\lambda_i} \mathbin{|}  \lambda_i \geq 0, \displaystyle\lim_{i \rightarrow \infty} \lambda_i= \infty \text{ and }a_i \in \C\right\}. \nonumber
\end{equation}
There is a non-Archimedean valuation $\mathrm{val} \colon \Lambda_0 \rightarrow \R,$  
\begin{equation}
\mathrm{val}\left(\sum_{i} a_i T^{\lambda_i}\right)= \inf\{ \lambda_i \mathbin{|} a_i \neq 0 \} \text{ and } \mathrm{val}(0)=\infty. \nonumber
\end{equation}
The maximal ideal of $\Lambda_0$ is denoted by $\Lambda_+:=\mathrm{val}^{-1}((0, \infty)).$ 
\end{enumerate}

\section{Preliminaries}\label{sec:prelim1}
\subsection{Tropical Geometry}\label{subsec:tropgeomintro}
The modified SYZ conjecture \cite{KS4} suggests that a Calabi-Yau manifold collapses to an integral affine manifold (possibly with singularities) toward the large complex structure limit. It has been folklore that holomorphic curves in Calabi-Yau manifolds should converge to $1$-skeletons in the integral affine manifolds satisfying a certain balancing condition. These 1-skeletons are referred to as tropical curves. 
The pioneering work of Mikhalkin \cite{M2} shows that the enumeration of holomorphic curves in toric surfaces equals the weighted count of tropical curves.
Since tropical curves are purely combinatorial objects, establishing a correspondence between tropical and holomorphic curves provides a powerful tool for enumerative geometry.

In this section, we recall some knowledge of tropical geometry that we will use later. Let us first give the definition of tropical discs. 
\begin{defn} 
	A parametrized tropical disc with end at $u$ in $M_{\mathbb{R}}$ is a triple $(h,\mathcal{T},w)$ satisfying the following properties: 
	\begin{enumerate}
		\item[(i)] $\mathcal{T}$ is a rooted tree that possibly contains unbounded edges. If $x$ is the root, then $\mathcal{T}$ has only trivalent vertices besides $x$. The set of vertices is denoted by $\mathcal{T}^{[0]}$, and the set of edges is denoted by $\mathcal{T}^{[1]}$. 
		\item[(ii)] $h:\mathcal{T}\rightarrow M_{\mathbb{R}}$ is a map such that $h(x)=u$, and $h(e)$ is an embedding of an affine line segment (resp. an affine ray) if $e \in \mathcal{T}^{[1]}$ is bounded (resp. unbounded).  
		\item[(iii)] The map $w:\mathcal{T}^{[1]}\rightarrow \mathbb{Z}_{> 0}$ assigns a weight to each edge with the following balancing condition. For any vertex besides $x$, the three adjacent edges $e_1,e_2$ and $e_3$ satisfy 
		\begin{align*}
		w(e_1)v(e_1)+w(e_2)v(e_2)+w(e_3)v(e_3)=0,
		\end{align*}
		where $v(e_i)$ is the primitive vector tangent to $h(e_i)$ that is pointed away from $v$.
	\end{enumerate}
	Furthermore, the triple $(h,\mathcal{T},w)$ is called a parametrized tropical disc of $\bar{Y}$ if every unbounded edge is of the form $m'+\mathbb{R}_{\geq 0}m$ where $m'\in {M}_{\mathbb{R}}$ and $m\in M$ is the primitive generator of an $1$-cone in $\Sigma$. A tropical curve/disc is the image of a parametrized tropical curve/disc. 
\end{defn}
\begin{remark}
	Notice that the balancing conditions of tropical curves/discs follow directly from that of the parametrized tropical curves/discs. 
\end{remark}

Motivated by \cite[Lemma 3.1]{A}, the Maslov index of a tropical disc is defined as follows. 
\begin{defn}\label{def:tropMI}
	Given a tropical disc $(h,\mathcal{T},w)$, its Maslov index $\mbox{MI}(h)$ is defined to be the twice the sum of weights of unbounded edges.
\end{defn}
We next recall the Mikhalkin weight \cite{M2} for the trivalent tropical discs. 

\begin{defn}\label{def:mikmult}
	Given a tropical disc $(h,\mathcal{T},w)$ with end at $u\in M_{\mathbb{R}}$. Let $v\in \mathcal{T}^{[0]}$ be a trivalent vertex with adjacent edges $e_1,e_2,e_3$. Then the (Mikhalkin) weight at $v$ denoted as $\mbox{Mult}_v$ is given by 
	\begin{align*}
	\mbox{Mult}_v:=    w(e_1)v(e_1)\wedge w(e_2)v(e_2)\in \wedge^2 M\cong \mathbb{Z}.
	\end{align*} 
	Then the weight $\mbox{Mult}(h)$ of the tropical disc $(h,\mathcal{T},w)$ is defined as 
	\begin{align*}
	\mbox{Mult}(h)= \prod_{v\in \mathcal{T}^{[0]}, v\neq x}\mbox{Mult}_v.
	\end{align*} 
\end{defn}

We will mainly consider tropical discs appearing naturally on the toric setting (or its slight variants). Recall that there exists a natural map
   \begin{align*}
   	 \mbox{Log}: (\mathbb{C}^*)^2&\rightarrow \mathbb{R}^2\cong M_{\mathbb{R}}\\
   	     (z_1,z_2)&\mapsto (\log{|z_1|},\log{|z_2|}).   	 
   \end{align*}
Tropical curves are drawn on $M_{\mathbb{R}}$ and their unbounded edges are required to be parallel to 1-cones of the fan $\Sigma$, each corresponds to a toric divisor. One can intuitively think of each unbounded edge intersecting the corresponding facet of moment polytope at infinity after taking Legendre transform. More generally, any affine line segment in the base of a SYZ fibration with a rational slope can be completed to a cylinder in the total space with help of a complex structure (to be canonically chosen in our geometric setting below), whose symplectic area obviously makes sense. Every trivalent vertex of the tropical disc topological corresponds to a pair of pants within the fibre over the vertex. 

On the other hand,  the intersection patterns of a holomorphic disc with toric divisors completely determine its topological type, and in particular its symplectic area and Maslov index. Thus it is natural to define the relative class of a given tropical disc to be that of the holomorphic disc which intersects the toric divisors in accordance with  unbounded edges of the tropical disc. Throughout, we will frequently use the terms such as relative classes, symplectic areas, Maslov indices, etc., for tropical discs in this spirit. 
 
We will sometimes need to impose further constraints to tropical discs. Namely, we fix a set of generic points in $\R^2$, and consider the counting invariant concerning tropical discs that pass through these points. Obviously, the constraints make discs more rigid, i.e., they behave like discs with lower Maslov indices. For this reason, we define \emph{generalized Maslov index} of a constrained tropical disc by
\begin{equation}\label{eqn:genmutrop}
MI' (h,\mathcal{T},w) :=  2|\mbox{unbounded edges of} \,\, \mathcal{T}| - 2 \, | \mbox{point-constraints}| 
\end{equation}
where unbounded edges are counted with their multiplicities.
This is of course consistent with the generalized Maslov index of a holomorphic disc defined in \cite{HLZ} (see Section \ref{subsubsec:bulkdeformintro}, also). Note that if there is no constraint, then $MI'$ agrees with the Maslov index of the tropical disc above.

\subsection{Scattering diagrams}\label{subsec:scattering}

Fix $R$ an Artin local $\C$-algebra or a complete local $\C$ algebra and let $\mathfrak{m}_R$ be the unique maximal ideal of $R$. For our purpose, we will take $R$ to be the Novikov ring, $\mathbb{C}[[t]]$ or $\mathbb{C}[[NE(Y)]]$\footnote{This is the completion of $\mathbb{C}[NE(Y)]$ with respect to some filtration.}, where $NE(Y)$ is the effective curve cone of a surface $Y$.
We now define a dimension 2 scattering diagram. 
\begin{defn}
	A scattering diagram $\fD$ is a collection of walls $\{(\mathfrak{d},f_{\mathfrak{d}})\}$ where
		\begin{itemize}
		\item $\fd \subseteq M_{\R}$ is either a ray of the form $\fd = m_{\mathfrak{d}}' + \R_{\geq 0} m_{\mathfrak{d}}$ or a line of the form $\fd = m_{\mathfrak{d}}' + \R m_{\mathfrak{d}}$, for some $m_{\mathfrak{d}}' \in M_{\R}$ and $m_{\mathfrak{d}} \in M \setminus \{ 0\}$. 
		The set $\fd$ is called the support of the line or ray. 
		\item $f_{\fd} \in \C [[z^{m_{\mathfrak{d}}}]] \hat{\otimes}_{\C}R \subseteq \C[M]\hat{\otimes}_{\C}R$, called wall functions, satisfy $f_{\fd} \equiv 1 \mod z^{m_{\mathfrak{d} }} \mathfrak{m}_R$.
	\end{itemize} such that
	for every power $k >0$, there are only a finite number of $(\fd, f_{\fd})$ with $f_{\fd} \not\equiv 1 \mod \mathfrak{m}_R^k$. We will also called a restriction of a scattering diagram to a convex subset of $M_{\mathbb{R}}$ a scattering diagram. 
\end{defn}

If $\fD$ is a scattering diagram, we write
\[
\mathrm{Sing} (\fD) = \bigcup_{\fd \in \fD, \fd: \mbox{ray}} \{m_{\mathfrak{d}}' \} \cup \bigcup_{\fd_1, \fd_2 \dim \fd_1 \cap \fd_2 =0} \fd_1 \cap \fd_2.
\]
Given a scattering diagram $\mathfrak{D}$, one can associate an asymptotic scattering diagram $\mathfrak{D}_{as}$ by replacing each ray $(m'_{\mathfrak{d}}+\mathbb{R}_{\geq 0},f_{\mathfrak{d}})$ (and each line $(m'_{\mathfrak{d}}+\mathbb{R}m_{\mathfrak{d}},f_{\mathfrak{d}})$) with the ray $(\mathbb{R}_{\geq 0},f_{\mathfrak{d}})$ (with the line $\mathbb{R}m_{\mathfrak{d}},f_{\mathfrak{d}})$ line respectively). 

Now consider a smooth immersion $\phi: [0,1] \rightarrow M_{\R} \setminus \mathrm{Sing} (\fD)$ such that endpoints avoid the support of the scattering diagram $\fD$ and all intersection of $\phi$ with the walls are transverse.
We then define the path-ordered product as follows:
For each power $k >0$, $\phi$ will cross only a finite number $s_k$ of walls with $f_{\fd} \not\equiv 1 \mod \mathfrak{m}_R^k$. 
We label them by $\fd_i$, where $i =1, \dots s_k$ with respect to the order of the path intersecting the walls.
Each wall $\fd$ determines an automorphism 
\begin{align}\label{formula: wc}
	\theta_{\phi, \fd_i} (z^m) = z^m f_{\fd_i}^{\langle n_0, m \rangle}, 
\end{align}
where $n_0 \in N$ is primitive normal to $\fd_i$, and $\langle n_0, \phi'(t_i) \rangle <0$, where $t_i$ is the moment $\gamma$ crosses the wall $\fd_i$.
Then we define $\theta^k_{\phi, \fD} = \theta_{\fd_{s_k}} \circ \cdots \circ \theta_{\fd_1}$. Then we define $\theta_{\phi, \fD} = \lim_{k \rightarrow \infty} \theta^k_{\phi, \fD}$.

We recall the following theorem, we refer the readers to \cite[Theorem 1.4]{GPS} for its proof.
\begin{theorem} \cites{GS1, KS1} \label{consistency}
	Let $\fD'$ be a scattering diagram. Then there exists a scattering diagram $\fD$ containing $\fD'$ such that $\fD \setminus \fD'$ consists only of rays, and such that $\theta_{\phi, \fD} = \mathrm{Id}$ for any closed loop $\gamma$ for which $\theta_{\phi, \fD}$ is defined. We will call such a scattering diagram consistent. 
	After combining $(\mathfrak{d},f_{\fd}),(\mathfrak{d}',f_{\fd'})$ into $(\mathfrak{d},f_{\fd}f_{ \fd'})$ if $\mathfrak{d}=\mathfrak{d}'$, the resulting $\fD$ is unique. 
\end{theorem}
Such resulting diagram $\mathfrak{D}$ is called consistent. It is easy to see if a scattering diagram is consistent, then its associate asymptotic scattering diagram is also consistent. Also, a restriction of a consistent scattering diagram to a convex subset is again consistent (with respect to loops within the convex subset). 

\begin{example}  Figure \ref{ex:firstA2} illustrated the first example of consistent scattering diagram. It corresponds to the cluster algebra of type $A_2$. From the geometry perspective, the diagrams is given by the del Pezzo of degree 5 with a cycle of five $(-1)$-curves which is indicated in \cite[Example 8.31]{GHKK}. 
	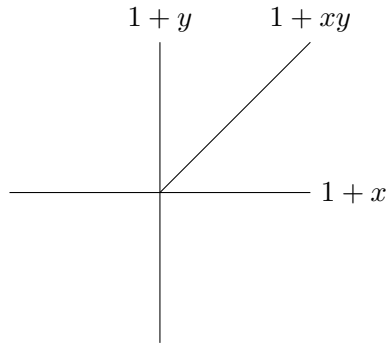
\begin{figure}[H]
 		\centering
 		\begin{tikzpicture}
 		\draw (-2,0) -- (2,0) node[right] {$1+x$};
 		\draw (0,-2) -- (0,2) node[above] {$1+y$};
 		\draw (0,0) -- (2,2) node[above] {$1+xy$};
 		\end{tikzpicture}
 		\caption{$A_2$ scattering diagram} \label{ex:firstA2}
 	\end{figure}
\end{example}

\subsubsection*{Broken lines} We recall the notion of the broken line, which will be important in studying the mirror of a Looijenga pair $(Y,D)$ in Section \ref{subsec:LGmirrorY}. Let $\bar{Y}$ be a projective toric surface. Recall that there is a natural short exact sequence from toric geometry
  \begin{align}\label{toric exact}
    0\rightarrow K_{\Sigma}  \rightarrow T_{\Sigma}:=\bigoplus_{i=1}^n \mathbb{Z}t_i \xrightarrow{r} M\rightarrow 0,
  \end{align} where $r(t_i)=m_i$ and $K_{\Sigma}=\mbox{Ker}(r)$. 

 \begin{defn}\label{broken line} Let $\mathfrak{D}$ be a consistent scattering diagram on $M_{\mathbb{R}}$.
	A broken line with stop at $u\in M_{\mathbb{R}}$ is a continuous map 
	\begin{align*}
	\mathfrak{b}:(-\infty,0]\rightarrow M_{\mathbb{R}}
	\end{align*} such that $\mathfrak{b}(0)=u$ together with the monomial data $c_iz^{m^{\mathfrak{b}(t)}_i}$ described below and satisfy the following properties:
	there exist
	\begin{align*}
	-\infty=t_0<t_1<\cdots <t_n=0
	\end{align*} such that $\mathfrak{b}|_{(t_{i-1},t_i)}$ is affine. For each $i=1,\cdots, n$, there is a monomial $c_iz^{m^{\mathfrak{b}(t)}_i}\in T_{\Sigma}$ such that 
	\begin{enumerate}
		\item For each $i$, $r (m^{\mathfrak{b}}_i)=-\beta'(t)$.
		\item $m^{\mathfrak{b}}_1=m_j$ for some $j\in \{1,\cdots,n\}$ and $c_1=1$. 
		\item $\mathfrak{b}(t_i)\in \mbox{Supp}(\mathfrak{D})\backslash \mbox{Sing}(\mathfrak{D})$. 
		\item If $\mathfrak{b}(t_i)\in \cap_i \mathfrak{d}_i$ (hence in a $1$-dimensional intersection of walls), then the monomial $c_{i+1}z^{m^{\mathfrak{b}(t_{i+1})}_{i+1}}$ is a term in 
		\begin{align}\label{broken line}
		\big(\prod_i\theta_{\mathfrak{b},\mathfrak{d}_i}\big)(c_iz^{m^{\mathfrak{b}(t_i)}_i}).
		\end{align}
	\end{enumerate} 
\end{defn}  We also generalize the notion of broken lines to allow limits of the broken lines in Definition \ref{broken line}

Notice that the dual of $K_{\Sigma}$ is canonically isomorphic to $\mbox{Pic}(\bar{Y})$ from \cite[Section 3.4]{F6}. Since $\bar{Y}$ is rational and $H^{1,0}(\bar{Y})\cong 0$, we have $K_{\Sigma}\cong H_2(\bar{Y},\mathbb{Z})$ canonically. Also, by the classification result of Cho-Oh \cite{CO}, there is a family of Maslov index two discs with Lagrangian boundary condition of class $\beta_i$ for each toric boundary $\bar{D}_i$. Mapping $t_i\in T_{\Sigma}$ to $\beta_i\in H_2(\bar{X},L_u)$ defines an isomorphism. Finally, $L_u\cong M\otimes S^1$ implies $H_2(L_u,\mathbb{Z})\cong M$. Thus we see that the short exact sequence \eqref{toric exact} is naturally identified with 
\begin{align*}
    0\rightarrow H_2(\bar{Y},\mathbb{Z})\rightarrow H_2(\bar{Y},L_u)\rightarrow H_1(L_u,\mathbb{Z})\rightarrow 0,
\end{align*} for any $u\in \mbox{Int}(P)\cong M_{\mathbb{R}}$ since the moment fibration has no monodromy. Indeed, 
\begin{defn}
	Given a broken line $\mathfrak{b}$, we say that $n$ is the length of the broken line $\mathfrak{b}$ and it has homology $[\mathfrak{b}]:=m^{\mathfrak{b}}_n(u)\in H_2(\bar{Y},L_u)$ using the above identification of short exact sequence and weight $\mbox{Mono}(\mathfrak{b}):=c_n$. 
\end{defn}
The following is the simplest example of the broken line. 
\begin{example}\label{initial broken line}
	For any $u\in M_{\mathbb{R}}$ and $m_i$, there exists a broken line $\mathfrak{b}:(-\infty,0]\rightarrow B_0$ with image of the form $u+\mathbb{R}_{\leq 0}m_i$. In this case, $n=1$ with $z^{\beta_i}$, where $\beta_i$ is the unique disc class of Maslov index two in a tubular neighborhood of $\bar{D}_i$. The weight of this broken line is $1$. 
\end{example}

\subsection{The scattering diagrams of Gross-Pandharipande-Siebert}\label{subsec:GPS}

The most relevant example of a scattering diagram to the purpose of this paper is the one constructed in \cite{GPS} (a combination of Theorem 2.8, Theorem 3.4, Theorem 4.4 and Proposition 5.2 therein), which we will denote by $\mathfrak{D}^{GPS}$. We give a detailed review on it, as one of our goals is to retrieve $\mathfrak{D}^{GPS}$ from holomorphic disc counting in the realm of Lagrangian Floer theory.

\subsubsection{Scattering Diagrams from Simple Blowups}
Let $\pi:\widetilde{Y}\rightarrow \bar{Y}$ be the blowup of $\bar{Y}$ at distinct generic\footnote{This is only to avoid the situation a ray from scattering falls in the support of an initial ray.} non-toric points $q_{ij}, j=1,\cdots,l_i$, for some $l_i\in \mathbb{Z}_{\geq 0}$ sitting in an irreducible component $\bar{D}_i$ of the toric boundary divisor $\bar{D}$. 
	Denote by $E_{ij}$ the exceptional divisor corresponding to $q_{ij}$, and by $\widetilde{D}$ the proper transform of $\bar{D}$, which is anticanonical in $\widetilde{Y}$. Then there is a natural consistent scattering diagram $\mathfrak{D}^{GPS}$ associated to the above geometry constructed as follows.
	
	Let $\mathfrak{D}_{in}^{GPS}$ be the scattering diagram consisting of lines $\{(\mathfrak{d}_{ij},f_{\mathfrak{d}_{ij}})\}$ where 
	$\mathfrak{d}_{ij}=m'_{ij}+\mathbb{R}m_{i}$ and $m'_{ij}$ is chosen such that the closure of the preimage of $p\circ \mbox{Log}^{-1}(\mathfrak{d}_{ij})$ under the moment map contains $p(q_{ij})$. 
For simplicity, we will often use the expression \emph{the wall determined by $q_{ij}$} to indicate such a wall $\mathfrak{d}_{ij}$.

	 The associated wall function is  given by $f_{\mathfrak{d}_{ij}}=1+t_{ij}z^{m_{i}}$, where $t_{ij}$ is the formal variable in $R=\mathbb{C}[[t_{ij}]]$. From Theorem \ref{consistency}, $\mathfrak{D}_{in}^{GPS}$ uniquely determines a consistent scattering diagram $\mathfrak{D}^{GPS}$. Below we will explain the geometric interpretation of its wall functions. 
	 If $(\mathfrak{d},f_{\mathfrak{d}})\in \mathfrak{D}^{GPS}\setminus \mathfrak{D}_{in}^{GPS}$, then $\mathfrak{d}$ must be a ray, i.e.  $\mathfrak{d}=m'_{\mathfrak{d}}+\mathbb{R}_{\geq 0}m_{\mathfrak{d}}$, and
	 the wall function $f_{\mathfrak{d}}$ is given as the generating series of certain algebraic curves counts. If $-m_{\mathfrak{d}}$ is not part of the toric fan of $\bar{Y}$, then 
	 let $\bar{Y}_{\mathfrak{d}}$ be the weighted toric blowup obtained by adding a new ray in the direction $-m_{\mathfrak{d}}$ to the fan of $\bar{Y}$ and denote the corresponding toric divisor $D_{\mathfrak{d}}$. Otherwise, let $\bar{Y}_{\mathfrak{d}}=\bar{Y}$ and $D_{\mathfrak{d}}$ be the toric boundary corresponds to $-m_{\mathfrak{d}}$. Assume that there exists non-negative integers $\mathbf{P}=(p_{ij})$ such that  
	 \begin{align}\label{?}
	 	\sum_{i,j}p_{ij}m_i=k_{\mathbf{P}}m_{\mathfrak{d}},
	 \end{align} for some $k_{\mathbf{P}}\in \mathbb{N}$.
	  The exceptional divisor in $\bar{Y}_{\mathfrak{d}}$ appears as a toric boundary component that we denote by $\bar{D}_{\mathfrak{d}}$. (The toric blowup given in Definition \ref{toric model} slightly generalizes this.) Let $\widetilde{Y}_{\mathfrak{d}}$ be the fibre product
	   \begin{align*}
	  	     \xymatrix{
	  	      \widetilde{Y}_{\mathfrak{d}} \ar[r]\ar[d]_{\pi_{\mathfrak{d}}}  &\widetilde{Y} \ar[d]^{\pi} \\
	  	      \bar{Y}_{\mathfrak{d}} \ar[r] & \bar{Y} }
	  	  \end{align*} with induced map $\pi_{\mathfrak{d}}:\tilde{Y}_{\mathfrak{d}}\rightarrow \bar{Y}_{\mathfrak{d}}$
		  and $\widetilde{D}_{\mathfrak{d}}$ is the proper transform of the toric boundary $\bar{D}_{\mathfrak{d}}$ of $\bar{Y}_{\mathfrak{d}}$. 
			  Then from \cite[Theorem 3.31]{GHK} (or see \cite[Section 5.7]{GPS}) the wall function is given as \begin{align} \label{eq:gps}
			 \log{ f_{\mathfrak{d}}}=\sum_{\mathbf{P}}k_{\mathbf{P}}N_{\beta_{\mathbf{P}}}t^{\mathbf{P}}z^{k_{\mathbf{P}}m_{\mathfrak{d}}},\end{align}
			   where $t^{\mathbf{P}}=\prod_{i,j}t^{p_{ij}}_{ij}$ and $\mathbf{P}$ runs over all possible $\mathbf{P}$ that satisfies \eqref{?}. The coefficient $N_{\beta_{\mathbf{P}}}$ computes the associated log Gromov-Witten invariant in $\tilde{Y}_{\mathfrak{d}}$ \cite{GS3}. 
	Intuitively, $N_{\beta_{\mathbf{P}}}$ counts the number of rational curves $C$  in $\widetilde{Y}$ such that 
    \begin{enumerate}
    	\item $\pi(C)$ is tangent to $\bar{D}_i$ at $q_{ij}$ with tangency multiplicity $p_{ij}$. 
    	\item $C$ intersects $\widetilde{D}$ at exactly one point in $\tilde{D}_{\mathfrak{d}}$.
    	\item $\sum_{i,j} p_{ij} m_{i}=k_{\mathbf{P}}m_{\mathfrak{d}}$ for some $k_{\beta}\in \mathbb{Z}_{> 0}$. 
    \end{enumerate} Such curve $C$ are known as an $\mathbb{A}^1$-curve, as its intersection with the Looijenga interior $\tilde{Y}\setminus \tilde{D}$ is an $\mathbb{A}^1$. 
   Notice that (1)(2)(3) uniquely determine the homology class of $\beta_{\mathbf{P}}\in H_2(\tilde{Y}_{\mathfrak{d}},\mathbb{Z})$. Indeed, there exists a unique homology class $\beta\in H_2(\bar{Y}_{\mathfrak{d}},\mathbb{Z})$ such that if $D_{\mathfrak{d}}$ if not one of the component ${D}_i$ of $\bar{D}_{\mathfrak{d}}$, then $\beta\cdot D_i=\sum_j p_{ij}$ and $\beta\cdot D_{\mathfrak{d}}=k_{\mathbf{P}}$. Otherwise $D_{\mathfrak{d}}=D_{i_0}$ for some $i_0$ and $\beta\cdot D_i=\sum_{j}p_{ij}$ if $i\neq i_0$ and $\beta\cdot D_{i_0}=\sum_{j}p_{i_0 j}+k_{\mathbf{P}}$. 
   	Then $\beta_{\mathbf{P}}$ is determined by  $\beta_{\mathbf{P}}=\pi_{\mathfrak{d}}^*\beta-\sum_{i,j}p_{ij}[E_{ij}]$. 
  We will refer readers to \cite[Section 4]{GPS} (or \cite[Definition 3.1]{GHK}) for the precise definition of $N_{\beta_{\mathbf{P}}}$. 
  
  \subsubsection{Scattering Diagrams from Orbifold Blowups}\label{subsubsec:orbiblowupGPS}
  Gross-Pandharipande-Siebert further considered the orbifold blowup $\pi:\widetilde{Y}\rightarrow Y$ at non-toric point $q_{ij}$ with multiplicity $r_{ij}$ and studied the enumerative meaning of the wall-functions in the associate scattering diagram. We will discuss the heuristic picture in Lagrangian Floer theory in Section \ref{sympl orbifold} and use this to explain the mirror symmetry of cluster varieties of rank two in Section \ref{sec:apptocluster}.
  
  In this case, $\widetilde{Y}$ has an $A_{r_{ij}-1}$-singularity for each point $q_{ij}$ and admits a unique structure as a non-singular Deligne-Mumford stack. The exceptional divisor $E_{ij}$ over $q_{ij}$ contains the orbifold point due to higher multiplicity given at the blowup point $q_{ij}$. Denote by $\mathfrak{D}^{GPS}_{in}$ the scattering diagram consisting of $\{(\mathfrak{d}_{ij},f_{\mathfrak{d}_{ij}})\}$, where $\mathfrak{d}_{ij}$ is the wall determined by $q_{ij}$ and $f_{\mathfrak{d}_{ij}}=1+t_{ij}z^{r_{ij}m_i}$. We will refer the readers to Section \ref{sympl orbifold} for the heuristic symplectic analogue. Again we write $\mathfrak{D}^{GPS}$ for the unique consistent scattering diagram constructed from $\mathfrak{D}^{GPS}_{in}$. For each ray $(\mathfrak{d},f_{\mathfrak{d}})\in \mathfrak{D}^{GPS}\setminus \mathfrak{D}^{GPS}_{in}$, one can similarly define $\bar{Y}_{\mathfrak{d}}$ and $\widetilde{Y}_{\mathfrak{d}}=\bar{Y}_{\mathfrak{d}}\times_{\bar{Y}}\widetilde{Y}$. Then the wall function $f_{\mathfrak{d}}$ is given by similar formula as \eqref{eq:gps} with $N_{\beta}$ the orbifold Gromov-Witten invariant intuitively counting the number of rational curves $C$ in $\widetilde{Y}$ such that 
  \begin{enumerate}
  	\item $\pi(C)$ passes through $q_{ij}$ exactly once with tangency multiplicity $p_{ij}$. 
  	\item $C$ intersects $\widetilde{D}$ at exactly one point.
  	\item $\sum_{i,j} p_{ij} m_{i}=k_{\beta}m_{\mathfrak{d}}$ for some $k_{\beta}\in \mathbb{N}$. 
  \end{enumerate}
  
 \subsection{Canonical Scattering Diagram of Gross-Hacking-Keel}\label{sec: can}
 We begin with the same geometric setup in \cite{GHK} and then summarize the construction of the canonical scattering diagram.
 
 \begin{defn}
 	A Looijenga pair $(Y,D)$ is a smooth projective rational surface $Y$ with $D\in |-K_{Y}|$ a reduced rational curve which has at least one singular point. 
 \end{defn}	
 
 If $D$ is irreducible, then it is a genus one curve with a single nodal point. Otherwise, it is a cycle of smooth rational curves.
 The following is a useful observation from the classification of surfaces. 
 \begin{prop}\cite[Proposition 1.3]{GHK} \label{toric model}
 	For any Looijenga pair $(Y,D)$, there exist two other Looijenga pairs $(\widetilde{Y},\widetilde{D})$ and $(\bar{Y},\bar{D})$ such that 
 	\begin{enumerate}
 		\item $\pi':\widetilde{Y}\rightarrow Y$ is a blowup of nodal point(s) of $D$, called the toric blowup,
 		and $\widetilde{D}=\pi'^*D$;
 		\item $(\bar{Y},\bar{D})$ is a toric pair and, $\pi:\widetilde{Y}\rightarrow \bar{Y}$ is a blowup at smooth points of $\bar{D}$, called a non-toric blowup, with $\widetilde{D}$ being the proper transform of $\bar{D}$. 
 	\end{enumerate}
 \end{prop}	
 We write $X:=Y\setminus D$, and similarly $\tilde{X}:=\widetilde{Y} \setminus \widetilde{D}$ and $\bar{X}:=\bar{Y} \setminus \bar{D}$. There exists a meromorphic form $\Omega$ on $Y$ restricting to a holomorphic volume form on $X$, unique up to $\mathbb{C}^*$-scaling. We  denote the corresponding holomorphic volume forms by $\tilde{\Omega}$ and $\bar{\Omega}$, respectively. A straightforward calculation shows that $\tilde{\Omega}=\pi'^*\Omega=\pi^*\bar{\Omega}$, where we abuse notations by putting $\pi':\tilde{X} \to X$ and $\pi:\tilde{X}\rightarrow \bar{X}$. 
 
  Given a Looijenga pair $(Y,D)$, Gross-Hacking-Keel \cite{GHK} associated an integral affine manifold $B_{GHK}$ with a singularity. Topologically, it is $\mathbb{R}^2$ with a unique singularity at the origin. There is a cone decomposition $B_{GHK}=\cup_{i=1}^n \sigma_i/\sim$, where $\sigma_i$ is a cone bounded by two rays $\mathbb{R}_{\geq 0}v_i$ and $\mathbb{R}_{\geq 0}v_{i+1}$, and each ray $\mathbb{R}_{\geq 0}v_i$ corresponds to a component $D_i$ of the boundary divisor $D$. Here the relation $\sim$ glues $\mathbb{R}_{\geq 0}v_1$ with $\mathbb{R}_{\geq 0}v_{n+1}$.   The canonical scattering diagram $\mathfrak{D}^{can}$ is then a collection $\{(\mathfrak{d},f_{\mathfrak{d}})\}$ satisfying the following.
   If $\mathfrak{d}$ is a ray generated by $av_i + bv_{i+1}$ for $a,b\in \mathbb{Z}_{\geq 0}$, then 
   	  \begin{align*}
   	     \log{f_{\mathfrak{d}}}=\sum_{k\geq 1} kc_kX_i^{-ak}X_{i+1}^{-bk}\in R[[X_i^{-a}X_{i+1}^{-b}]],
   	  \end{align*} where $R=\mathbb{C}[NE(Y)]$ and $NE(Y)$ is the monoid generated by effective curve classes in $Y$. The coefficient $c_k$ is the generating function of the log Gromov-Witten invariants  
   	  \begin{align*}
   	  c_k=\sum_{\beta} N_{\beta}z^{\beta},
   	  \end{align*} where the summation is over all possible classes $\beta\in H_2(Y,\mathbb{Z})$ with incidence relation $\beta.D_i=ak,\beta.D_{i+1}=bk$ and $\beta.D_j=0$, for $j\neq i,i+1$. The coefficient $N_{\beta}$ is the algebraic count of $\mathbb{A}^1$-curves in class $\beta$ as in Section \ref{subsec:GPS}. 
 
      The integral affine manifold with singularity $B_{GHK}$ and the canonical scattering diagram $\mathfrak{D}^{can}$ are independent under deformation of the Looijenga pair $(Y,D)$ as stated in \cite[Lemma 3.9]{GHK}, thanks to the deformation invariance of log Gromov-Witten invariants. If $(\widetilde{Y},\widetilde{D})\rightarrow (Y,D)$ is a toric blowup, then their corresponding integral affine manifolds with singularities are naturally isomorphic, and so are their associated canonical scattering diagrams by \cite[Lemma 1.6]{GHK}. 
      
      If $(Y,D)\rightarrow (\bar{Y},\bar{D})$ is a non-toric blowup, from the deformation invariance of the canonical scattering diagram, we may assume that $Y$ is the simple blowup of $\bar{Y}$ at mutually distinct points $q_{ij} \in \bar{D}_i$. Recall that these data determines the scattering diagram $\mathfrak{D}^{GPS}$. 
      The relation between the canonical scattering diagram $\mathfrak{D}^{can}$ of $(Y,D)$ and the scattering diagram $\mathfrak{D}^{GPS}$ is as follows: 
      we first take the asymptotic scattering diagram $(\mathfrak{D}^{GPS})_{as}$ of $\mathfrak{D}^{GPS}$.

      For any $av_i+bv_{i+1}$, $(a,b)\in \mathbb{N}^2\setminus \{(1,0),(0,1)\}$ and $(a,b)$ is primitive, we have a ray $\mathfrak{d}=\mathbb{R}_{\geq 0}(av_i+bv_{i+1})$ in $\mathfrak{D}^{can}$ with the wall function 
         \begin{align*}
            f_{\mathfrak{d}}=\prod_{(\mathfrak{d}',f_{\mathfrak{d}'})\in (\mathfrak{D}^{GPS})_{as}} f_{\mathfrak{d}'},
         \end{align*} here $(\mathfrak{d}',f_{\mathfrak{d}'})$ runs through all possible rays in $(\mathfrak{D}^{GPS})_{as}$ with corresponding curve class $\beta\in H_2(\widetilde{Y}_{\mathfrak{d}'},\mathbb{Z})$ such that $(\pi_{\mathfrak{d}'})_*\beta\cdot D_i=a,(\pi_{\mathfrak{d}'})_*\beta\cdot D_{i+1}=b$. 
     If $(a,b)=(1,0)$, then wall function $f_{\mathfrak{d}}$ attached to the ray $\mathfrak{d}=\mathbb{R}_{\geq 0}v_i$ is given by 
       \begin{align*}
          f_{\mathfrak{d}}=\prod_{j=1}^{j=l_i}(1+t^{[E_{ij}]}z^{-m_i}) \prod_{(\mathfrak{d}',f_{\mathfrak{d}'})}f_{\mathfrak{d}'}.
       \end{align*} Here $(\mathfrak{d}',f_{\mathfrak{d}'})$ runs through all possible rays in $(\mathfrak{D}^{GPS})_{as}$ with the corresponding curve class $\beta\in H_2(\widetilde{Y}_{\mathfrak{d}'},\mathbb{Z})$ such that $(\pi_{\mathfrak{d}'})_*\beta\cdot D_i>0, (\pi_{\mathfrak{d}'})_*\beta\cdot D_j=0$ for $j\neq i$ but $\beta$ is not a multiple of $E_{ij}$ for any $j=1\cdots, l_i$. To sum up, given any Looijenga pair $(Y,D)$, one can construct the associated canonical scattering diagram $\mathfrak{D}^{can}$ of $(Y,D)$ from the asymptotic scattering diagram $(\mathfrak{D}^{GPS})_{as}$ of $(\widetilde{Y},\widetilde{D})\rightarrow (\bar{Y},\bar{D})$\footnote{Here we identify $m_i$ with $-\phi_i(v_i)$ in \cite{GHK}.}. 
      We refer readers to \cite[Section 3.4]{GHK} for further information. 
      \begin{remark}
      	 It is expected that the orbifold analogue of the work of Gross-Hacking-Keel mirror construction also hold. However, the authors are not aware of that in the literature.
      \end{remark}

\subsection{Lagrangian Floer theory}\label{subsec:LFT1}
We briefly review Lagrangian Floer theory of a compact Lagrangian and its bulk deformation following \cite{FOOO}. We begin by fixing a generic almost complex structure $J$. For a compact Lagrangian $L$ in a symplectic manifold $Y$, one can assign a possibly curved $A_\infty$-algebra structure on $H^\ast (L;\Lambda_0)$ as follows.
If we write $\mathcal{M}_{k+1} (Y,L,\beta;J)$ for the moduli space of $J$-holomorphic discs in class $\beta \in \pi_2 (Y,L)$ with $k+1$ boundary marked points $z_0,\cdots ,z_k$ then one has
\begin{equation}\label{eqn:AIoperations}
m_{k,\beta} (h_1,\cdots, h_k):= (ev_0)_! (ev_1 ,\cdots,ev_k)^\ast (h_1 \times \cdots \times h_k) ,\quad m_k = \sum_{\beta \in \pi_2 (X,L)} m_{k,\beta} \, T^{\omega(\beta)}
\end{equation}
for $k \geq 0$ and $h_i \in H^\ast(L;\Lambda)$ where $h_1 \times \cdots \times h_k$ is the wedge product of pull-backs of $h_i$ under the projection map and $ev_i$ is the evaluation at $z_i$.
We additionally put $m_{1,0}(h) = (-1)^{ n+ \deg x +1} d_{\mathrm{dR}} h$. The degree of $m_{k,\beta}$ is given by $2-k - \mu(\beta)$ 
where $\mu (\beta)$ is the Maslov index of the class $\beta \in \pi_2 (X,L)$, which directly follows from 

\begin{equation}\label{eqn:dimmkp1}
 \dim \mathcal{M}_{k+1} (Y,L,\beta;J) = \mbox{dim}_{\mathbb{R}}L + \mu (\beta) + k - 2.
 \end{equation}
The family $\{m_k\}_{k \geq 0}$ of multi-linear operations satisfies the quadratic relations
$$ \sum_{k_1 + k_2 = k+1 } (-1)^\star m_{k_2} (h_1,\cdots,  m_{k_1} ( h_i ,\cdots, h_{i+k_1 -1}), \cdots, h_k) =0$$
called the $A_\infty$-relations, where 
$\star = |h_1|' + \cdots + |h_{i-1}|'$. When $m_1^2=0$, then the resulting homology is referred to as the Lagrangian Floer cohomology of $L$.

The formula \eqref{eqn:AIoperations} includes the case where the term $m_0 (1):=m_0 \in H^\ast (L;\Lambda_+)$ is nontrivial, which gives an obstruction for $m_1$ to  become a differential. In order to derive any meaningful homological invariants, one should deform the $A_\infty$-structure 
by a so-called weak bounding cochain $b \in 
H^1 (L;\Lambda_+)$, which satisfies
\begin{equation}\label{eqn:weakMCeqn}
m_0 (1) + m_1(b) + m_2(b,b) +\cdots =  W(b) \cdot 1_L
\end{equation}
where $W(b)$ is a Novikov constant depending on $b$, known as the superpotential of $L$. If \eqref{eqn:weakMCeqn} has a solution, then $L$ is said to be weakly unobstructed.
One can easily see from \eqref{eqn:dimmkp1} that
$W(b)$ is contributed by the Maslov index $2$ discs. 

When $L$ is a Lagrangian in $X=Y \setminus D$ for a log Calabi-Yau pair $(Y,D=\cup_i D_i)$ with $D_i$ irreducible and $c_1(D_i)$ positive, those discs are precisely the ones that hit $D$ exactly once. If some of $D_i$ is not positive, then one should additionally take into account the higher Maslov discs attached with negative or zero Chern number sphere bubbles.
Typical examples are toric manifolds with $L$ being a toric fibre, see for e.g., \cite{FOOO6}. 
It is conventional that one uses the exponential coordinate to write $W(b)$ in such a case. That is, we set
\begin{equation}\label{eqn:expcoordch}
 z_i:= \exp x_i  \qquad i=1,\cdots n
\end{equation}
where we take $ b =\sum  x_i d \theta_i$ where $\{d \theta_i : i=1,\cdots n\}$ is a basis of $H^1 (L)$. Throughout, we will use the notation
  \begin{align*}
     z^{\partial \gamma}:=z_1^{( \partial \gamma,d\theta_1 )} \cdots z_n^{ ( \partial \gamma,d\theta_n ) },
  \end{align*}
 for $\partial \gamma \in H_1 (L,\mathbb{Z})$, where $(\,\, , \,\,) $ is the natural pairing between homology and cohomology.\footnote{We use $\partial \gamma$ to denote a general element in $H_1 (L,\mathbb{Z})$, as  we will frequently look at classes in $H_1(L,\mathbb{Z})$ appearing as the boundary class of a holomorphic disc.} It is worthwhile to mention that $z^{\partial \gamma}$ is independent of the choice of the basis $\{d\theta_i\}$ and the notation is thus intrinsic.

The exponential coordinate also helps us to include holonomy variables for flat $\C^\ast$-connections on $L$ more naturally, which enables us to expand the set of solutions $b$ of \eqref{eqn:weakMCeqn} over $\Lambda_0$.
We remark that the coordinate change \eqref{eqn:expcoordch} makes sense, provided the divisor axiom \cite[Lemma 13.1]{F1} (or see \cite[Lemma 11.8]{FOOO7}, also):
\begin{equation}\label{eqn:divisor1}
  \sum_{ \sum n_i =n} m_{k+n,\beta}(b^{\otimes n_0}, h_1, b^{\otimes n_2}, \cdots, b^{\otimes n_{k-1}}, h_k, b^{\otimes n_k})=\frac{1}{n!} \big(  \partial \beta, b  \big)^n m_k (h_1, \cdots, h_k)
\end{equation}
for $\beta \in \pi_2 (Y,L)$, which tells us that $W(b)$ is actually a function in $e^{x_i}$.

Let us denote by $n_\beta (L)=n_\beta^{Y} (L)$ the number of the holomorphic discs in class $\beta$ with $\mu (\beta)=2$ whose boundary passes through a generic point in $L$, or in other words, $n_\beta (L)$ is the degree of the map $ev_0 : \mathcal{M}_1 (Y,L;\beta) \to L$ provided the weak unobstructedness \eqref{eqn:weakMCeqn} (and all necessary transversality).  It is easy to see that the coefficient of $z^{\partial \beta}$ in $W$ is given by $n_\beta (L) T^{\omega(\beta)}$. 
Indeed, the coefficient is given by $\langle m_{0,\beta} (1), [pt] \rangle $ by the divisor axiom since $m_{0,\beta} (1)$  is a multiple of the unit (weak unobstructedness), and this is precisely the number of discs in class $\beta$ passing thru a generic point representing $[pt]$.

 Alternatively, one can define $W$ as a function on the `moduli space' of the pair $(L,\nabla)$ where $L$ varies over torus fibres and $\nabla$ is a flat $U(1)$-connection on the trivial line bundle $L \times \C$ by 
\begin{equation}\label{eqn:Wglobaltdual}
W (L,\nabla) = \sum_{\beta, \mu (\beta)=2} n_\beta (L) T^{\omega(\beta)} hol_{\partial \beta} \nabla,
\end{equation}
which fits more into the SYZ setup (see \cite{A} for more details). The superpotential presented here can be thought of as a local restriction of \eqref{eqn:Wglobaltdual}.

  \subsubsection{Bulk-Deformed $A_{\infty}$ Structures and Bulk-Deformed Superpotentials}\label{subsubsec:bulkdeformintro}
  
  One can further deform the $A_\infty$-structure on $H^\ast (L;\Lambda)$ along the direction of $H^\ast (X)$ by inserting interior marked points to holomorphic discs and requiring them to pass through (the Poincare duals of) given ambient cocycles in $H^\ast(X)$. It is called the bulk-deformation of the Floer theory of $L$.

Suppose an ambient cocycle $\mathfrak{b} \in H^{even} (X;\Lambda_+)$ is given.
Consider for each $\beta \in \pi_2 (X,L)$, we denote $\mathcal{M}_{k+1,l}(L, \beta, J)$ the moduli space of stable $J$-holomorphic discs with $k+1$ boundary marked points and $l$ interior marked points representing the class $\beta$. Its dimension is given by $n + \mu (\beta) +k-2 + 2l$ where the additional $2l$ compared with \eqref{eqn:dimmkp1} comes from the freedom of locations of interior markings.
We then impose the incidence condition to the interior markings by taking fibre product to obtain
\begin{equation} \label{eqn:fibre_prod}
\mathcal{M}_{k+1,l}(L,\beta, J; \mathfrak{b}^{\otimes l}):=\mathcal{M}_{k+1,l}(L, \beta, J)_{ ev^{int}} \times_{X^l} \prod_{i=1}^l \widetilde{PD[\mathfrak{b}]}.  
\end{equation}
where $\widetilde{PD[\mathfrak{b}]}$ is a cycle representing $PD[\mathfrak{b}]$ and $ev^{int}$ is the evaluation map associated with interior marked points.
Finally, the bulk-deformed $A_\infty$-operations are defined by
$$\fm_{\beta,k}^{\fb}(h_1, \cdots ,h_k)=\sum_{l} \frac{1}{l !}(ev_0)_!(ev_1, \cdots, ev_k)^*(h_1\times \cdots \times h_k),\quad \fm_{k}^{\fb} = \sum_\beta \fm_{\beta,k}^{\fb} T^{\omega(\beta)} $$
Notice that $\fm_{\beta,k}^{\fb}$ is contributed by $\mathcal{M}_{k+1,l}(L,\beta, J; \mathfrak{b})$ with $l$ varying over $\Z_{\geq 0}$.\footnote{Strictly speaking, \eqref{eqn:fibre_prod} makes sense when $\fb$ is a single geometric cycle, not a linear combination of such. In the latter case, however, we can simply defined $\fm_{\beta,k}^{\fb}$ by expanding it linearly in accordance with the linear combination forming $\fb$.}

The superpotential $W(b)$ also deforms accordingly now by solving the analogous equation for $m_k^{\mathfrak{b}}$:
$$ m_0^{\fb} (1) + m_1^{\fb} (b) + m_2^{\fb} (b,b) +\cdots =  W^{\fb} (b) \cdot 1_L$$
and the resulting $W^{\mathfrak{b}}$ is called the bulk-deformed potential. In \cite{HLZ}, the second and third authors considered a special type of a bulk-deformation of the Floer theory of toric fibres in a Fano surface, for which $\mathfrak{b}$ is taken to be a linear combination $\sum t_i q_i$ of generic points with $t_i^2=0$. 
One needs to impose $l$-many point-constraints to discs of Maslov index $\mu=2l+2$ in order for them to contribute to $W^{\mathfrak{b}}$. The dimension formula of the disc moduli (see Section \ref{subsubsec:bulkdeformintro}) tells us that these constrained discs behave like Maslov 2 discs, and more generally, the \emph{generalized Maslov index} $\mu'$ of a holomorphic disc was introduced in \cite{HLZ} as twice the number of point-constraints subtracted from the ordinary Maslov index $\mu$ (i.e., $\mu'=\mu - 2l$). Notice that \eqref{eqn:genmutrop} is precisely its tropical analogue. $W^{\mathfrak{b}}$ shows certain discontinuity when $L$ varies over toric fibres due to the existence of \emph{generalized} Maslov zero discs. Such a phenomenon is generally called a wall-crossing, which we recall below in more general context.

  \subsubsection{Pseudo-isotopies and wall-crossing}\label{subsubsec:lfwallopengw}

Consider two Lagrangian submanifolds $L$ and $L'$ in $(X,\omega)$ which are related by a smooth isotopy $\phi_t$. Namely, $\phi_t$ is a diffeomorphism on $X$ such that $\phi_0 = id$ and $\phi_1 (L) = L'$. We further assume that for each $t$, the almost complex structure $J_t = ( \phi_t )^\ast J$ is tame with respect to $\omega$.
We can relate the superpotentials for $L$ and $L'$ in the following way.

Notice that $J_0(=J)$-holomorphic discs bounding $L'$ isotope to $J_1$-holomorphic discs bounding $L$ thorough $\phi$. The Floer theory of $L$ does not depend on the choice of almost complex structures in the following sense: the isotopy results in an $\AI$-algebra isomorphism 
\begin{equation*}
\left\{f_k = \sum_\beta f_{k,\beta} T^{\omega(\beta)}\right\}_{k \geq 0} : H^\ast (L; \Lambda)^{\otimes k} \to H^\ast (L';\Lambda),
\end{equation*}
where $f_{k,\beta}$ is obtained by counting discs in
\begin{equation*}\label{eqn:modfuktrick}
\cup_{t \in [0,1]} \mathcal{M}_{k+1} (L,\beta; J_t)
\end{equation*}
whose boundary marked points are subject to suitable incidence conditions. The degree of $f_{k,\beta}$ is given by $1-k-\mu(\beta)$.

The $A_\infty$-homomorphism $\{f_k\}_{k \geq 0}$ induces a map between associated sets of weak bounding cochains
\begin{equation}\label{eqn:fastbbb}
 f_\ast : b \mapsto f_0 (1) + f_1 (b) + f_2 (b,b) + \cdots.
\end{equation}
This construction is now broadly referred to as \emph{Fukaya's trick} since it has first appeared in \cite{F1}. See \cite{Y5} for more recent development. 

\begin{lemma}\cite{F1}\label{wall-crossing}
Let $L$ and $L'$ be two Lagrangians that can be interpolated by an isotopy $\phi$.
\begin{enumerate}
\item
The superpotential $W$ and $W'$ of $L$ and $L'$ satisfy 
$$W'(f_\ast (b) ) = W(b)$$
after rescaling the term by suitable powers of the Novikov parameter.\footnote{We refer readers to \cite[Lemma 4.2]{T4} for details on the rescaling factors. This will not be needed for our purpose, since $L$ and $L'$ in the application are infinitesimally close to each other.}
\item
The map $f_\ast$ only depends on the homotopy type of the isotopy $\phi_t$. Namely, if there is an isotopy $\Phi_{s,t}$ over $[0,1]^2$ between two isotopies $(\phi_1)_t=\Phi_{1,t}$ and $(\phi_2)_t=\Phi_{1,2}$ from $L$ and $L'$, then $\phi_1$ and $\phi_2$ induced the same map $f_\ast$ on the weak Maurer-Cartan space of $L$ and $L'$.
\end{enumerate}
\end{lemma}

The contributions to \eqref{eqn:fastbbb} are from discs with $\mu (\beta)=0$, and by dimension reason, a generic Lagrangian torus fibre does not bound such discs. Collecting  Lagrangians bounding Maslov zero discs forms a structure called the wall in the base $B$ of the torus fibration, which is conjecturally of real codimension $1$ in $B$ for generic almost complex structures. 

Let us now focus on the situation where $X$ admits a special Lagrangian torus fibration $\pi : X \to B$ with respect to a holomorphic volume for $\Omega$, and $L,L'$ are two different fibres of $\pi$. Given a reference point $u_0\in B_0$ and a choice of the basis $\check{e}_1,\check{e}_2\in H_1(L_{u_0},\mathbb{Z})$, we will define the local affine coordinates around $u_0$. For any $u\in B_0$ in a small neighborhood of $u_0$, one choose a path $\phi$ contained in $B_0$ connecting $u,u_0$. Let $C_{k}$ be the $S^1$-fibration over the image of $\phi$ such that the fibres are in the homology class of parallel transport of $\check{e}_k$ along $\phi$. Then the local  complex symplectic affine coordinates are defined by
\begin{align}\label{affine coor}
x_{\check{e}_k}(u)=\int_{C_{k}}\mbox{Im}\Omega.
\end{align} It is straight-forward to check that the transition functions fall in $GL(2,\mathbb{Z})\rtimes\mathbb{R}^2$, and thus the above coordinates give an integral affine structure on $B_0$. We will say that the affine line $x_k=const$ is an affine line defined by $\check{e}_k$.

One can show that the loci of torus fibres bounding Maslov zero discs form an affine line in $B$ with respect to these affine coordinates (see, for e.g., \cite[Proposition 5.6]{L8}).  Such a line is called the wall, and the above discussion tells us that $f_\ast$ is nontrivial only if the isotopy goes across some wall, in which case we call $f_\ast$ a wall-crossing transformation. Strictly speaking, the wall should be defined as the loci of fibres $L_u$ which admit nontrivial open Gromov-Witten invariants $\tilde{\Omega} (\gamma;u)$ for some relative class $\gamma \in H_2(X,L_u)$. See \cite[Definition 4.12]{L4} for the precise definition of the invariants. 
If the underlying isotopy goes across the wall at the point $u$, one can roughly think of $f_\ast$ as the generating series of $\tilde{\Omega} (\gamma;u)$. 
    \begin{theorem}\label{YS} \cite[Theorem 6.15]{L8}
Let $\pi : X \to B$ be a special Lagrangian fibration on a symplectic manifold $M$ with $\dim_\R M=4$. Suppose there exists a wall in $B$ contributed by a disc class $\gamma \in \pi_2 (X,L)$ of Maslov index zero.	Then the wall is an affine line (ray) along the direction corresponding to $\partial \gamma \in H^1(L)$, and the associated wall-crossing transformation is of the form 
\begin{equation}\label{eqn:wctrans}
z^{\partial \gamma'}\mapsto z^{\partial \gamma'} f_{\gamma}^{\langle \gamma',\gamma\rangle}.
\end{equation}
	  where $f_{\gamma}\in 1+\Lambda[[T^{\omega(\gamma)}z^{\partial\gamma}]]$. Moreover, $\log{f_{\gamma}}$ is given by the generating function of the  open Gromov-Witten invariants of $d\gamma, d\geq 1$ and can be derived from the pseudo-isotopies of $A_{\infty}$-structures on the fibres.
   \end{theorem} 
The last author then constructed a scattering diagram based on Lagrangian Floer theory, as was speculated in \cite{F3, KS1}. 
Notice that the wall-crossing transformations in \eqref{eqn:wctrans} preserve $\frac{dz_1}{z_1}\wedge \frac{dz_2}{z_2}$ and are exactly of the form \eqref{formula: wc}.
Therefore, one can form a scattering diagram $\mathfrak{D}^{LF}$ by taking the ray/lines to be the loci parametrizing the fibres bounding holomorphic discs of Maslov index zero, which will be affine with respect to the complex affine structure on the base. The corresponding wall function is taken to be $f_{\gamma}$ in \eqref{eqn:wctrans}. Given a loop in the base, the composition 
of $f_\ast$ \eqref{eqn:fastbbb} for each wall that the loop goes across is simply the path-ordered product of the scattering diagram $\mathfrak{D}^{LF}$. Therefore, $\mathfrak{D}^{LF}$ is consistent from (2) of Lemma \ref{wall-crossing} provided the absence of negative Maslov discs. See \cite{L8} for more details of the proof. 
The precise statement is given as follows. 
  \begin{theorem}\label{LF scattering diagram} \cite[Theorem 4.16]{L8}
  	Let $(X,\omega)$ be a symplectic Calabi-Yau manifold of dimension $4$ and let $J$ be a tamed almost complex structure. Assume that $X'\subseteq X$ with $X'\rightarrow B$ be a Lagrangian fibration such that 
  	\begin{enumerate}
  		\item there are no singular fibres over $B$;
  		\item there exists an affine structure on $B$ such that the loci of fibres bounding holomorphic discs of Maslov index zero in a fixed class (up to parallel transport) are constrained in an affine line; 
  		\item the moduli space $\mathcal{M}(X,L,\gamma;J)$ of stable discs is compact for a fibre $L$ of $X'\rightarrow B$ and $\gamma\in H_2(X,L)$;
  		\item no fibres bound $J$-holomorphic discs of negative Maslov index. 
  	\end{enumerate}
  Then there exists a consistent scattering diagram $\mathfrak{D}^{LF}$ on $B$ constructed from pseudo-isotopies between fibres. 
  \end{theorem} 
In terms of calculation,
 when a new wall is produced by the collision of walls, the new walls and wall functions can thus be captured by (2) of Lemma \ref{wall-crossing} together with the so-called Kontsevich-Soibelman lemma (see e.g. \cite[Theorem 6]{KS1}).

\section{Floer theory on toric surfaces}\label{sec:holodisctoric}

In this section, we prove some useful facts about holomorphic discs whose boundaries lie on the moment fibre $L_u=p^{-1} (u)$, which will be the crucial ingredient for the proof of our main theorem in Section \ref{sec:mainholotrop}. We remark that holomorphic discs bounding $L_u$ are completely classified by the work of Cho-Oh \cite{CO}, and they admit very concrete coordinate-wise descriptions from which we can read off much of topological information of discs. Our task here is to assign a tropical disc to each of these holomorphic discs, whose end lies at $u$, and whose intersection pattern with the boundary divisors precisely match with the original holomorphic disc. 

In the second half, we compute the superpotentials for semi-Fano toric surfaces by means of tropical counting. While they have already been computed in the previous work of Chan-Lau \cite{CL}, the method we use here is completely different, and is of independent interest. These toric potentials serve as a stepping stone to compute potentials for general log Calabi-Yau surfaces thorough blowup and blowdown process (under some non-negativity assumption).

\subsection{Construction of tropical discs from holomorphic discs in toric surfaces}\label{subsec:weakcorresholotrop} 
We will establish a weak correspondence between holomorphic and tropical discs for toric surfaces in this subsection. 
More precisely, to a  
holomorphic disc with a fixed set of point-constraints (on the toric boundary), we assign a tropical disc whose unbounded edges have directions and locations determined by the point-constraints. This will impose an obstruction for existence of certain holomorphic discs to exist, which will be crucial in proving Theorem \ref{correspondence: scattering diagram}.

Let us begin with the following useful fact about holomorphic discs in toric manifolds. Given a moment torus fibre $L$ in a toric manifold $\bar{Y}$, up to $(\mathbb{C}^*)^n$-action, we may assume that $L$ is the fibre over the origin in $\mathbb{R}^n$ which is identified with $\mbox{Int}\, P$ via the Legendre transformation. Then $L$ is the fixed locus of the anti-holomorphic involution $\sigma:(z_1,\cdots, z_n)\mapsto (\frac{1}{\overline{z_1}},\cdots, \frac{1}{\overline{z_n}})$ on $(\mathbb{C}^*)^n$. 

Any holomorphic disc with boundary on $L$ restricted to $(\mathbb{C}^*)^n (\subset \bar{Y})$ can be doubled, and then extend to a rational curve in $\bar{Y}$. This gives the following lemma. 

\begin{lemma}
\label{doubling}
Let $\bar{Y}$ be a toric manifold, and $L$ a moment fibre.
   Given a holomorphic disc $f:(D^2,\partial D^2)\rightarrow (\bar{Y},L)$, there exists a unique rational curve $C_{f}$ in $\bar{Y}$ such that it contains the image of $f$, and that $C_f \cap (\mathbb{C}^*)^n$ is invariant under $\sigma$. 
\end{lemma} 

We now restrict to the case $n=2$. Suppose that a curve $C$ in $(\mathbb{C}^*)^2$ is defined to be the zero loci of a Laurent polynomial $F=F(z_1,z_2)$. 
 Consider the amoeba $\mathcal{A}$ of $C$, the image of $C$ under the log map 
\begin{equation} \label{log map}
\begin{array}{cccl}
\mbox{Log}: & ( \mathbb{C}^*)^2 &\rightarrow& \mathbb{R}^2  \\
&(z_1,z_2) & \mapsto & (\log{|z_1|},\log{|z_2|}).
\end{array}
\end{equation}
We recall the \emph{Ronkin function} $N_{F}(u)$ \cite{PR, R4} defined on $\mathbb{R}^2$. It is given as the push-forward of $\log{|F|}$, i.e. 
\begin{align*}
N_{F}(u)=\frac{1}{(2\pi i)^2}\int_{L_u}\log{|F|}\frac{dz_1}{z_1}\wedge \frac{dz_2}{z_2}
\end{align*} 
for $u \in \mathbb{R}^2$. Notice that
$N_F$ is well-defined even for $u$ lying in $\mathcal{A}$. 
It is also known that $N_{F}: \R^2 \to \R$ is a convex function, strictly convex on $\mathcal{A}$ and is linear (up to translation) on each connected component of the complement of $\mathcal{A}$ \cite{R3}. For $E$ a component of $\mathbb{R}^2 \setminus \mathcal{A}$, we write $N_F^{E}$ to be the unique affine function on $\mathbb{R}^2$ that coincides with $N_F|_E$ on $E$.
Finally we set $N_F^{\infty}:=\mbox{max}_E N_F^E$.

The spine $\mathcal{S}$ of $\mathcal{A}$ (or of $C$) is now defined as the corner locus of $N_F^{\infty}$, i.e. the locus where $N_F^{\infty}$ is not locally linear. On a component $E$ of $\mathbb{R}^2 \setminus \mathcal{A}$, $N_F^E$ can be written as $N_F^E=\nu^E_1u_1+\nu^E_2 u_2+a_E$, and in this case, $\nu_E:=(\nu^E_1,\nu^E_2)\in \mathbb{Z}^2$ is called the index of the component $E$\footnote{From the Jensen's formula \cite{A6}, the index $\nu_E$ can be determined by the intersection number of $C_f$ with certain holomorphic discs.}. Then $\mathcal{S}$ can be also described as the zero loci of the tropical polynomial $F_s=\sum_{\nu_E}s^{a_E}z_1^{\nu^E_1}z_2^{\nu^E_2}$, where $s$ is the formal variable of the Puiseux series. Therefore, $\mathcal{S}$ is a tropical curve by definition.\footnote{If we view $s$ to be real variable, then $N^{\infty}=\lim_{s\rightarrow \infty} N_{F_s}$.}

Now we are ready to prove that every holomorphic disc with boundary on the moment map fibre has a (not necessarily unique) corresponding tropical disc. The idea is to use the spine $\mathcal{S}$ of the doubling $C_f$ of a holomorphic disc, but the tricky part here is that in general, the genus of the spine can be strictly larger than that of the original curve (see Remark 2.3 \cite{MR} for such an example). For this reason, we will make a further modification of $\mathcal{S}$ in the proof.

\begin{lemma}\label{weak correspondence} Let $\bar{Y}$ be a toric surface and  consider a moment fibre $L_u$ over $u\in \mathbb{R}^2$. For each holomorphic disc $f:(D^2,\partial D^2)\rightarrow (\bar{Y},L_u)$, there exists a tropical disc with stop at $u$ that satisfies the following:
\begin{itemize}
\item[$\bullet$]
if the holomorphic disc $f$ intersects a component of the toric boundary divisor at a point of multiplicity $m$, then the corresponding tropical disc has an unbounded edge defined by the vanishing cycle of the toric boundary and of weight $m$.
\item the edge adjacent to $u$ is defined by $[f(\partial D^2)]$.
\end{itemize}  
\end{lemma}
\begin{proof}
Lemma \ref{doubling} states that there exists a unique rational curve $C_f\subseteq \bar{Y}$ containing $\mbox{Im}f$. Let $\mathcal{S}$ be the corresponding spine. 
    The $\mathbb{Z}_2$-symmetry on the rational curve $C_f$
    implies that 
     \begin{enumerate}
     	\item $\mathcal{A}$ is symmetric under $u\mapsto -u$,
     	\item the defining equation $F$ of $C_f$ satisfies $F(z_1^{-1},z_2^{-1})=z_1^az_2^bF(z_1,z_2)$ for some $a,b\in \mathbb{Z}$. 
     \end{enumerate}
    Then straightforward calculation using the definition of $N_F$ shows that 
    \begin{align*}
      N_F(-u)=N_F(u)+au_1+bu_2.
    \end{align*} Notice that adding a fixed linear function to every $N^E_F$ does not change the corner locus of $N^{\infty}_F$.
    Thus, the tropical polynomial of $\mathcal{S}$ is also $\mathbb{Z}_2$-invariant.  As mentioned, we will modify $\mathcal{S}$ into another tropical curve $\mathcal{S}'$ in order to ensure that the resulting tropical curve is of genus $0$.
    
    Let $\Delta'$ be the Newton polytope of the tropical polynomials of $\mathcal{S}$, that is, the convex hull of the set $\{\nu_E\}$, where $E$ varies over all the components of $\mathbb{R}^2\setminus \mathcal{A}$. 
    Then we take $\mathcal{S}'$ to be the tropical curve defined by the tropical polynomial $\sum_{\nu_E\in \partial \Delta'}t^{a_E}z_1^{\nu^E_1}z_2^{\nu^E_2}$. From  \cite[Proposition 3.11]{M2}, the tropical curve $\mathcal{S}'$ is dual to a subdivision of $\Delta'$. Since the coefficient corresponding to any interior lattice point of $\Delta'$ is zero, we see that $\mathcal{S}'$ is a tree from \cite[Proposition 4.19]{M2}(and see the proof therein). On the other hand, the coefficients of $\mathcal{S}$ and $\mathcal{S}'$ over the lattice points on $\partial \Delta'$ still remain the same.    
   Hence, the sets of affine lines obtained by extending the unbounded edges of $\mathcal{S}$ and $\mathcal{S'}$ agree with each other, which are determined by the intersection of $C_f$ with the toric boundary of $\bar{Y}$. Lastly, it is obvious that $\mathcal{S}'$ is $\mathbb{Z}_2$-invariant.
   
        \begin{figure}[h]
   	\begin{center}
   		\includegraphics[scale=0.4]{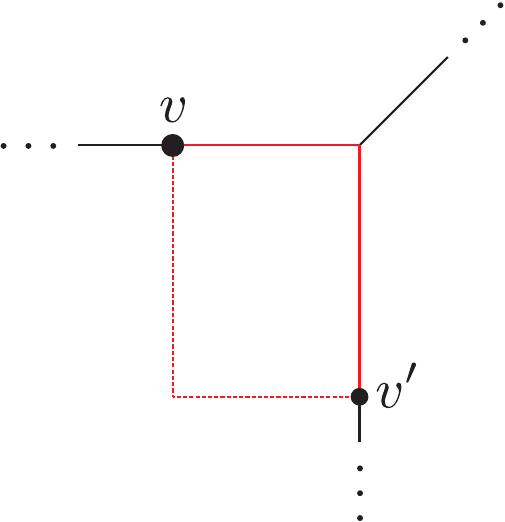}
   		\caption{$\Z_2$-symmetry of a tropical curve}
   		\label{fig:genustree}
   	\end{center}
        \end{figure}
   Notice that the involution $\sigma$ preserves the fibres of $\mbox{Log}$ and thus induces a map $\underline{\sigma}:\mathbb{R}^2\times \mathbb{R}^2$ simply given by $(u_1,u_2)\mapsto (-u_1,-u_2)$. 
    We claim that $u$ is contained in $\mathcal{S}'$. Let $v$ be a vertex in $\mathcal{S}'$, and $v'$ its reflection image, $v=\underline{\sigma}v'$, which also lies in $\mathcal{S}'$.  We may assume $v \neq v'$. 
    There exists a path in $\mathcal{S}'$ connecting $v$ and $v'$. If $u\notin \mathcal{S}'$, then the image of the path under $\underline{\sigma}$ provides another path between $v$ and $v'$, which contradicts the fact that $\mathcal{S}'$ is a tree (see Figure \ref{fig:genustree}). Since $\mathcal{S}'$ is a tropical curve of genus zero which is symmetric about $u$, one can halve it to gain a tropical disc with stop at $u$. It is known that the amoeba $\mathcal{A}$ contains the spine $\mathcal{S}$ \cite[Proposition 1 (iii)]{PR}. Notice that the compactification $\bar{\mathcal{A}}\subseteq P$ contains the points $p(C_f\cap \bar{D})$, where $\bar{D}$ is the boundary divisor. The modification from $\mathcal{S}$ to $\mathcal{S}'$ does not change the unbounded edges (up to enlongation) and thus the first part of the lemma is proven.
    The second part of the lemma is a consequence of the balancing condition at all vertices. 
\end{proof}

   We emphasize again that  the lemma provides a strict obstruction of existence of holomorphic discs.

\subsection{Tropical interpretation of the  potentials for toric semi-Fano surfaces}\label{subsec:spbintrop}

We end this section by examining the relationship between the bulk-deformed superpotential of a compact toric Fano surface and the superpotential of a compact toric semi-Fano surface obtained from the Fano surface by toric blowup. Any compact toric semi-Fano surface can be obtained in this way, with the exception of the Hirzebruch surface $\mathbb{F}_2 = \mathbb{P}(\mathscr{O}_{\mathbb{P}^1}\oplus \mathscr{O}_{\mathbb{P}^1}(-2))$. This process provides a means for calculating the superpotential on any toric semi-Fano surface (except $\mathbb{F}_2$) using a tropical calculation on a toric Fano surface.

Let $Y$ be a compact semi-Fano toric surface with $\widetilde{Y}$ a semi-Fano toric blowup of $Y$ (at a single point). Note that $Y$ and $\widetilde{Y}$ are necessarily smooth. The fans for both $Y$ and $\widetilde{Y}$ are complete. The fan for $\widetilde{Y}$ has one more 1-cone than that for $Y$. Let $\tilde{w}$ denote the primitive vector in this 1-cone of the fan for $\widetilde{Y}$ not appearing in the fan for $Y$. We have that $\widetilde{w} = w_1 + w_2$, where $w_1$ and $w_2$ are primitive vectors in adjacent $1$-cones of the toric fan for $Y$. We choose the labeling $w_1,w_2$ so that the ordered pair $(w_1, w_2)$ is positively oriented, meaning that $w_2$ is counterclockwise from $w_1$. We choose primitive vectors in all of the 1-cones of $Y$, labeling them $w_i$ as we move counterclockwise, i.e. so that $w_i$ and $w_{i+1}$ are adjacent, in the sense that they bound a 2-cone of the fan, and $(w_i,w_{i+1})$ is positively oriented. Note that, by the smoothness of $Y$, we have $\det(w_i,w_{i+1}) = 1$ for all $i$.

Let $q$ be a point in $M_\R$. We associate a scattering diagram $\mathfrak{D}_s$ with the point $q_s = q + s\tilde{w}$ as follows: $\mathfrak{D}_s$ is the scattering diagram with one ray $(\mathfrak{d}_v, f_{\mathfrak{d}}) = (q_s + v\R_{\geq 0}, 1 + z^{-v}t)$ for each primitive vector $v$ in a 1-cone of the fan for $Y$. Here the ring $R$ we are using for the functions associated with the walls of the scattering diagram is the ring $\C[t]/(t^2)$. When we consider scattering diagrams associated with multiple point insertions $q_s, p_1,\hdots, p_n$, we will use the ring $\C[t, t_1,\hdots, t_n]/(t^2, t_1^2, \hdots, t_n^2)$, with one variable for each point insertion. See \cite{HLZ} for further details on the choice of ring. Note that $\mathfrak{D}_s$ is not consistent, though the wall-crossing function $\theta_{\phi,\mathfrak{D}_s} = \text{Id}$ for any closed loop $\phi$ that is contractible in $M_\R\setminus \{q_s\}$.

As $s$ approaches infinity, one chamber of $\mathfrak{D}_s$ will exhaust $M_\R$. This is the open cone spanned by $-w_1$ and $-w_2$ based at $q_s$. We let $U_s$ denote this open cone. Letting $p_1,\hdots, p_n$ be generic points in $U_s$, we will study the generalized Maslov index $2$ tropical discs (with respect to the points $q_s,p_1,\hdots, p_n$) with stop in $U_s$.

We establish some facts about the toric fan of $Y$.

\begin{lemma}\label{SemiFano Fan Sums}
	Let $Y$, $\widetilde{Y}$, $w_1$, $w_2$, and $\widetilde{w}$ be as above. That is, let $Y$ be a compact toric semi-Fano surface with $\widetilde{Y}$ a semi-Fano toric blowup of $Y$ (at a single point). Let $\tilde{w}$ be the primitive vector in the 1-cone of the fan for $\widetilde{Y}$ not appearing in the fan for $Y$, and let $w_1$ and $w_2$ be the primitive vectors in the adjacent $1$-cones with $w_1 + w_2 = \tilde{w}$, labeled so that $(w_1,w_2)$ is positively oriented. Given any primitive vector $w$ lying in a 1-cone of the toric fan of $Y$, we have that $w = bw_1 + cw_2$, where $b,c < 2$.
\end{lemma}

Given the exhaustive description of semi-Fano toric surfaces given by Chan-Lau \cite{CL}, this could be checked directly. We present an elementary inductive proof here.

\begin{proof}
	By the smoothness of $Y$, $w_1$ and $w_2$ form an integral basis of $M$, so we can write $w = bw_1 + cw_2$ for some integers $b,c$. By assumption, the 1-cones corresponding to $w_1,w_2$ are adjacent in the fan of $Y$, in the sense that they bound a 2-cone of the fan, so $b$ and $c$ cannot both be positive. Without loss of generality, we will assume $c > 0$ and show $c = 1$. By the smoothness of $Y$ and our labeling convention, $\det (w_1, w_2) = 1$.
	
	The primitive vector in the next 1-cone after that of $w_2$, proceeding counterclockwise, is $w_3$. We write $w_3 = b_3w_1+c_3w_2$. If $c_3\leq 0$, then all remaining $w_i = b_iw_1 + c_i w_2$ must have $c_i < 0$, and the lemma holds, so we assume that $c_3 > 0$. This implies $b_3 < 0$. Since $\widetilde{Y}$ is semi-Fano, the self intersection number in $\widetilde{Y}$ of the divisor corresponding to the $1$-cone containing $w_2$ is at least $-2$. It is known that this self intersection number is equal to $-\det(w_1+w_2, w_3) = -c_3 + b_3$. We see that $-c_3 + b_3 \leq -2$, with equality only when $b_3 = -1$ and $c_3 = 1$. Hence, $w_3 = -w_1 + w_2$.
	
	We induct, assuming $w_{i} =-(i-2)w_1 + w_2$. The vector $w_{i+1} = b_{i+1}w_1 + c_{i+1} w_2$ is the primitive vector in the next $1$-cone counterclockwise from that of $w_i$. Again, if $c_{i+1} <0$, we have exhausted all possible examples that could violate the lemma, so the lemma holds. Instead, assume that $c_{i+1} > 0$. From the semi-Fano condition on $\widetilde{Y}$, it follows that the self intersection number of the divisor in $\widetilde{Y}$ corresponding to the $1$-cone containing $w_{i}$ is greater than or equal to $-2$, so 
	\begin{align*}
		-\det (w_{i-1},w_{i+1}) &= -\det \begin{pmatrix}{-(i-3)}&{b_{i+1}}\\{1}&{c_{i+1}} \end{pmatrix}\\
		&= b_{i+1} + (i-3) c_{i+1}\\
		&\geq -2
	\end{align*}
	Since $w_{i+1}$ is counterclockwise from $w_{i}$ and $Y$ is smooth, we have that $$1 = \det \begin{pmatrix}{-(i-2)}&{b_{i+1}}\\{1}&{c_{i+1}}\end{pmatrix},$$ so $-1 = b_{i+1} + (i-2)c_{n+1} > b_{i+1} + (i-3)c_{n+1} \geq -2$. Since $c_{i+1}>0$, it then follows that $c_{i+1} = 1$ and $b_{i+1} \geq -(i-1)$. Since $w_i = -(i-2)w_1+w_2$, it follows that $b_{i+1} = -(i-1)$, and our induction is complete.
	
	Returning to our original $w = bv_1 + cv_2$, we see that, since $c > 0$ and $Y$ has finitely many 1-cones, the vector $w$ must eventually appear in the form $w = w_i = -(i-2)w_1 + w_2$. Thus $c = 1$, as desired.
	
\end{proof}

The proof of this lemma immediately gives the following lemma.

\begin{lemma}\label{SemiFano No New Directions}
	Let $Y$, $\widetilde{Y}$, $w_1$, $w_2$, and $\widetilde{w}$ be as in Lemma \ref{SemiFano Fan Sums}, and let $w = bw_1 + cw_2$ be a primitive vector lying in a 1-cone of the toric fan of $Y$. If $c > 0$, then $w_1 + w$ is a primitive vector lying in a 1-cone of the toric fan of $\widetilde{Y}$. Likewise, if $b > 0$, then $w_2 + w$ is a primitive vector lying in a 1-cone of the toric fan of $\widetilde{Y}$.
\end{lemma}

\begin{proof}
	In the notation of the proof of Lemma \ref{SemiFano Fan Sums}, our first statement amounts to observing, for $i\geq 3$, that $w_1 + w_i = w_{i-1}$.
\end{proof}

Let $h: \mathcal{T}\to M_{\mathbb{R}}$ be a generalized Maslov index 2 disc (with respect to the points $q_s, p_1,\hdots, p_n$) with stop $u$ in the chamber $U_s$ of $\mathfrak{D}_s$ described above. Here the \textit{stop} of the disc $h$ is the point in $\R^2$ to which the unique univalent vertex of $\mathcal{T}$ is mapped, and $U_s$ is the open (affine) cone based at $q_s$ determined by $-w_1$ and $-w_2$. We establish the following lemma regarding tropical discs with multiple unbounded edges entirely outside of a closed half plane (which will, in our uses, contain the chamber $U_s$).

\begin{lemma}\label{Tropical Clipping Half Plane}
	Consider generic points $q_s, p_1,\hdots, p_n$, all lying outside an open half plane $H$, and let $h$ be a generalized Maslov index 2 tropical disc with stop at a point $u\notin H$. Then each connected component of $h^{-1}(H)$ can completely contain at most one unbounded edge of $\mathcal{T}$.
\end{lemma}

Intuitively, this says that, to get from one unbounded edge contained in $H$ to any other, our path must leave $H$. 

\begin{proof}
	Since our points $q_s,p_1,\hdots, p_n$ are generic, our moduli space of generalized Maslov index 2 tropical discs with stop at $u$ must be finite. It thus suffices to show that if a connected component of $h^{-1}(H)\subseteq \mathcal{T}$ has multiple unbounded edges with image contained entirely in $H$, then $h$ lies in a positive dimensional family of tropical discs with the same stop $u$ and marked point constraints $q_s, p_1, \hdots, p_n$.
	
	Choose a connected component of $h^{-1}(H)\subseteq \mathcal{T}$. We obtain a new graph $\mathcal{T}'$ by removing every edge and vertex of $\mathcal{T}$ that does not intersect our chosen connected component of $h^{-1}(H)$. The graph $\mathcal{T}'$ is a connected subgraph of $\mathcal{T}$, and is hence a tree with unbounded edges. Let $\mathcal{T}''$ be the subtree of $\mathcal{T}'$ obtained by removing all edges that are unbounded in $\mathcal{T}'$ but bounded in $\mathcal{T}$. That is, $\mathcal{T}''$ is the subgraph of $\mathcal{T}'$ obtained by removing every edge not contained in $h^{-1}(H)$. Again, $\mathcal{T}''$ is a tree.
	
	We define the map $h': \mathcal{T}'\to M_\R$ so that $h'$ agrees with $h$ on $\mathcal{T}''$ and agrees with $h$ on every edge of $\mathcal{T}'$ that was unbounded in $\mathcal{T}$. Finally, we define $h'$ so that it maps the remaining edges, which are bounded in $\mathcal{T}$ but unbounded in $\mathcal{T}'$, to rays in the same direction as the line segments they were sent to by $h$. This new object $h'$ is like a tropical curve with all vertices contained in $H$, except its unbounded edges that are not contained in $H$ may be in non-toric directions, since their directions may have been inherited from bounded edges of $h$. This can be understood as a tropical curve in some toric blowup of $Y$. Let $\Delta$ be the degree of this tropical curve $h'$.
	
	Let $m$ be the number of edges in $\mathcal{T}'\setminus \mathcal{T}''$. For each edge in $\mathcal{T}'\setminus \mathcal{T}''$, we take an associated point $q_i$ contained in the image under $h'$ of that edge but not contained in $H$.
	Let $\mathcal{M}$ be the moduli space of tropical curves of degree $\Delta$ with the image of each edge of $\mathcal{T}'\setminus \mathcal{T}''$ covering its associated point $q_i$. Note that all such tropical curves have the same number of unbounded edges, and this number must be greater than $m$. We let $\ell$ be the positive integer such that $m + \ell$ is this number of unbounded edges. Note that $\ell$ is precisely the number of unbounded edges of $\mathcal{T}$ whose image under $h$ is contained in $H$ By a standard argument, we find that the dimension of $\mathcal{M}$ is $\ell - 1$.
	
	Every curve $h''$ in a small open neighborhood of $h'$ in $\mathcal{M}$ gives rise to a tropical disc in the same class as our original $h$ subject to the same point constraints as $h$. This new disc $h''':\mathcal{T}\to M_\R$ agrees with $h$ on $\mathcal{T}\setminus \mathcal{T}'$ and agrees with $h''$ on $\mathcal{T}''$. On $\mathcal{T}'\setminus \mathcal{T}$, which is the rest of $\mathcal{T}$ and consists only of edges, the maps $h$ and $h''$ send each edge into the same affine line as each other. This is because of the constraints $q_i$. For $h''$ sufficiently close to $h'$, we have that $h''$ maps at least some of each such edge into $H$. It follows that we are able to define $h'''$ on $\mathcal{T}'\setminus \mathcal{T}''$ simply by changing the lengths of these edges as mapped by $h$.
	
	Note for clarity that the image of $h'''$ minus $H$ is exactly the same as the image of $h$ minus $H$, and that the image of $h'''$ intersect $H$ is exactly the same as the image of $h'$ intersect $H$.
	
	This then gives us an $\ell-1$ dimensional family of generalized Maslov index 2 discs relative to the original point constraints. Since $\ell$ is precisely the number of unbounded edges of $\mathcal{T}$ with image under $h$ entirely contained in $H$, we have the desired result.
\end{proof}

This brings us to the following lemma:

\begin{lemma}\label{Tropical Clipping Main}
	Let $Y$, $w_1$, and $w_2$ be as above. Let $q_s, p_1,\hdots, p_n$ be point constraints positioned as above, so that $p_1,\hdots, p_n$ all lie in the exhausting chamber of the scattering diagram $\mathfrak{D}_s$ determined by the point $q_s = q + s(w_1+w_2)$ with $s$ taken large after fixing the $p_i$. Let $h$ be a tropical disc in $Y$ of generalized Maslov index 2 with respect to the point constraints $q_s, p_1,\hdots, p_n$, with stop at a point $u\in U_s$, and assume there is some edge $e$ of $h$ with $h(e)$ containing $q_s$. Then $e$ is unbounded, with its outward primitive tangent vector equal to $w_i$ for $i$ equal to $1$ or $2$, and one of the two edges sharing a vertex with $e$ is also unbounded, with outward pointing primitive tangent vector having positive $w_j$ component ($j = 1,2$, $j\neq i$). The third edge $e'$ adjacent to this vertex must then have weight 1, and the primitive vector $v(e')$ tangent to $h(e')$ and pointing toward the vertex is either the vector $w_1 + w_2$ or a primitive vector belonging to a 1-cone of the toric fan of $Y$.
\end{lemma}

\begin{proof}
	Letting $\mathcal{T}$ be the domain tree of $h$, we assign each edge an ``inward'' and ``outward'' direction, so that any injective path in $\mathcal{T}$ starting at the root moves in the outward direction.
	
	Assume $h$ has some edge $e$ with $q_s$ lying on $h(e)$, and let $v(e)$ be the primitive vector tangent to $h(e)$ pointing in the outward direction. We first show that $v(e)$ is contained in the closed cone generated by $w_1$ and $w_2$.
	
	Assume for the sake of deriving a contradiction that $v(e)$ lies outside this closed cone. Writing $v(e) = bw_1 + cw_2$, we have that either $b < 0$ or $c < 0$, and we will assume without loss of generality that $c < 0$. It follows that $-v(e)$, the vector tangent to $h(e)$ pointing in the inward direction, is pointing toward a vertex lying inside the open half plane $H$ of points $q_s + a_1w_1 + a_2w_2$ with $a_2 > 0$.
	
	Since we have this vertex lying in $H$, we must have at least one unbounded edge lying entirely in $H$. Noting that the complement of $H$ contains all of our marked points, by Lemma \ref{Tropical Clipping Half Plane} we get that any path along the disc from one unbounded edge lying entirely in $H$ to another must leave $H$. Let $e'$ be the unbounded edge contained in $H$ that can be connected to $e$ by a path in the disc lying entirely in $H$. Further, let $e_1,\hdots, e_m$ be all of the edges of $h$ intersecting the connected component of $h^{-1}(H)$ to which the edge $e'$ belongs. Next, let $v(e_i)$ be the primitive vector tangent to $h(e_i)$ pointing into $H$, and let $v(e')$ be the primitive vector in the unbounded $h(e')$ direction. Letting $w(e_i)$ be the weight of the edge $e_i$, the vertex balancing condition gives that $v(e') = \sum_{i=1}^m w(e_i)v(e_i)$. Here we are using that $e'$ will have weight $1$. Since each $v(e_i)$ has positive $w_2$ coordinate, and $m\geq 2$, it follows that $v(e')$ has $w_2$ coordinate greater than or equal to $2$. By Lemma \ref{SemiFano Fan Sums}, we have a contradiction. Thus, $v(e)$ lies in the closed cone generated by $w_1$ and $w_2$.
	
	Assume without loss of generality that $v(e) \neq w_1$, and hence has positive $w_2$ coordinate. Since the edge $h(e)$ contains the point $q_s$, moving along $h(e)$ in the outward direction now enters $H$. Another application of Lemma \ref{Tropical Clipping Half Plane} then shows that the edge $e$ must be unbounded. It follows that $v(e)$ must equal $w_2$.
	
	Consider the unique vertex adjacent to $e$. If we consider the three vectors pointing away from this vertex along the three adjacent edges, we must have that one of these vectors has positive $w_1$ coordinate. We then repeat the above argument with the half plane $H'$ of points $q_s + a_1w_1 + a_2w_2$ with $a_1 > \epsilon$ for some small $\epsilon$ to show this edge is also unbounded.
	
	Finally, by Lemma \ref{SemiFano No New Directions}, the final edge adjacent to this vertex, which is the edge named $e'$ in the lemma statement, must have weight 1 and outward pointing primitive vector $v(e')$ either equal to $w_1 + w_2$ or lying in a 1-cone of $Y$, as desired.
	
	
	
\end{proof}

This lemma allows us to completely describe the generalized Maslov index 2 discs in $Y$ with respect to the point constraints $q_s, p_1,\hdots, p_n$ in terms of the generalized Maslov index 2 discs in $\widetilde{Y}$ with respect to the point constraints $p_1,\hdots, p_n$. Every such tropical disc in $Y$ through the point $q_s$ can be ``clipped" by removing the (unbounded) edge through $q_s$ and its neighboring vertex, and extending the only surviving edge (now unbounded) to infinity.

We now describe a procedure for obtaining the superpotential on $\widetilde{Y}$ from the superpotential on $Y$ using a single point constraint $q_s$. Let $U_s$ be our chamber of interest, as before, and let $W_Y$ be the superpotential for $Y$. Since $Y$ is toric semi-Fano, Chan-Lau \cite{CL} tells us that $W_Y$ has a term $a_vz^{v}$ for each primitive vector $v$ in a toric direction, where $a_v$ is 1 if the toric divisor corresponding to $v$ has self intersection not equal to $-2$, and is otherwise determined in the following way: Let $m_v$ be the (total) number of consecutive $-2$ toric divisors in the run of consecutive $-2$ toric divisors containing the divisor corresponding to $v$, and let $\ell_v$ be the position of the divisor corresponding to $v$ within this run (counting from either end). Then, in the terminology of Chan-Lau, $a_v$ is equal to the number of admissible sequences with center $\ell_v$ of length at most $m_v$. Numerically, we have $a_v = {{m_v+1}\choose{\ell_v}}$.

Let $h$ be a generalized Maslov index 2 disc in $Y$ with respect to our single point constraint $qs$ with stop in $U_s$. We will assign this disc a weight $w(h)$. If $h$ does not go through $q_s$, then it consists of a single ray in a toric direction $v$ of $Y$. In this case, we let $w(h)$ be the coefficient of $z^{v}$ in $W_Y$. If $h$ does go through $q_s$, then by Lemma \ref{Tropical Clipping Main}, $h$ consists of three edges, one adjacent to the stop and the other two unbounded. The edge through $p$ has outward pointing primitive tangent vector $-w_i$. Let $-v$ be the other outward pointing primitive tangent vector to the other unbounded edge. Then the primitive tangent vector on the compact edge adjacent to the stop, pointing toward the stop, is $w_i+v$. We let $w(h)$ be the product of the coefficient of $z^{w_i}$ (which is $1$) and the coefficient of $z^{v}$ (which may not be $1$).

Letting $v(h)$ be the primitive vector tangent to the edge adjacent to the stop, pointing toward the stop, we claim that $W_{\widetilde{Y}} = \sum_h w(h) z^{v(h)}$. If $h$ has no vertices, then $w(h) = {{m_{v(h)}+1}\choose{\ell_{v(h)}}}$. If $h$ has a single vertex, then $w(h) = {{m_{v(h)}+1}\choose{\ell_{v(h)} + 1}}$, where here we choose $\ell_{v(h)}$ so that the divisor adjacent to $w_i$ is in position $1$. Here we are using Lemma \ref{SemiFano No New Directions} and the last observation in the proof of Lemma \ref{SemiFano Fan Sums} to find $w(h)$. The coefficient of $z^{v(h)}$ is then precisely ${{m_{v(h)}+1}\choose{\ell_{v(h)}}} + {{m_{v(h)}+1}\choose{\ell_{v(h)} + 1}} = {{m_{v(h)}+2}\choose{\ell_{v(h)} + 1}}$, as desired.

Finally, let $Y$ be a toric Fano surface and consider a sequence $\widetilde{Y}_n\to \cdots \to \widetilde{Y}_1 \to Y$ of single point toric blowups, such that each $\widetilde{Y}_j$ is semi-Fano. By inductively applying the above procedure, we can relate the superpotential on $\widetilde{Y}_n$ to the bulk-deformed superpotential on $Y$ corresponding to appropriately chosen marked points.

Let $w_i'$ now denote the primitive vector lying in the 1-cone of the toric fan of $\widetilde{Y}_i$ that does not appear in the fan of $\widetilde{Y}_{i-1}$ (this corresponds to the $i$th blowup of $Y$). Let $\mathcal{D}_i$ be the scattering diagram associated with the single point $p_i$ and $\widetilde{Y}_i$. This has a ray in the $v$ direction for each primitive vector $v$ in a 1-cone of the fan of $\widetilde{Y}_i$. Letting the point $p_i$ tend to infinity in the $-w_{i}'$ direction, we get one chamber $U_i$ exhausting $M_\R$. By positioning the points $p_1,\hdots, p_n\in M_\R$ such that the chamber $U_i$ contains all points $p_j$ with $j > i$, we claim that the bulk-deformed superpotential on $Y$ at a point in $\bigcap_{i=1}^{n}U_i$ is equal to the un-deformed superpotential on $\widetilde{Y}_n$.

This follows inductively from our above procedure involving single blowups. Since $Y$ is Fano, the above procedure gives us that the superpotential for $\widetilde{Y}_1$ equals the bulk-deformed superpotential on $Y$ in $U_{1}$ with the point $p_{1}$. Considering now the point $p_{2}$ in the chamber $U_1$, the above procedure gives us the superpotential on $\widetilde{Y}_2$ using the superpotential on $\widetilde{Y}_1$ and a weighted count of generalized Maslov index 2 tropical discs with respect to $p_2$. This weighted count corresponds exactly to the number of ways of modifying such a tropical disc by ``unclipping" (replacing an unbounded edge by a generalized Maslov index 2 tropical disc with stop on that edge) to get a generalized Maslov index 2 tropical disc with respect to $p_1$ and $p_2$, so the bulk-deformed superpotential on $Y$ with the points $p_1$ and $p_2$ in the region $U_1\cap U_2$ agrees with the superpotential on $\widetilde{Y}_2$. Inducting gives the desired general result, repeated here:

\begin{theorem}\label{SemiFano Toric Blowup Theorem}
	Let $Y$ be a toric Fano surface, and let $\widetilde{Y}_n\to \cdots \to \widetilde{Y}_1 \to Y$ be a sequence of toric blowups, such that each $\widetilde{Y}_i$ is semi-Fano. Let $p_1,\hdots, p_n$ be points in $M_\R$ arranged as described above, and let $U_1,\hdots, U_n$ be the open cones with vertex at those points described above. The bulk-deformed superpotential for $Y$ in the chamber $\bigcap_{i=1}^nU_i$ with respect to these points is equal to the superpotential for $\widetilde{Y}_n$.
\end{theorem}

\section{SYZ fibrations on Log Calabi-Yau Surfaces and wall-crossing}\label{sec:constsyznon-toric}
In this section, we construct a special Lagrangian torus fibration on $Y\setminus D$ for a log Calabi-Yau surfaces $(Y,D)$ modifying the construction of Auroux \cite{A5} and Abouzaid-Auroux-Kartzarkov \cite{AAK}, and prove that the fibration satisfies conditions in Theorem \ref{LF scattering diagram} away from singular fibres. Notice that the K\"ahler forms in all these constructions are not Ricci-flat. 
The main difficulty to study the Floer theory of $Y\setminus D$ lies in the non-completeness of $Y \setminus D$ as it may spoil the compactness of relevant disc moduli spaces. This occurs because on $Y \setminus D$, we use the restricted symplectic form (or the associated metric) from $Y$.
Once we have the compactness of the moduli spaces, the fibration will automatically produce a scattering diagram well-defined away from neighborhoods of singular fibres, which is one of the main subjects of the paper. It will be discussed thoroughly in Section \ref{sec:mainholotrop}.

\subsection{Construction of a SYZ fibration on a non-toric blowup} \label{Sec: non-toric blowup}
Let us first look into the toric surface $\bar{Y}$ and its non-toric blowup $\widetilde{Y} \to \bar{Y}$.  
 We will equip $\bar{Y}$ with a K\"ahler form $\bar{\omega}$ which is invariant under the $T^2$-action. 
 
In general, for given a $T^2$-action on $\bar{Y}$ compatible with the complex structure and a K\"ahler class, we can take any K\"ahler form, say $\bar{\omega}_0$, in the class, and average it with respect to the $T^2$-action. This leads to the $T^2$-invariant K\"ahler form $\bar{\omega}$ on $\bar{Y}$ that belongs to the K\"ahler class we began with. Since the first Betti number of $\bar{Y}$ vanish, the $T^2$-action is Hamiltonian. It is worth mentioning that for the fixed $T^2$-action on $\bar{Y}$, the base of the moment map fibration (with respect to the K\"ahler forms in different K\"ahler classes) with the complex affine coordinates are canonically isomorphic. 

Let $\{q_i \in \bar{D}_{\sigma(i)}\}_i \subset \bar{Y}$ be the blowup center for $\widetilde{Y} \to \bar{Y}$. Namely, $\widetilde{Y}$ is obtained by taking a simple blowup at each $q_i$. We denote by $E_i:=\pi^{-1} (q_i)$ the exceptional divisor associated with $q_i$.
Notice that we changed the notation from \ref{subsec:GPS} for simplicity. The $q_i$ and $E_i$ were previously written in the form of  $q_{i j}$ and $E_{i j}$ indicating they are attached to the toric divisor indexed by $j=\sigma(i)$.
We will first construct a particular K\"ahler form $\tilde{\omega}_{\epsilon}$ on $\widetilde{Y}$  from the torus-invariant K\"ahler form $\bar{\omega}$ downstairs. 

\begin{lemma}\label{Kahler form}
	Given $\epsilon>0$, there exists a K\"ahler form $\tilde{\omega}_{\epsilon}$ on $\widetilde{Y}$ and  a pair of neighborhoods $\tilde{U}_i \subseteq  \tilde{U}'_i$ of $E_i$ with $\tilde{U}'_i \cap \tilde{U}'_j = \emptyset$ for $i \neq j$ that satisfy the following:
	\begin{enumerate}
		\item $\tilde{\omega}_{\epsilon}=\pi^*\omega$ outside $\cup_i \tilde{U}_i$;  
		\item Symplectic areas of exceptional divisors with respect to $\tilde{\omega}_{\epsilon}$ sum up to $\epsilon$.
		\item For each $q_i\in \bar{D}_{\sigma(i)}$, the sub-toric $S^1$-action that fixes $\bar{D}_{\sigma(i)}$ locally lifts to $\tilde{U}'_i\subseteq \widetilde{Y}$. (In particular, $\tilde{U}'_i$ is $S^1$-invariant.)
	\end{enumerate}
	Moreover, we have $\tilde{U}'_i\searrow E_i$ as $\epsilon \rightarrow 0$. 
\end{lemma}
\begin{proof}
	Set $p_i:= p(q_i) \in \partial P$ where $p:\bar{Y}\rightarrow P$ is the moment map. Let $V_i \subseteq P$ be a small neighborhood of $p_i$ such that $V_i\cap V_j=\emptyset$ if $p_i \neq p_j$.  
	Consider the subgroup $S^1\subseteq T^2$ fixing $\bar{D}_{\sigma(i)}$. If we write the toric action locally as $(e^{i\theta_1},e^{i\theta_2})\cdot (x,y)=(e^{i\theta_1}x,e^{i\theta_2}y)$ for some local coordinates $(x,y)$ near $q_i$ with $\bar{D}_{\sigma(i)} = \{ y=0\}$, then the mentioned $S^1$-action is given by $(x,y) \mapsto (x, e^{i\theta} y)$. Let us choose $S^1$-invariant neighborhoods $U_i$ and $U'_i$ of $q_i$  such that $U_i\subsetneq U'_i \subseteq p^{-1}(V_i)$. We have $U'_i\cap U'_j=\emptyset$ for $i\neq j$. 
	We finally take $\tilde{U}_i:=\pi^{-1}(U_i)\subsetneq \tilde{U}'_i:=\pi^{-1}(U'_i)$ as  neighborhoods of the exceptional divisor $E_i$. 
	
	Since $U'_i$'s are mutually disjoint, the local $S^1$-action lifts to $\tilde{U}'_i$ by the universal property of blowup, which we denote by $a_\theta : \tilde{U}'_i \to \tilde{U}'_i$. Observe that the complement of $\cup_i \tilde{U}_i$ in $\widetilde{Y}$ admits a $T^2$-action lifted from the toric action on $\bar{Y}$. The local $S^1$-action can be identified with the corresponding factor of this $T^2$-action away from $E_i$, in particular, on $\tilde{U}'_i \setminus \tilde{U}_i$. 
	
	It is a standard fact that there exists a K\"ahler metric $\tilde{\omega}$ on $\widetilde{Y}$  which coincides with $\pi^*\bar{\omega}$ outside $\tilde{U}_i$ (see, for e.g., \cite[p. 185]{GH}). 
	Since $\bar{\omega}$ is invariant under the toric action, so is $\tilde{\omega}|_{\widetilde{Y} \setminus \cup \tilde{U}'_i}$ under the lifted $T^2$-action on $\widetilde{Y} \setminus \cup \tilde{U}'_i$. 
	On the other hand, one can average $\tilde{\omega}|_{\tilde{U}'_i}$ by the local $S^1$-action with respect to the Haar measure $d \mu (\theta)$ on $S^1$. This produces an $S^1$-invariant local K\"ahler form $\tilde{\omega}_i=\frac{\int_{S^1} a_\theta^\ast \tilde{\omega} \,\, d \mu (\theta)}{\int_{S^1}  d \mu (\theta)}$ on $\tilde{U}'_i$. Notice that the local form $\tilde{\omega}_i$ agrees with the global form $\tilde{\omega}$ on the overlapped region $\tilde{U}'_i\setminus \tilde{U}_i$ by $T^2$-invariance of $\tilde{\omega}$.
	
	We conclude that there exists a K\"ahler form on $\widetilde{Y}$ which equals $\tilde{\omega}=\pi^*\omega$ on the complement of $\cup_i \tilde{U}_i$ and restricts to $\tilde{\omega}_i$ on $\tilde{U}'_i$. Repeating the same procedure for all $q_i$'s, we obtain a global K\"ahler form, say $\tilde{\omega}_{\epsilon}$, on $\widetilde{Y}$ which certainly satisfies the conditions (1) and (3) in the statement. The condition (2) can be easily achieved by adjusting the size of the K\"ahler potential for $\tilde{\omega}|_{\tilde{U}_i}$. 
	
\end{proof}

\begin{remark}\label{rmk:momentmaplocal}
We remark that the local $S^1$-action on $\tilde{U}_i$ above is Hamiltonian, since it is symplectic (preserving $\tilde{\omega}_{\epsilon}$) and the first Betti number of $\tilde{U}_i$ vanishes. Also, the $S^1$-action can actually be defined on a larger region away from $\cup_{j\neq i} E_j$, as it is induced from the toric action on $\bar{Y}$ which lifts to $\widetilde{Y} \setminus \cup_i \tilde{U}'_i$. In particular, we can find a moment map $\mu_{S^1}: (\pi\circ p)^{-1}(V_i)\rightarrow \mathbb{R}$ for this action.
\end{remark}

With a K\"ahler form $\tilde{\omega}_{\epsilon}$ on $\widetilde{Y}$ constructed in the lemma and the meromorphic volume form $\tilde{\Omega}$ on $\widetilde{Y}$ having simple pole along $\widetilde{D}$, we adapt the constructions in \cite{A5, G2} into our setting in order to obtain an almost special Lagrangian fibration on $\tilde{X}=\widetilde{Y}\setminus \widetilde{D}$.

\begin{lemma}\label{Lagrangian fibration}
	There exists a Lagrangian fibration $\tilde{X}\rightarrow B$ with respect to $\tilde{\omega}_{\epsilon}$ such that
	\begin{enumerate}
		\item the fibration coincides with the pull-back of toric fibration on $\bar{X}$ outside $(\pi\circ p)^{-1}(V_i)$; 
		\item $\mbox{Im}\, {\tilde{\Omega}} |_{\tilde{L}}=0$ for any Lagrangian torus fibre $\tilde{L}$ after a suitable normalization of the phase; 
		\item all the singular Lagrangian torus fibres are immersed $S^2$, which are in one-to-one correspondence with blowup points $\{q_i\}$.
	\end{enumerate}
\end{lemma}

\begin{proof}
As in the proof of Lemma \ref{Kahler form}, let us take local coordinates $(x,y)$ around $q_i$ such that the toric divisor $D_{\sigma(i)}$ (that contains $q_i$) is locally given by $\{ y=0\}$, and the toric action trivializes as $(x,y) \mapsto (e^{i \theta_1} x, e^{i \theta_2} y)$. We normalize the moment map $\mu_{S^1}:(\pi\circ p)^{-1}(V_i)\rightarrow \mathbb{R}$ in Remark \ref{rmk:momentmaplocal} so that it vanishes on the proper transform of the $x$-axis, and $\mu_{S_1}=\epsilon_i<\epsilon$ at the isolated $S^1$-fixed point in $E_i$.

Notice that $\{|x|=const.\}$ projects to a circle in the $S^1$-reduction which is automatically a Lagrangian by the dimension reason. 
Therefore \cite[Theorem 1.2]{G2} applies to produce a Lagrangian fibration  
\begin{equation}\label{eqn:locfibbu}
(\pi\circ p)^{-1}(V_i) \rightarrow \mathbb{R}^2 \qquad z \mapsto (|x(\pi(z))|,\mu_{S^1}(z))
	\end{equation} 
	 with respect to $\tilde{\omega}_{\epsilon}$ on $\tilde{U}_i$, where $x(\pi(z))$ denotes the $x$-coordinate of $\pi(z)$.  
	Since $\tilde{\omega}_{\epsilon}=\pi^*\omega$, the fibration coincides with the pull-back of the moment map fibration from $\bar{X}$ outside of $(\pi\circ p)^{-1}((\pi\circ p)(\tilde{U}_i))$. Therefore, the fibration above can be glued with the pull-back of the moment fibration of $\bar{X}\rightarrow \mbox{Int}P$. Instead of using explicit coordinates as in \eqref{eqn:locfibbu}, one may take a splitting of $T^2$ (giving the toric action on $\bar{Y}$) into the product of the stabilizer of $D_{\sigma(i)}$ and its complementary circle, and replace $|x(\pi(z))|$ by the pull-back of the moment map for the latter factor. This way, one can identify the target of the fibration \eqref{eqn:locfibbu} coherently with the Lie algebra of $T^2$ for every $i$.

	To see the second part of the lemma, we first show that $\tilde{\Omega}$ is invariant under the $S^1$-action in a neighborhood of the proper transform of $\bar{D}_{\sigma(i)}$. Since $\tilde{\Omega}=\pi^{\ast}\bar{\Omega}$ outside of $(\pi\circ p)^{-1}(V_i)$ and $\bar{\Omega}$ is torus-invariant, $\tilde{\Omega}$ is $S^1$-invariant outside $(\pi\circ p)^{-1}(V_i)$. 
	Together with the fact that the $S^1$-action is holomorphic (i.e., it preserves the complex structure), this implies that the holomorphic $(2,0)$-form $a_{\theta}^*\tilde{\Omega}$ on $(\pi\circ p)^{-1}(V_i)$ can be extended as a nowhere vanishing holomorphic $(2,0)$-form on the entire $\tilde{X}$, which we still denote by $a_{\theta}^*\tilde{\Omega}$. Then the ratio $\tilde{\Omega}/a_{\theta}^*\tilde{\Omega}$ is a holomorphic function on $\tilde{X}$ which is $1$ on an open set, and we conclude that  $a_{\theta}^*\tilde{\Omega}=\tilde{\Omega}$ by maximum principle. Then (2) of the lemma again follows from \cite{G2}*{Theorem 1.2}.
	
	Finally, the fibre of \eqref{eqn:locfibbu} over $(r,\lambda)\in \mathbb{R}^2$ is singular if and only $(r,\lambda)=(|a_i|,\epsilon_i)$ where $\epsilon_i=\int_{E_i}\tilde{\omega}_{\epsilon}$ and $a_i = x(q_i)$. It has nodal singularities only, and hence the fibre can be at worst immersed. The singular point of the Lagrangian fibration happens at the isolated fixed point of the $S^1$-action, and hence is one-to-one corresponding to $q_i$. 
	
\end{proof}

Away from a neighborhood of the discriminant, the base $B$ of the Lagrangian fibration constructed in Lemma \ref{Lagrangian fibration} admits the affine structure which can be described as follows.
Let us first identify the interior of the polytope $\mbox{Int}\, P$ with $M_{\mathbb{R}}=\mathbb{R}^2$ via the Legendre transform so that the complex affine coordinates from the toric fibration $p:\bar{X} \rightarrow \mbox{Int}\, P$ agrees with the standard coordinates on $\mathbb{R}^2$. Recall that the Lagrangian fibration on $\tilde{X}=\widetilde{Y}\setminus \widetilde{D}$ in Lemma \ref{Lagrangian fibration} is the pull-back of the moment fibration of $X$ away from neighborhoods $(\pi\circ p)^{-1}(\bigcup_i V_i)$  of the exceptional divisors. Hence one can consider the pull-back of the complex affine structure on a subset $B'$ of $B$ which is diffeomorphic to $\mathbb{R}^2\setminus \bigcup_i V_i$, and we may still call it the complex affine structure. 

\begin{lemma}
Let $L$ be a moment fibre in $\bar{Y}$, and fix $\pi^{-1}(L)$ as a reference fibre for the complex affine structure (so that it sits over the origin). 
Given any $C>0$, there exists $\epsilon>0$ such that the affine coordinate $x_{\partial \gamma_i}>C$ of the singular fibre corresponds to $q_i$, if $\epsilon_i<\epsilon$. 
\end{lemma}	
\begin{proof}
   This is because the fibration constructed in Lemma \ref{Lagrangian fibration} coincides with the pull-back of the moment fibration outside of $(\pi\circ p)^{-1}(\bigcup_i V_i)$. The holomorphic volume form $\bar{\Omega}$ near the interior of the boundary divisor component $\{x=0\}$ (in some local coordinates) is of the form $f\frac{dx}{x}\wedge dy$, where $f$ is a non-zero function. On the other hand, recall that the complex affine coordinates are given by $x_{\partial \gamma_i}=\int_{C_{\partial \gamma_i}} \mbox{Im} \tilde{\Omega}$ where $C_{\partial \gamma_i}$ is the 2-chain swept by the cycle $\partial \gamma_i \in H_1 (L)$, starting from $L$ and moving toward the singular fibre corresponding to $q_i$.
 As $\epsilon_i \to 0$, one of the boundary components of $C_{\partial \gamma_i}$ approaches $x=0$, and hence $\int_{C_{\partial \gamma_i}}\mbox{Im}\bar{\Omega} \rightarrow \infty$.
\end{proof}

In particular, as $\epsilon$ approaches $0$, one can take $V_i$ in the construction arbitrarily small so that $B'\cong (\mathbb{R}^2\setminus \bigcup_i V_i) \nearrow \mathbb{R}^2$ as $\epsilon\rightarrow 0$. When needing to emphasize this dependence of $B'$ on $\epsilon$, we will write $B_\epsilon$ instead of $B'$ in what follows.

\subsection{Holomorphic discs in non-toric blowups}\label{subsec:holoinntbl}

We now look into holomorphic discs bounding a fibre $\tilde{L}$ of the fibration Lemma \ref{Lagrangian fibration}, when $\tilde{L}$ is away from $(\pi\circ p)^{-1}(\bigcup_i V_i)$. 
Since we are using the standard (toric) complex structure on $\bar{Y}$ and the one on $\widetilde{Y}$ naturally induced through the blowup process, $\pi : \widetilde{Y} \to \bar{Y}$ is holomorphic.  This enables us to establish the following one-to-one correspondence between the holomorphic discs bounding by $L$ and $\tilde{L}$. If $f:(D^2,\partial D^2)\rightarrow (\bar{Y},L)$ is a holomorphic map, then its proper transform gives a holomorphic map $\tilde{f}:(D^2,\partial D^2)\rightarrow (\widetilde{Y},\tilde{L})$. On the other hand, if $\tilde{f}:(D^2,\partial D^2)\rightarrow (\widetilde{Y},\tilde{L})$ is holomorphic, then composition with $\pi$ gives a holomorphic map $f=\pi\circ \tilde{f}:(D^2,\partial D^2)\rightarrow (\bar{Y},L)$. See \cite{VWX} for related discussions.

Suppose that the disc $\tilde{f}:(D^2,\partial D^2) \rightarrow (\widetilde{Y},\tilde{L})$ has Maslov index zero. This happens if and only if $\mbox{Im}(\tilde{f})\cap \widetilde{D}=\emptyset$. Then its projection $f$ can only intersect $D$ at $q_i$ for some $i$. If $|\mbox{Im}(\tilde{f})\cap \widetilde{D}| = l$, then $\mbox{Im}(f)$ passes $l$-many point-constraints in $D$ counted with multiplicity, and its Maslov index is given by $2l$. With these point-constraints being regarded as bulk-insertions, $f$ has generalized Maslov index $0$.
More generally, the same argument shows that the Maslov index of $\tilde{f}$ is the same as the generalized Maslov index of $f$ allowing multiple insertions of $q_i$. To sum up, we have the following lemma.

\begin{lemma}\label{non-toric MI co} There exists a bijective correspondence between the Maslov index $\mu$ holomorphic discs for $(\tilde{X},\tilde{L})$ and  the generalized Maslov index $\mu$ holomorphic discs for $(X,L)$ viewing $q_i$ as bulk insertions. In particular, $L_u$ bounds a generalized Maslov index zero disc if and only if $\tilde{L}_u$ bounds for a Maslov index zero disc.
\end{lemma} 

The following corollary is a direct consequence of (2) of Lemma \ref{Lagrangian fibration}, since fibres in $\tilde{X}$ are special with respect to $\tilde{\Omega}$ ((2) of Lemma \ref{Lagrangian fibration}). Alternatively, one can use the following fact combined with Lemma \ref{non-toric MI co}:
when point-constraints approach the toric boundary divisor, any wall for the bulk-deformed potential limits to a straight line (the proof is elementary). 

\begin{corollary}
	If $\tilde{L}_t\subseteq \tilde{X}$ is a family of torus fibre over $B'$ bounding holomorphic discs of Maslov index zero in a fixed relative class, then $\tilde{L}_t$ sit over an affine line in $B'$. 
\end{corollary}

Notice that the statement holds only for genuine holomorphic discs, but not for stable maps having nontrivial bubbled-off components that may appear in the compactification $\widetilde{Y}$, especially when there exists holomorphic spheres with negative Chern numbers. 
We can get rid of such contributions simply by restricting ourselves to the open part $(\tilde{X},\tilde{\omega}_{\epsilon})$, but the price to pay is that the compactness of the disc moduli is not automatically guaranteed, which we discuss now.

\subsubsection{Compactness of the disc moduli for $\tilde{X}$}\label{sec: compactness}

In order to make sense of Floer theory for Lagrangian fibre $\tilde{L}$ in $\tilde{X}$, one needs to separately show the compactness of the corresponding disc moduli spaces since $\tilde{X}$ is non-compact and the metric is incomplete.
Once we have the compactness, Theorem \ref{LF scattering diagram} automatically produces a scattering diagram $\mathfrak{D}^{LF}_{\epsilon}$ on $B_{\epsilon}$. We will see later that it converges to a limiting scattering diagram as $\epsilon$ approaches $0$.

\begin{remark}
	One may construct a scattering diagram from the geometry of $\widetilde{Y}$ instead, for which Theorem \ref{LF scattering diagram} directly applies. However, there can be a lot  more Maslov index zero discs due to negative Chern number spheres in $\widetilde{D}$ and,  none of existing techniques in Floer theory seem allowing us to explicitly compute such contributions. 
\end{remark}

\begin{lemma}\label{lem:cptnessnodiv}
	Let $\tilde{L}:=\pi^{-1}(L)$ be a torus fibre in $\tilde{X}$ which is a pre-image of a moment map fibre in $X$. Then the moduli space $\mathcal{M}(\tilde{X},\tilde{L},\tilde{\beta};J)$ of holomorphic discs of class $\tilde{\beta} \in \pi_2 (\tilde{X},\tilde{L})$  is compact for any $\tilde{\beta}$.
\end{lemma} 	

Throughout (as well as in the statement), $J$ denotes the \emph{standard} complex structure unless specified otherwise, and by abuse of notation, it will also stand for that on $\bar{X}$.
We remark that the proof  heavily relies on the usage of the standard complex structure.

\begin{proof}
Recall from Lemma \ref{non-toric MI co} that the set of Maslov index zero discs $\tilde{f}:(D^2,\partial D^2)\rightarrow (\tilde{X},\tilde{L})$ is in one-to-one correspondence with that of generalized Maslov index zero discs $\bar{f}:(D^2,\partial D^2)\rightarrow (\bar{Y},\bar{L})$. We denote by $\mathcal{M}''$, the subspace of $\mathcal{M}(\bar{Y},L,\beta;J)$ consisting of such discs $\bar{f}$.
Hence, the moduli space of stable discs $\mathcal{M}(\tilde{X},\tilde{L},\tilde{\beta};J)$ is homeomorphic to $\mathcal{M}''$. Notice that although $\mathcal{M}''$ does not contain stable maps with sphere-bubble components completely lying in the toric boundary $\bar{D}$, it can still contain discs with sphere bubbles which lift to interior $(-2)$-curves in $\tilde{X}$. These sphere bubbles in $\bar{Y}$ need to pass through prescribed point constraints, and are rigid. 

We show that $\mathcal{M}''$ is compact in the following.
Let $\Sigma'$ be a refinement of $\Sigma$ by adding $\mathbb{R}_{\leq 0}v$ for each $1$-cone $\mathbb{R}_{\geq 0}v$ of $\Sigma$. Let $Y'$ be the toric (orbi-)surface associated $\Sigma'$, and $\pi':Y'\rightarrow \bar{Y}$ the birational modification on the boundary corresponding to the refinement of the fan. Then the $(\mathbb{C}^*)^2$-action on $\bar{Y}$ lifts to $Y'$ by the universal property of blowups. In particular, the Hamiltonian $T^2$-action lifts to $Y'$, and  hence is again Hamiltonian since $H^1(Y')=0$. In other words, the moment maps of $Y'$ and $\bar{Y}$ are compatible, and one can naturally identify the interior of the moment polytopes $P'$ and that of $P$. 

Take $u \in \mathbb{R}^2$ over which the fibre $\bar{L}$ sits, where we identify both $\mbox{Int}\, P$ and $\mbox{Int}\, P'$ with $\R^2$ via the Legendre transform. 
Up to a $(\mathbb{C}^*)^2$-action, we may assume that $u=(0,0)\in \mathbb{R}^2\cong \mbox{Int}\, P$. 
Consider a holomorphic disc $\bar{f}:(D^2,\partial D^2)\rightarrow (\bar{Y},\bar{L})$ which lies in $\mathcal{M}''$.
Recall from the discussion in Section \ref{subsec:weakcorresholotrop} that with help of the anti-holomorphic involution $\sigma$ on $\bar{X} (\cong (\C^\ast)^2)$, one can double $\bar{f}$ into a unique rational curve $C_{\bar{f}}$ in $Y'$ containing $\mbox{Im}\bar{f}$ . We denote its proper transform in $Y'$ by $C'_{\bar{f}}$. 

Observe that the involution $\sigma$ acting on $(\mathbb{C}^*)^2\subseteq Y'$ extends to the whole $Y'$ due to its symmetry. Then $C'_{\bar{f}}$ is a rational curve in $Y'$ whose intersection with the toric boundary divisors can occur only at $q'_i$ and $\sigma(q'_i)$, where $q'_i=\pi'^{-1}(q_i)$. 
It is shown in \cite[Lemma 4.2]{GPS} that the moduli space rational curves in $Y'$ of a fixed homology class which can intersect the toric boundary only at prescribed non-toric points is compact and has no sphere bubble. This moduli space admits an induced action of $\sigma$, and the fixed loci of the action precisely consists of the proper transform $C'_{\bar{f}}$ for some $\bar{f}$. Therefore, $\mathcal{M}''$ can be identified with the fixed loci, and hence is compact.
\end{proof}	

\begin{corollary}\label{cor:defdlfepsilon}
There exists a consistent scattering diagram $\mathfrak{D}^{LF}_{\epsilon}$ on $B_{\epsilon}$ whose walls consist of Lagrangian fibres that bound Maslov $0$ discs (or, more precisely, their corresponding open Gromov-Witten invariants are nonzero).
\end{corollary}

\subsubsection{Computing wall-crossing transformations}

Lemma \ref{lem:cptnessnodiv} enables us to apply Theorem \ref{LF scattering diagram} to the Lagrangian fibration $\tilde{X}$ with small neighborhoods of singular fibres being removed. It results in a scattering diagram, whose ``initial rays" are originated from these singular fibres. Here, the initial rays refer to the ones induced by Maslov 0 discs emanating from these singular fibres, and we often call these discs \emph{initial (Maslov 0) discs} for this reason. We warn readers that a priori there may exist more Maslov 0 discs coming out from neighborhoods of the singular fibres since we did not established the compactness of the disc moduli spaces for fibres near the singular fibres and the Floer theory cannot apply directly. Later we will see that one can use Lemma \ref{weak correspondence} to exclude such possibility. 
 
Consider a Lagrangian torus fibre $\tilde{L}\subseteq (\pi\circ p)^{-1}(V_i)$ in a neighborhood of the singular fibre.
In terms of local coordinates $(x,y)$ around the blowup point $q_i=(a_i,0)$ as in \eqref{eqn:locfibbu}, $\tilde{L}$  bounds Maslov 0 discs if and only if it sits over $|x \circ \pi|=|a_i|$. For later use, we write $\gamma$ for the relative class of the contributing Maslov 0 disc, which are portions of (the proper transform of) the parallel to $y$-axis. We are interested in those lying above the singular fibre (i.e., $\mu_{S_1} > \epsilon_i$), since we are considering the case  when the location of the singular fibre is arbitrarily close to the divisor $y=0$ (i.e., $\epsilon_i \to 0$). Let us denote the corresponding ray (wall) by $l_i$, i.e., in local coordinates,
\begin{equation}\label{eqn:initialwallli}
l_i =\{  (r,\lambda) \mid r= |x\circ \pi| = |a_i|, \,\, \lambda > \epsilon_i \}.
\end{equation}
Notice that $l_i$ forms a ray starting from (the image of) $q_i$, whose direction is normal to the toric divisor that $q_i$. Hence, $\{l_i\}$ and the initial rays(walls) in Section \ref{subsec:GPS} coincide with each other, after identifying the interior of the polytope $\mbox{Int}\, P$ with $\mathbb{R}^2$ via the Legendre transform. 

We show that $\{l_i\}$ coupled with wall functions from Floer theory agrees with the initial scattering diagram $\mathfrak{D}^{GPS}_{in}$ in Section \ref{subsec:GPS}.
Let us compute the wall-crossing transformation associated with the wall $l_i$. Our strategy is to first symplectically embed the neighborhood of the singular fibre $(\pi\circ p)^{-1}(V_i \setminus \partial P)$ into a simpler local model in \cite[Example 3.1.2]{A5}, the blowup $\WT{\C^2}$ of $\C^2$ at a generic point in the $x$-axis. It is crucial to have a $S^1$-invariant K\"ahler form on this region for this purpose. Our actual geometric situation is possibly different from this local model, in that the divisor containing $q_i$ may serve as a nontrivial sphere bubble in disc counting, while in the local model only a part of the divisor appears (as the $x$-axis). However, this does not affect Fukaya's trick across the corresponding wall $l_i$, as $l_i$ is located away from this divisor, and hence no discs for Fukaya's trick across $l_i$ can interact with the divisor. Therefore the wall-crossing for $l_i$ agrees with that for the corresponding wall in \cite[Example 3.1.2]{A5}.

The advantage of the latter is that one can detect the wall-crossing via the change of superpotentials whose computation is simpler. A similar idea was used in \cite{Y5}. One may regard the local coordinates $(x,y)$ above as coordinates on $\C^2$. By abuse of notation, let $\gamma$ denote the relative class in $\pi_2 (\WT{\C^2},\tilde{L})$ of the proper transform of the basic disc intersecting $x$-axis once. Similarly, we denote by $\delta$ the relative disc class of the one intersecting $y$-axis once. Then $\partial \gamma$ and $\partial \beta$ form a basis of $H_1(\tilde{L})$, and we can
take $z^{\partial \gamma}$ and $z^{\partial \delta}$ as variables for the superpotential.
It is not difficult to see that the wall-crossing transformation going across $l_i$ can be written in the form of
\begin{equation}\label{eqn:wcunknownf}
(z^{\partial \gamma} ,z^{\partial \delta}) \mapsto (z^{\partial \gamma},z^{\partial \delta} (1+ f(z))).
\end{equation}
The first component is the identity map since $\partial \gamma$ does not intersect the boundary of the Maslov zero disc at all, and hence contributes trivially to the pseudo-isotopy, and the second follows from Theorem \ref{YS}. On the other hand, the explicit calculation in \cite{A5} shows that the superpotential for $\tilde{L}$ with $|x \circ \pi| < a_i$ is given by
$$T^{\omega(\gamma)}z^{\partial \gamma} + T^{\omega(\delta)}z^{\partial \delta}$$
whereas $\tilde{L}$ with $|x \circ \pi| > a_i$ has 
$$ T^{\omega(\gamma)}z^{\partial \gamma} + T^{\omega(\delta)}z^{\partial \delta} + T^{\omega(\gamma+\delta)}z^{\partial \gamma} z^{\partial \delta}.$$
By (1) of Lemma \ref{wall-crossing}, the two should be compatible with the coordinate transition \eqref{eqn:wcunknownf}. Therefore we have
$$T^{\omega(\gamma)}z^{\partial \gamma} + T^{\omega(\delta)}z^{\partial \delta} (1+ f(z)) = T^{\omega(\gamma)}z^{\partial \gamma} + T^{\omega(\delta)} z^{\partial \delta} + T^{\omega(\gamma+\delta)}z^{\partial \gamma} z^{\partial \delta},$$
which leads to $f(z) = T^{\omega(\gamma)} z^{\partial \gamma}$.
 Thus we have shown 

\begin{lemma} \label{contribution of initial discs}
Suppose $\widetilde{Y}$ is a non-toric blowup of a toric surface $\bar{Y}$, and a point $q_i$ in the blowup center lies in the irreducible toric divisor $D_{\sigma(i)}$. Let $l_i$ be the wall associated to $q_i$ (see \eqref{eqn:initialwallli}) for the Lagrangian fibration on $\widetilde{Y}$ constructed in Lemma \ref{Lagrangian fibration}.
Then its wall-function is given by $1+T^{\omega(\gamma)}z^{\partial \gamma}$ where $\gamma$ is the relative class of the proper transformation of the holomorphic disc in $\bar{Y}$ intersecting $D_{\sigma(j)}$ exactly once.
\end{lemma}

A direct consequence is that $\{l_i\}$ together with the wall functions (of the form $1+f(z)$) calculated above forms the (initial) scattering diagram $\mathfrak{D}_{in}^{GPS}$ in Section \ref{subsec:GPS}.

\subsection{Orbifold blowups and Floer theory} \label{sympl orbifold}
Although we are mainly concerned with the case of simple blowups in the article, we believe that majority of our result extends to blowups with higher multiplicities (orbifold blowups) provided that Fukaya's trick is established in the orbifold setting. Here, we briefly examine the wall-crossing in Floer theory on some standard local model for such a blowup. 

Recall from \ref{subsubsec:orbiblowupGPS} that a non-toric blowup with a higher multiplicity  results in the nontrivial exponents in the wall functions attached to the corresponding wall in $\mathfrak{D}^{GPS}$. It creates an orbifold singularity in the total space, and the Maslov zero disc responsible for the wall-crossing passes through this singular point. Hence it is natural to consider the orbifold Floer theory of the Lagrangian torus fibres sitting over the wall.
Let us examine this in a local model which will turn out to be given as the \emph{global quotient orbifold}.

We begin with $\mathbb{C}^2$ with coordinates $(x,y)$ (equipped with the standard toric structure). As the blowup we are interested in is non-toric, let us take a nonreduced scheme supported at the point $(1,0)$, that is, the subscheme corresponding to the ideal generated by $(x-1)^r$ and $y$. ($r$ corresponds to $r_{ij}-1$ in \ref{subsubsec:orbiblowupGPS}.) The blowup will produce an orbifold projective line $E_r$ with exactly one orbifold singular point of order $r$. It admits a homogeneous coordinate $[a:b]$ where $[a:b] = [\rho a : \rho^r b]$ for $\rho \in \mathbb{C}^\times$. Now, the blowup can be identified as a subspace of $\mathbb{C}^2 \times E_r$ given as
$$ \widetilde{Y}_{r,loc}= \{ ((x,y), [a:b] ) \in  \mathbb{C}^2 \times E_r : (x-1)^r b = y a^r\}.$$

Like in the case of ordinary (non-toric) blowup, $\widetilde{Y}_{r,loc}$ admits a $S^1$-action. The action has exactly one isolated fixed point $((1,0),[0:1])$. $((1,0),[0:1])$, away from which $\widetilde{Y}_{r,loc}$ is smooth. It also fixes the proper transform of $x$-axis, $\{ ((x,y),[a:b]) : b=0 \}$. We denote its complement in $\widetilde{Y}_{r,loc}$ by $ \tilde{W}_{r,loc}$. This is isomorphic to the affine variety
$$\{(x,y,\tilde{a}) :  x^{r} =\tilde{a} y \} \cong \mathbb{C}^2 / \mathbb{Z}_r $$
via $\tilde{a} = \frac{a^r}{b}$, under which the point $((1,0),[0:1])$ in $\tilde{W}_{r,loc}$ maps to the (unique) orbifold singular point of order $r$. In fact, there exists a (global) quotient map 
$$\tilde{W}_{1,loc} \to \tilde{W}_{r,loc} \qquad ((\underline{x},\underline{y}),[\underline{a}:\underline{b}]) \to ((\underline{x},\underline{y}^r), [\underline{a}:\underline{b}^r] ) $$
from the local model for the ordinary non-toric blowup, which is nothing but the quotient map $\mathbb{C}^2 \to \mathbb{C}^2 / \mathbb{Z}_r $ by the identification above.

Notice that the quotient map sends a torus fibre $L_{1,u}$ in $\tilde{W}_{1,loc}$ to the torus fibre $L_{r,u}$. Restricting to the individual torus fibre, the quotient map gives a $r$-fold cover, where the loop $|\underline{y}| = \mathrm{const.}$ in $L_{1,u}$ wraps the loop $|y| = \mathrm{const.}$ in $L_{r,u}$ $r$-times through the covering, yet the complementary loop $|\underline{x}|=\mathrm{const.}$ maps isomorphically $|x|=\mathrm{const.}$ in $L_{r,u}$. On the local model $\tilde{W}_{r,loc}$, Floer theory of its torus fibres can be completely understood by $\mathbb{Z}_r$-invariant part of Floer theory of their liftings in $\tilde{W}_{1,loc}$, which has already been discussed in Section \ref{subsec:holoinntbl} in detail. For our purpose, we fist need to inspect the effect of the group action on the Maurer-Cartan space of the upstairs torus fibre $L_{1,u}$. Let $e_{\underline{x}}$ and $e_{\underline{y}}$ denote the standard generator of $H^1(L_{1,u}; \mathbb{Z})$, i.e., $\int_{|i|=\mathrm{const.}} e_j = \delta_{ij}$ ($i,j \in \{ \underline{x},\underline{y}\}$).

Since the action rotates the loop $|\underline{y}|=\mathrm{const.}$ by $2\pi / r$, the generator of the first (integral) cohomology of the torus fibre down stairs can be identified with $e_x=e_{\underline{x}}$ and $e_y =\frac{1}{r} e_{\underline{y}} $ (so that $e_y$ integrates over $|y|=\mathrm{const.}$ to give $1$). A general element of $H^1(L_{r,u}; \mathbb{R})$ is then a linear combination $x_1 e_x + x_2 e_y = x_1 e_{\underline{x}} + (x_2 /r) e_{\underline{y}}$. Recall that the potential is written in terms of the exponential coordinates $z_i = \exp (x_i)$, and the above discussion tells us that the exponential coordinates $(\underline{z_1},\underline{z_2})$ for $L_{1,u}$ and $(z_1,z_2)$ for $L_{r,u}$ are related by
$$ z_1 = \underline{z_1},\quad z_2 = (\underline{z_2})^r.$$
On the other hand, we have computed in Lemma \ref{contribution of initial discs} the wall-crossing formula for the unique wall $\{|\underline{x}|=1\}$ in the standard local model $\tilde{W}_{1,loc}$ as
\begin{equation}\label{eqn:wcupstairwloc}
(\underline{z_1}',\underline{z_2}') = (\underline{z_1},\underline{z_2}(1+ T^{\omega(\gamma)} \underline{z_1})),
\end{equation}
where $\gamma$ is the class of the Maslov index 0 disc responsible for the wall and where $(\underline{z_1},\underline{z_2})$ and $(\underline{z_1}',\underline{z_2}')$ are the (exponential) coordinates on the Maurer-Cartan spaces of the torus fibres in the chambers $\{|\underline{x}| <1\}$ and $\{|\underline{x}| >1\}$, respectively. In perspective of Floer theory of $\tilde{W}_{r,loc}$ yet without turning on the bulk insertions from twisted sectors, the formula \eqref{eqn:wcupstairwloc} is still valid if we write it in terms of coordinates $(z_1,z_2)$, which amounts to taking $\mathbb{Z}/r$-invariant part of Floer operations upstairs. It leads to the following wall-crossing formula
$$(z_1',z_2') =( z_1, z_2 (1+T^{\omega(\gamma)} z_1^r)))
$$
that is valid in the local model $\tilde{W}_{r,loc}$ whose Floer theory is taken as the $\Z_r$-invariant part of that on $\tilde{W}_{1,loc}$ along the same spirit as \cite{CH}.
The corresponding global statement should immediately follow, once the orbifold analogue of the Fukaya's trick is established (i.e., Lemma \ref{wall-crossing} in the orbifold setting). We remark that for our purpose, it does not require a full package of the orbifold Floer theory in the sense that the bulk-deformation by nontrivial twisted sectors is not at all involved. 
 
\subsection{Toric blowups}\label{Sec: toric blowup}
 Assume that $\pi':\widetilde{Y}\rightarrow Y$ is a toric blowup, i.e. 
  all of $q_i\in Y$ are located at corners of $D$.
  When $Y$ is a smooth toric surface, then so is $\widetilde{Y}$ whose moment polytope is obtained by chopping off the corresponding corner of the polytope.
  We set $\widetilde{D}:=\pi'^{-1}(D)$, which represents $c_1 (\tilde{X})$, and $\tilde{X}=\widetilde{Y}\setminus \widetilde{D}, X=Y\setminus D$. 
  
  Since $\tilde{X}$ are $X$ are biholomorphic, the holomorphic volume forms $\Omega$ and $\tilde{\Omega}$ on $X$ and $\tilde{X}$ are related by $\tilde{\Omega}=\pi'^*\Omega$. Any K\"ahler form $\tilde{\omega}$ on $\tilde{X}$ naturally induces a K\"ahler form $\omega'=(\pi'^{-1})^*\tilde{\omega}$ on $X$.\footnote{Later, we will apply this to $\tilde{\omega}=\tilde{\omega}_{\epsilon}$ constructed in the previous subsection.} In particular, $L\subseteq X$ is a Lagrangian with respect to $\omega'$ if and only if $\pi'^{-1}(L)\subseteq \tilde{X}$ is a Lagrangian with respect to $\tilde{\omega}$.  However, $\omega'$ does not necessarily extend to the compactification $Y$. Nevertheless, we can use the following lemma to link Floer theory on $X$  and that on $Y$. 
  
\begin{lemma} Assume that $D$ supports an ample divisor.\footnote{Under the same assumption, Gross-Hacking-Keel \cite[Remark 0.4]{GHK} proved that the mirror family is algebraic.}
	Given any $\omega'$ on $X$ as above, there exists a sequence of K\"ahler form $\omega_i$ on $Y$ such that $\omega_i=\omega'$ on $U_i$ for some sequence $\{U_i\}_{i \geq 1}$ of relatively compact open subsets with $U_i \nearrow X$.
\end{lemma}
\begin{proof}
   Assume that $D'=\sum_j a_jD_j$ is an ample divisor with support on $D$ for some $a_j\geq 0$. Then by replacing $D'$ by $lD'+\sum_{j:a_j=0}D_j$, for $l\gg 0$, we may assume that $a_j>0$ for all $j$. 
   Then for $k\gg 0$, $\sum_j ka_jD_j$ is a very ample divisor and thus $X$ is Stein. 
   
    Choose a Hermitian metric $e^{-\psi''}$  of the line bundle $\mathcal{O}_{Y}(D')$ such that the associate curvature is a K\"ahler form $\omega''$. In particular, $\omega''$ is defined globally on $Y$. Since $X$ is Stein and hence the $\partial \bar{\partial}$-lemma holds, we have $
         \omega'|_{X}=i\partial \bar{\partial}\psi', \omega''|_{X}=i\partial \bar{\partial}\psi''$ for some smooth functions $\psi',\psi''$ on $X$. 
    Take $\omega_i=i\partial \bar{\partial}\psi_i$, where $\psi_i=\widetilde{max}(\psi'+C_i,\psi'')$ is the regularized max function (see, for instance, \cite[Theorem I.5.18]{D3}) and $C_i\nearrow \infty$. One may take a regularized max function $\psi_i$  to be $\psi_i=\mbox{max}\{\psi'+C_i,\psi''\}$ if $|(\psi'+C_i)-\psi''|>\epsilon$ for some $\epsilon>0$. Thus, $\omega_i=\omega'$ on $U_i:=\{\psi_i'+C_i-\psi''>\epsilon\}$. Since $\psi''\sim \log{|w|}$ near $D$, where $D=\{w=0\}$ locally, we have $\psi''\nearrow \infty$ near $D$. Thus $U_i$ is relatively compact, and $i\partial\bar{\partial}\psi=\omega''$ is globally defined. Therefore $\omega_i$ is a globally defined K\"ahler form which coincides with $\omega'$ on $U_i$.

\end{proof}

\section{Scattering diagrams and the Landau-Ginzburg mirrors of log Calabi-Yau surfaces}\label{sec:mainholotrop}

Continuing the setting in the previous section, let $(Y,D)$ be a Looijenga pair, and consider its toric blowup $(\widetilde{Y},\widetilde{D})$. 
Recall from Corollary \ref{cor:defdlfepsilon} that the special Lagrangian fibration on $\widetilde{Y} \setminus \widetilde{D} \cong Y \setminus D$ (Lemma \ref{Lagrangian fibration}) induces a consistent scattering diagram $\mathfrak{D}^{LF}_\epsilon$ on $B_\epsilon$, away from $\epsilon$-neighborhoods of the singular fibres of the fibration $\tilde{X}=\widetilde{Y}\setminus \widetilde{D} \to \R^2$.
Yet, we do not have a good control of rays emanating from these neighborhood, except that there are canonical ones $l_i$ induced by initial Maslov $0$ discs from singular fibres. In this section we show that $\mathfrak{D}^{LF}_\epsilon$ behaves nicely as $\epsilon \to 0$, and its limit agrees with $\mathfrak{D}^{GPS}$ in Section \ref{subsec:GPS}.

We also give an explicit calculations of the Landau-Ginzburg potential which counts Maslov 2 discs in $\widetilde{Y}$ and $Y$ bounding torus fibres under some non-negativity assumption (the same assumption of Section \ref{Sec: toric blowup}). This will be done by establishing the correspondence between tropical and holomorphic discs. Interestingly, the same technique can apply to some non-Fano examples to recover previous computation by Auroux \cite{A} and \cite{A5} using tropical geometry.

\subsection{$\mathfrak{D}^{LF}$ and $\mathfrak{D}^{GPS}$}\label{sec: LF/GPS}

We now prove that $\mathfrak{D}^{LF}_{\epsilon}|_{B_{\epsilon'}}$ coincides with $\mathfrak{D}^{LF}_{\epsilon'}$ for $\epsilon'<\epsilon$, which enables us to obtain a well-defined limiting scattering diagram $\mathfrak{D}^{LF}$ on $\mathbb{R}^2=\lim_{\epsilon\rightarrow 0}B_{\epsilon}$. Our strategy is to compare $\mathfrak{D}^{LF}_\epsilon$ with the known one $\mathfrak{D}^{GPS}$ given in Section \ref{subsec:GPS} modulo a certain energy level depending on $\epsilon$. Then we proceed inductively as the energies (symplectic areas) of contributing discs increase.
 
For later argument, we fix, once and for all,  $u_i\in \mathfrak{d}_i$ for each initial ray in $\mathfrak{D}^{GPS}$ close enough to infinity, and denote by $\mathfrak{d}_i'$ the part of the ray starting from $u_i$ with $\omega(\gamma_{\mathfrak{d}})$ decreasing from $u_i$ as it moves away to  infinity. 
We also choose an open neighborhood $K_i$ of for each $\mathfrak{d}_i'$ such that $K_i\cap K_j=\emptyset$ if $i\neq j$, and set $K:= \cup K_i$. See Figure \ref{fig:vivjk}.

\begin{figure}[h]
	\begin{center}
		\includegraphics[scale=0.45]{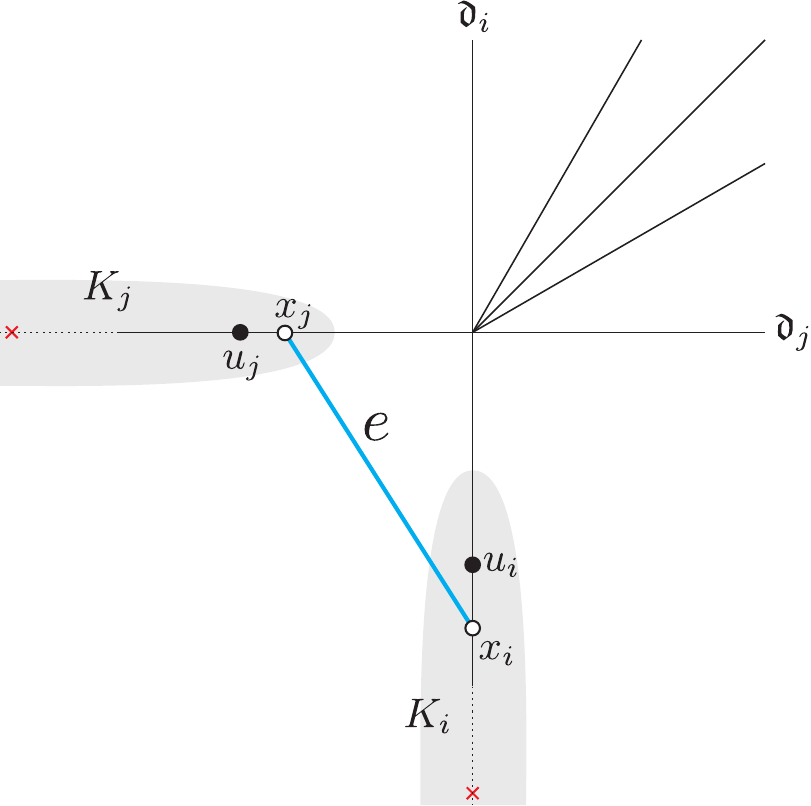}
		\caption{The areas of affine line segments between $K_i$ and $K_j$ are bounded below.}
		\label{fig:vivjk}
	\end{center}
\end{figure}

\begin{lemma}\label{minimal area}
	In the above setting, choose two arbitrary points $x_i\in K_i$ and $x_j\in K_j$ with $i\neq j$ such that the affine line segment connecting $x_i$ and $x_j$ has a rational slope (Figure \ref{fig:vivjk}). Then there exists $\hbar>0$, independent of the choice of $x_i$ and $x_j$, such that the symplectic area of the affine line segment is bounded below by $\hbar$. 
\end{lemma}	

Here, the symplectic area of the affine line segment means that of the holomorphic cylinder lying over the segment. Notice that the usage of the terminology is consistent with the way we define the symplectic area of a tropical disc, see Section \ref{subsec:tropgeomintro}.

\begin{proof}
	Notice that if the affine line segment connecting $x_i,x_j$ has a rational slope, there exists a holomorphic cylinder with boundaries on $\tilde{L}_{x_i}$ and $\tilde{L}_{x_j}$, whose symplectic area equals that of the line segment by definition. Such holomorphic cylinders can be constructed as follows: identify the maximal torus of the toric surface as $(\mathbb{C}^*)^2\cong \mathbb{R}^2\times T^2$ via the polar decomposition and $J$ is the standard complex structure on $(\mathbb{C}^*)^2$. Fix a constant section $s$ of the torus fibration. For each $x$ on the line segment connecting $x_i,x_j$ and $v\in T_x\mathbb{R}^2$ be a tangent vector of the line segment at $x$. There exists a unique $S^1\subseteq \mbox{Log}^{-1}(x)\cong T^2$ such that $\mbox{Log}_*(JTS^1)\in \mathbb{R}v$ passing through $s(x)$ when $v$ is rational. Then union of the circles above the line segments form a holomorphic cylinder with boundaries on $L_{x_i},L_{x_j}$. Since the distance between $K_i,K_j$ is positive, the area of such holomorphic cylinder is bounded below by a positive number from \cite[Proposition 4.4.1]{S6}.
\end{proof}	

The key step in the induction on energy is the following statement, which roughly tells us that any disc traveling through the complement of $B_\epsilon$ should have a large enough energy unless it is one of $l_i$'s.

\begin{lemma} \label{no initial} Given $\epsilon>0$, let $\mathfrak{D}^{LF}_\epsilon$ denote the scattering diagram  constructed from the admissible Lagrangian fibration over $B_{\epsilon}$. There exists $\lambda=\lambda(\epsilon)>0$ and a convex (bounded) region  $B_{\lambda,\epsilon}\subseteq B_\epsilon$ such that a ray entering  $ B_{\lambda,\epsilon}$ other than $l_i$'s (induced by initial Maslov 0 discs) has symplectic area at least $\lambda$. Moreover, we have $\lim_{\epsilon\rightarrow 0}\lambda=\infty$ and $\lim_{\epsilon\rightarrow 0}B_{\lambda,\epsilon}=M_{\mathbb{R}}$.  

\end{lemma}	

For the proof, we will use the following sub-lemma, which can be viewed as a tropical analogue of Gromov compactness theorem. 

	 \begin{lemma}\label{lower bound vertices}
	 	 There exists a sequence $\{\lambda_n\}$ of real numbers with $\lim_{n \to \infty} \lambda_n = \infty$ such that for each $n\in \mathbb{Z}_{> 0}$, $\lambda_n$ gives a lower bound for the symplectic areas of tropical discs with $n$ non-root vertices whose unbounded edges are contained in the rays of $\mathfrak{D}^{GPS}_{in}$ and whose edge adjacent to the root is contracted.
		 \end{lemma}
		 
     \begin{proof}
%
%
%
%
For each tropical disc with only one vertex, i.e. a sub-ray of an initial ray in $\mathfrak{D}^{GPS}_{in}$, if the root is away from the corresponding $K_i$, then its symplectic area is bounded below by a positive number $\lambda_0>\min\{\hbar, \lambda_0\}>0$.     	
Since there are only finitely many initial rays and hence finitely many intersection points among the initial rays, every tropical disc with exactly one non-root vertex has a lower bound for symplectic area, say $\lambda_1>0$. In particular, any tropical disc with exactly one non-root vertex has its symplectic area larger than $2\lambda_0>2\min\{\hbar, \lambda_0\}$, because it contains two unbounded edges with adjacent vertices outside any of $K_i$. Suppose that now we are given a tropical disc $(h,\mathcal{T},w)$ with $k+1$ non-root vertices with symplectic area 
$\lambda$. Deleting the edge adjacent to the root induces sub-tropical discs $(h_i,\mathcal{T}_i,w_i)$ with $k_i$ non-root vertices such that $\sum_i k_i=k$. Denote the common root of $(h_i,\mathcal{T}_i,w_i)$ by $u_1$. There are two cases, $u_1\notin K_{i_0}$ for any $i_0$ or $u_1\in K_{i_0}$ for some $i_0$. In the first case, we have $\lambda\geq \sum_i \lambda_{k_i}$. In the second case, $\lambda\geq \sum'_i \lambda_{k_i}$, where $\sum'$ is the sum omitting the unique sub-tropical disc corresponding to the unbounded edge in $K_{i_0}$. We still have $\sum'_i k_i=k$ in this case, so $\lambda>\min\{\hbar, \lambda_1\}$. Now delete $u_1$ and the edge adjacent to it. We are left with sub-tropical discs $(h'_i,\mathcal{T}'_i,w'_i)$ with $k'_i$ non-root vertices such that $\sum_i k'_i=k-1$. Denote by $u_2$ the vertex adjacent to the second deleted edge. If $u_2$ does not fall in any of $K_i$, then $\lambda>\sum_i \lambda_{k'_i}$. Otherwise, the second deleted edge connects two different $K_i$'s, and hence has symplectic area at least $\hbar$ from Lemma \ref{minimal area} (see Figure \ref{fig:indtropre}). In particular, we have $\lambda>\lambda_{n-2}+\hbar$. To sum up, we have 
  \begin{align}
  	\lambda_{k+1}> \min \{ \lambda_{k-1}+\hbar, \sum_{\sum_i k_i=k-1} \lambda_{k_i}  \},
  \end{align} where the summation has at least two terms. Then by induction, it is easy to show that $\lambda_k\geq (\frac{k}{3}+1)\min\{\lambda_0,\hbar\}$. 
	  
	  \begin{figure}[h]
	\begin{center}
		\includegraphics[scale=0.35]{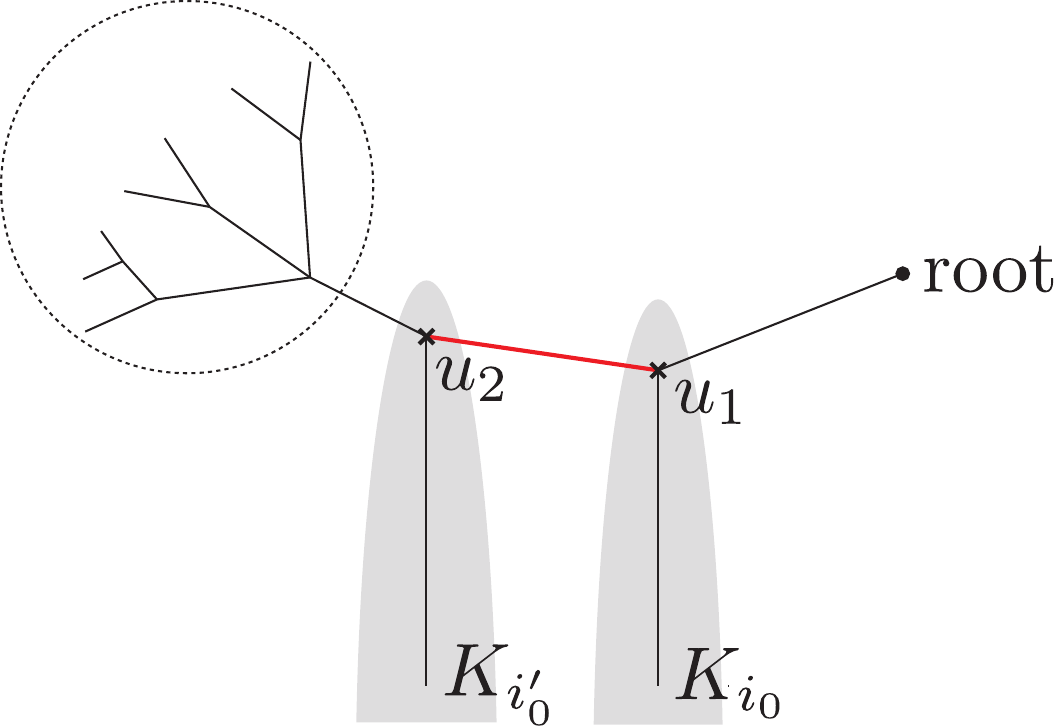}
		\caption{sub-tropical discs of $\mathcal{T}$}
		\label{fig:indtropre}
	\end{center}
\end{figure}

     \end{proof} 

\begin{proof}[Proof of Lemma \ref{no initial}]
	Assume that there exists a ray $\mathfrak{d}$ of $\mathfrak{D}^{LF}_{\epsilon}$  passing through $u\in \partial B_{\epsilon}$. By definition of the scattering diagram $\mathfrak{D}^{LF}_{\epsilon}$, there exists a holomorphic disc $\tilde{f}:(D^2,\partial D^2)\rightarrow (\tilde{X},\tilde{L}_u)$ of Maslov $0$. Lemma \ref{weak correspondence} produces a tropical disc with stop at $u$ whose edge adjacent to $u$ is parallel to $\mathfrak{d}$ and whose unbounded edges are contained inside rays in $\mathfrak{D}^{GPS}_{in}$ (i.e. initial rays of $\mathfrak{D}^{GPS}$, see Section \ref{subsec:GPS}). We claim that there are only finitely many such tropical discs (up to elongation of the edge adjacent to the root) if we further require their symplectic areas to be less than a fixed number $\lambda>0$.
	 Once the claim is shown, one can simply take $B_{\lambda,\epsilon}$ to be a convex region in $B_{\epsilon}$ containing all the vertices of such tropical discs. The last part of the lemma holds because $\lim_{\epsilon\rightarrow 0}B_{\epsilon}=M_{\mathbb{R}}$.

    Lemma \ref{lower bound vertices} tells us that for any tropical disc $(h,\mathcal{T},w)$ with symplectic area less than a fixed constant $\lambda$, the number of non-root vertices 
    is bounded above by some $N \in \mathbb{Z}_{> 0}$. 
Hence,  the claim will follow if we can prove that for each $n \in \mathbb{Z}_{> 0}$, there are only finitely many tropical discs with symplectic areas less than $\lambda$ which has precisely $n$ non-root vertices. We use induction on $n$. 
  
Let us first consider the case $n=1$. Since all the unbounded edges are contained in finitely many initial rays, there are only finitely many intersections among these initial rays. Initial rays cut out at their common intersection points give tropical discs without non-root vertices (whose roots lie at the intersection points), and their symplectic areas are bounded below, say, by $\lambda' >0$, since there are only finitely many such. Observe that each  unbounded edge of a tropical disc with a single non-root vertex is precisely such a portion of an initial ray. These unbounded edges may additionally carry nontrivial weights (the edge adjacent to the root is automatically determined by these data, up to elongation), but the symplectic area bound $\lambda$ on the whole disc forces the sum of the weights to be bounded above, since each unbounded edge contributes at least $\lambda'$ to the total area. Therefore, there are only finitely many possible combinations between weights, as is desired for $n=1$.
    
    Suppose we know that there are only finitely many tropical discs with at most $k$ non-root vertices whose symplectic areas are less than $\lambda$, and consider tropical discs with $k+1$ non-root vertices.
    Removing the edge adjacent to the root, one obtains sub-tropical discs $(h_i,\mathcal{T}_i,w_i)$ with $k_i$ non-root vertices such that $\sum_i k_i = k$. 
    If there are at least two $(h_i,\mathcal{T}_i,w_i)$ which do not have any non-root vertices, then their intersection can only happen outside the region $K$ chosen before Lemma \ref{minimal area}. Hence their symplectic areas are bounded below by some constant $\hbar'$. In particular, the weight of those edges are bounded by $\lambda/\hbar'$ and thus there are finitely many of such tropical discs.      
   
    Assume now that there is only one $(h_i,\mathcal{T}_i,w_i)$ without non-root vertices among the sub-tropical discs. Then the other sub-tropical discs have at most $k$ non-root vertices and symplectic areas less than $\lambda$. From induction hypothesis, there are only finitely many such tropical discs. Extending the edge adjacent to the root beyond the root if necessary, (union of) these tropical discs intersect the initial rays at finitely many points. We see that $(h_i,\mathcal{T}_i,1)$ should reach one of these points to form the original disc with $k+1$ non-root vertices.
    Therefore, there exists a lower bound $\hbar''$ for the symplectic area of $(h_i,\mathcal{T}_i,1)$, independent of $k$. Again the weight $w_i$ for the edge is bounded and thus there are finitely many tropical discs in this case as well, which finishes the proof. 
\end{proof}

We finally prove the main theorem of the paper, which provides a recipe to retrieve the scattering diagram of Gross-Pandharipande-Siebert from Lagrangian Floer theory. Recall that the fibration over $B'=B_{\epsilon}$ is obtained by removing neighborhoods of singular fibre from the fibration constructed in Lemma \ref{Lagrangian fibration}. Since $B_{\epsilon}$ is simply-connected by construction, the fibration has no monodromy over $B_{\epsilon}$, and hence, the relative classes  are well-defined up to parallel transport along $B_\epsilon$.

\begin{theorem} \label{correspondence: scattering diagram}
	Given $\epsilon>0$, the scattering diagram $\mathfrak{D}^{LF}_{\epsilon}$ defined in Corollary \ref{cor:defdlfepsilon} 
	coincides with $\mathfrak{D}^{GPS}$ on $B_{\lambda,\epsilon}$ for $\epsilon \ll 1$ up to contributions from discs with area$>\lambda=\lambda(\epsilon)$. In particular, one can recover $\mathfrak{D}^{GPS}$ from Lagrangian Floer theory by 
	 \begin{align*}
	    \lim_{\substack{  \epsilon \to 0}} \mathfrak{D}^{LF}_{\epsilon}(\mbox{mod }T^{\lambda})|_{B_{\lambda,\epsilon}}=\mathfrak{D}^{GPS}. 
	 \end{align*}
\end{theorem}

\begin{proof} Let $\gamma_i$ be the relative homology class of the disc associated with the blowup point $q_i$ in Lemma \ref{contribution of initial discs}, which can be represented as a holomorphic disc with boundary in a torus fibre over $B_{\epsilon}$. It is easy to see that $\{[E_i]\}$ (and $\{\gamma_i\}$) forms a linearly independent subset in the cohomology.
	From Lemma \ref{contribution of initial discs} and the identification $t_{ij}z^{m_{\sigma(i)}}\leftrightarrow T^{\omega(\gamma_i)}z^{\partial \gamma_i}$, each blowup point $q_i$ contributes to the same initial ray to the scattering diagram of $\mathfrak{D}^{LF}_{\epsilon}$ and $\mathfrak{D}^{GPS}$. From Lemma \ref{no initial}, the above are all the initial rays in the scattering diagrams $\mathfrak{D}^{LF}_{\epsilon} (\mbox{mod }T^{\lambda})|_{B_{\lambda,\epsilon}}$ and $\mathfrak{D}^{GPS}|_{B_{\lambda,\epsilon}}$. Thus, we have 
	 \begin{align*}
	   \mathfrak{D}^{LF}_{\epsilon} (\mbox{mod }T^{\lambda})|_{B_{\lambda,\epsilon}}=\mathfrak{D}^{GPS}|_{B_{\lambda,\epsilon}}
	 \end{align*} for every $\lambda,\epsilon>0$
	from Theorem \ref{consistency}, and this completes the proof since $B_{\lambda,\epsilon}\nearrow \mathbb{R}^2$.  
\end{proof}

As a byproduct, we can verify the folklore conjecture
\begin{center}
{\it``the open Gromov-Witten invariants $=$ the log Gromov-Witten invariants."}
\end{center}
in our geometric setup, by comparing the coefficients for the wall functions of the two scattering diagrams in Theorem \ref{correspondence: scattering diagram}.
Our argument passing to tropical geometry has a great advantage that we can avoid the difficulty of comparing the virtual fundamental classes in the algebraic and symplectic version of log Gromov-Witten invariants.

We can identify homology classes of the curves that the two invariants count in the following way.
To a given relative class $\gamma$ of a disc in $\tilde{X}$, one can associate a homology class $\bar{\gamma}\in H_2(\widetilde{Y},\mathbb{Z})$ as follows. Fix $u_0\in B_{\epsilon}$ and choose a representative $C_{\gamma}$ of $\gamma\in H_2(\tilde{X},\tilde{L}_{u_0})$. Let $l$ be an affine ray starting from $u_0$ parallel to the direction defined by $\partial \gamma$ (it moves away from the disc $\gamma$). Let $\mbox{Cyc}_{\partial \gamma}$ be a trivial $S^1$-fibration over $l$ with fibres in the class $\partial \gamma$. Since $[\mbox{Cyc}_{\partial \gamma} \cap L_u]=[C_{\gamma}\cap L_u]\in H_1(L_u,\mathbb{Z})$, there exists $C'$ in $L_u$ homologous to $C_{\gamma}\cup \mbox{Cyc}_{\partial \gamma}$  
where both are thought of as non-compact $2$-chains without boundary in $\tilde{X}$. The explicitly chosen such a $2$-chain has a one-point compactification in $\widetilde{Y}$ as a $2$-cycle, whose homology class will be denoted by $\bar{\gamma}$. 

\begin{corollary}\label{open-closed} Given a relative class $\gamma$, there exists a proper open set $U_{\gamma}\subseteq B_\epsilon$ (for any small enough $\epsilon$) such that 
 \begin{enumerate}
 	\item if $\tilde{\Omega}(\gamma;u)=0$ for all $u\in B_\epsilon \setminus U_{\gamma}$, then $N_{\bar{\gamma}}=0$. 
 	\item if $u\in B_\epsilon \setminus U_{\gamma}$ and $\tilde{\Omega}(\gamma;u)\neq 0$, then $\tilde{\Omega}(\gamma;u)=N_{\bar{\gamma}}$.
 \end{enumerate}	
\end{corollary}

Recall from Section \ref{subsubsec:lfwallopengw} and Section \ref{subsec:GPS} that $\tilde{\Omega}(\gamma;u)$ and $N_{\bar{\gamma}}$ are the open Gromov-Witten invariant and the log Gromov-Witten invariant, respectively.

\begin{proof}
	The first part of the corollary is a direct consequence of Theorem \ref{correspondence: scattering diagram} combined with the discussion in Section \ref{subsec:GPS}.

	To see the second part, notice that every ray in the scattering diagram has a corresponding tropical curve. Thus, there are finitely many set of non-negative integers $\mathbf{P}=(p_{ij})$ satisfying \eqref{?} and $\bar{\gamma}$ coincides with the image of $\beta_{\mathbf{P}}$. In particular, there are only finitely rays in
	rays in $\mathfrak{D}^{GPS}$ which involves $N_{\bar{\gamma}}$	by Lemma \ref{lower bound vertices}.
 Projecting down $\gamma$ by $\pi : \widetilde{Y} \to \bar{Y}$, the intersection pairing of $\pi_*(\gamma)$ with each boundary component $\bar{D}_i$ is determined. All the intersections of $\pi_*(\gamma)$ with $\bar{D}$ occur at $\{q_1,\cdots, q_m\}$, since $\mu(\gamma)=0$. 
 
 Moreover, the multiplicities and the positions of unbounded edges of the tropical discs of relative class $\gamma$ are determined topologically by $\bar{\gamma}$. Indeed, the multiplicity $p_{ij}$ of the unbounded edge determined by $q_{ij}$ can be calculated via the relation 
	 \begin{align*}
	    \bar{\gamma}=\pi^* \pi_*(\bar{\gamma})-\sum_{i,j} p_{ij} E_{ij}.
	 \end{align*}

Fix  positions and multiplicities of all the edges of a tropical curve (say, in class $\bar{\gamma}$), but one unbounded edge. By balancing condition, the multiplicity and the direction of the last unbounded edge is automatically determined. Actually the position of the last unbounded edge is also determined by \cite[Proposition 6.12]{M8} (see \cite[Theorem 3.3.10]{K3}, also).
\footnote{This is a tropical analogue of \cite[Lemma 4.1]{GPS}.} 

Suppose the edge adjacent to the root of a tropical disc of relative class $\gamma$ lies in some affine line $l$. 
We essentially showed in the proof of Lemma \ref{no initial} that there are finitely many tropical discs (up to elongation of the edge adjacent to the root) in the class $\gamma$. Therefore $\tilde{\Omega}(\gamma;u')$ is constant along $l$ near infinity towards the direction along which $\omega(\gamma)\gg 0$. 
It suffices to take $U_{\gamma}$ such that $U_{\gamma}\cap l$ is contained the part of $l$ where $\tilde{\Omega}(\gamma)$ is constant, then the second part of the corollary holds from Theorem \ref{correspondence: scattering diagram}. 
\end{proof}	
\begin{remark}
	The unique rational curve in Lemma \ref{doubling} corresponding to relative $\gamma$ is different from those $\mathbb{A}^1$-curves contributing to $N_{\bar{\gamma}}$ even topologically. We refer readers to Example 5.3 \cite{HLZ}.   
\end{remark}

\begin{remark}
	 Although the toric model $(\bar{Y},\bar{D})$ of a Looijenga pair $(Y,D)$ is always a smooth toric surface, theorems in Section \ref{sec: LF/GPS} generalize to the case when $\bar{Y}$ only have orbifold singularity at the corners. This is because $X$ is still a smooth symplectic manifold and the local model for the non-toric blowup remains the same. 
\end{remark}

\subsection{The Landau-Ginzburg mirror of the compactification $Y$}\label{subsec:LGmirrorY}
For the rest of the section, we will assume that $\widetilde{Y}$ is semi-Fano and, consider an admissible SYZ fibre $\tilde{L}_u$ in $\widetilde{Y}$ (i.e., a pull-back of a moment map torus $L_u$ in $Y$). We will show that the superpotential of $\tilde{L}_u$ can be calculated tropically. More specifically, we have the following theorem. 
\begin{theorem}\label{thm: main}
	Assume that $\widetilde{D}$ contains no negative Chern number spheres. Then $n_\beta (u):=n^{\widetilde{Y}}_{\beta}(\tilde{L}_u)$ coincides with the weighted count of the broken lines with respect to the scattering diagram $\mathfrak{D}^{LF}$ in Theorem \ref{LF scattering diagram}. 
\end{theorem}
Here, $n^{\widetilde{Y}}_{\beta}(\tilde{L}_u)$ is the algebraic count of Maslov index 2 discs in class $\beta$. See Section \ref{subsec:LFT1}. We first need a few preliminary lemmas.

  \begin{lemma}\label{initial discs MI=2}
  	  Assume that $l$ is an affine ray from $u$ in the direction of $\partial \beta$ for some $\beta\in H_2(\widetilde{Y},\tilde{L}_u)$ such that $\tilde{\omega}(\beta)$ is decreasing along the ray. There exists a constant $\hbar>0$ such that if
  	  \begin{enumerate}
  	  	\item $n_{\beta}(u)$ is non-zero and constant along the ray, or 
  	  	\item $\omega(\beta)<\hbar$ and $n_{\beta}(u)\neq 0$,
  	  \end{enumerate}
  	    then $\beta=\beta_i+\alpha$ for some $\alpha\in H_2(\widetilde{Y},\mathbb{Z})$ and some $i$. 
  \end{lemma}

\begin{proof} First we assume that $n_{\beta}$ is non-zero along $l$, then there exists a stable holomorphic disc of relative class $\beta$ with boundary on $L_u$ for for each $u\in l$.
   Applying Lemma \ref{weak correspondence} to the disc component of such a stable disc, we obtain a tropical disc with stop at $u$ whose edge adjacent to $u$ is parallel to $l$. From the hypothesis, we may assume that $\tilde{\omega} (\beta)<\lambda$ for some constant $\lambda>0$. Recall that any broken lines except those without bending (see Example \ref{initial broken line}), the corresponding Maslov index two tropical disc will contain a Maslov index zero tropical sub-disc, which has area less than $\omega(\beta)$. Since Lemma \ref{no initial} tells us that there are only finitely many tropical discs on Maslov index zero in $\widetilde{Y}$ with symplectic area less than $\lambda$, there are only finitely many intersections of the tropical discs of Maslov index zero with symplectic area less than $\omega(\beta)$ with $l$.   
   Therefore, if $u \in l$ is close enough to infinity, there exists no tropical discs with stop on $l$ between $u$ and the infinity while the symplectic area is less than $\lambda$ other than the (multiple) of the initial disc. In particular, $l$ is defined by the vanishing cycle of a toric boundary divisor. Since $\beta$ is of Maslov index two, the corresponding tropical disc has the same relative class as the basic disc. The second case follows from the same line of the argument of the first case.

\end{proof} 
	We next show a weak version of tropical/holomorphic correspondence theorem. 
\begin{lemma}	\label{split attractor flow}
	 If $n^{\widetilde{Y}}_{\beta}(u)\neq 0$ for a generic point $u$, there exists a broken line $\mathfrak{b}$ with end at $u$ such that $[\mathfrak{b}]=\beta$. 
\end{lemma}	 
\begin{proof}
		For a generic $u$,  $n^{\widetilde{Y}}_{\beta}(u)$ is well-defined, i.e. $u\notin \mbox{Supp}(\mathfrak{D}^{LF})(\mbox{mod } T^\lambda)$ where $\lambda>0$ is a constant, since there are finitely many rays in $\mbox{Supp}(\mathfrak{D}^{LF})(\mbox{mod } T^\lambda)$. We will take $\lambda<\int_{\beta_u}\tilde{\omega}+\lambda_1$, where $\lambda_1$ is the lower bound of tropical discs of Maslov index zero (see Lemma \ref{lower bound vertices}).
	Consider an affine ray $l$ emanating from $u$ in the direction of $\partial \beta$ such that $\tilde{\omega} (\beta)$ is decreasing along the ray. By genericity of the position of $u$, we may assume that the affine ray intersects all the rays in $\mbox{Supp}(\mathfrak{D}^{LF})(\mbox{mod } T^\lambda)$ transversally. 
Traveling along the ray, the following two scenarios can occur.	
	\begin{enumerate}
	\item[(i)] $n^{\widetilde{Y}}_{\beta}$ is constant along the whole affine ray, then the lemma directly follows from Lemma \ref{initial discs MI=2}. 
	\item[(ii)] $n^{\widetilde{Y}}_{\beta}$ jumps at some point $u_1\in \mbox{Supp}(\mathfrak{D}^{LF}) (\mbox{mod } T^{\lambda})$, which means that a certain bubbling phenomenon occurs at $\tilde{L}_{u_1}$. Notice that by Gromov compactness theorem, there are only finitely many rays in the scattering diagram $\mathfrak{D}^{LF}(\mbox{mod }T^{\lambda})$. Let $u'$ on the affine ray close to $u_1$ such that $n^{\widetilde{Y}}_{\beta}(u')\neq n^{\widetilde{Y}}_{\beta}(u)$ and by the assumption of the generic position of $u$ we can assume that exactly one ray $l_{\mathfrak{d}_1}$ pass through $u_1$ with $\omega(\gamma_{\mathfrak{d}_i})<\lambda$. 
	 Apply Lemma \ref{wall-crossing} (1) to the path connecting $u',u$ along $l$
	  implies that 
	  \begin{align*}
	     W(u)=\theta_{\mathfrak{d}_1}W(u') \mbox{ mod }T^{\lambda}.
	  \end{align*} Therefore, there exists $\beta'\in H_2(\widetilde{Y},L_{u'})$ with $n_{\beta'}(u')\neq 0$ such that the coefficient of $z^{\beta}$ in $\theta_{\mathfrak{d}_1}(z^{\beta'})$ is non-zero.
	 
	 We then replace $(u,\beta)$ by $(u',\beta')$, and repeat the above procedure. Notice that $\tilde{\omega}(\beta_u)-\tilde{\omega}(\beta'_{u'})\geq \lambda_1 $, and the process ends in finite process, say at $u_1,\cdots,u_n$ and then it reduces to the first case. 

    Finally, we define the broken line $\mathfrak{b}:(-\infty,0]\rightarrow N_{\mathbb{R}}$ with $\mathfrak{b}(0)=u$ such that $\mathfrak{b}(t_i)=u_{n-i}$ for $-\infty=t_0<t_1<\cdots< t_n=0$ (here we set $u=u_0$). Then (1) in Definition \ref{broken line} follows from the choice of the affine rays above, (2) follows from Lemma \ref{initial discs MI=2}, (3) can be achieved by genericity of the position of $u$, and (4) follows from Lemma \ref{wall-crossing}.
	\end{enumerate}
\end{proof}

\begin{proof}[Proof of Theorem \ref{thm: main}]
 Fix $\beta\in H_2(Y,L_u)$. Notice that each bending of the broken line will increase its corresponding area by at least $\lambda_1>0$ (see the notation in Lemma \ref{lower bound vertices}). Thus a broken line $\mathfrak{b}$ such that $[\mathfrak{b}]=\beta$ can bend at most $\frac{\tilde{\omega}(\beta)}{\lambda_1}$ times. Together with Lemma \ref{lower bound vertices} implies that there are finitely many places the broken lines can bend. Thus, there are finitely many broken lines $\mathfrak{b}$ such that $[\mathfrak{b}]=\beta$ by Lemma \ref{weak correspondence}. Let $n(\beta;u)$ be the maximal number of edges of a broken line that can represent $\beta$. We will prove the theorem by induction on $n(\beta;u)$. The theorem reduces to Lemma \ref{initial discs MI=2} when $n(\beta;u)=1$. 
 
Assume that the theorem is true, i.e., $n_{\beta}(u)=n^{trop}_{\beta}(u)$, for all pairs of $(\beta;u)$ if $n(\beta;u)\leq n$ where $n^{trop}_{\beta} (u)$ is the number of broken lines in class $\beta$ with end on $u$. Given a pair $(\beta;u)$ such $n(\beta;u)=n+1$. There exists an affine ray emanating from $u$ corresponding to $\partial \beta$ such that $\tilde{\omega} (\beta)$ is decreasing along the affine ray. 

Suppose that $n_{\beta}$ jumps at $u_1$. Using the same argument as in the proof of  Lemma \ref{split attractor flow}, one can find $\beta_1\in H_2(Y,L_{u_1})$ with $n(\beta_1 ;u_1)<n(\beta;u)$. By the induction hypothesis, we have $n_{\beta_1}(u_1)=n^{trop}_{\beta_1}(u_1)$, and we can derive $n_{\beta}(u)$ and $n^{trop}_{\beta}$ from their compatibility with the wall-crossing formula.
The two wall-crossing formulas are identified in Theorem \ref{correspondence: scattering diagram}, and this finishes the proof. 
\end{proof}

We now consider the effect of the toric blowup $\widetilde{Y} \to Y$ on the counting $n_\beta$. Note that $Y$ is at worst semi-Fano, since contracting exceptional divisors in $\widetilde{D}$ do not decrease the self-intersections of the other divisors.

\begin{lemma}\label{blow down}
	Let $\pi':(\widetilde{Y},\widetilde{D})\rightarrow (Y,D)$ be a simple toric blowup. Let $\tilde{L}\subseteq \tilde{X}$ be an admissible SYZ fibre in the toric model of $\widetilde{Y}$ and $L=\pi'(\tilde{L})$. If the superpotential of $\tilde{L}$ is $\sum_{\beta\in H_2(\widetilde{Y},\tilde{L}):\mu(\beta)=2}n_{\beta}^{\widetilde{Y}}T^{\tilde{\omega}(\beta)}z^{\partial \beta}$, then the superpotential of $L$ is given by
	  \begin{align*}
	     \sum_{\substack{\beta\in H_2(\widetilde{Y},\tilde{L}) \\ \beta\cdot E=0, \, \mu(\beta)=2}} n_{\beta}^{\widetilde{Y}}T^{\omega(\pi'_*\beta)}z^{\partial \beta}.
	  \end{align*}
\end{lemma}	
\begin{proof}
	 Let $u:(\Sigma,\partial \Sigma)\rightarrow (Y,L)$ be a stable disc. From the semi-Fano condition of $Y$, the Lagrangian $L$ does not bound any holomorphic discs of negative Maslov index. If it is	 
	  of Maslov index 0 or 2 of relative class $\beta\in H_2(Y,L)$, then its image does not hit the corner of $D$. Therefore, its proper transform $\tilde{u}:(\Sigma,\partial \Sigma)\rightarrow (\widetilde{Y},\tilde{L})$ is a stable disc of relative class $\tilde{\beta}\in H_2(\widetilde{Y},\tilde{L})$ avoiding the exceptional divisor $E$ of $\widetilde{Y}\rightarrow Y$ and the corners of $\widetilde{D}$. In particular, the proper transform preserves Maslov indices, in this case. 
	  
	  Conversely, if $\tilde{u}:(\Sigma,\partial \Sigma)\rightarrow (\widetilde{Y},\tilde{L})$ is a stable disc of Maslov index 2. Then the projection $u=\pi'\circ \tilde{u}:(\Sigma,\partial \Sigma)\rightarrow (Y,L)$ is a stable disc. Notice that for every non-exceptional curve $\tilde{C}\subseteq \widetilde{Y}$, one has
	  \begin{align*}
	     -K_{\widetilde{Y}}\cdot[\tilde{C}]=-K_{Y}\cdot [\pi'(\tilde{C})]-[\tilde{C}]\cdot [E]\leq -K_{Y}\cdot [\pi'(\tilde{C})]. 
	  \end{align*} Therefore, the Maslov index of projection is bigger than or equal to that of the original disc $\tilde{u}$, and the equality holds if and only if the image of $\tilde{u}$ does not contains a component intersecting the exceptional divisor. 
	  
	  Consequently, the stable discs of $(Y,L)$ of Maslov index 2 are in one-to-one correspondence with the stable discs of $(\widetilde{Y},\tilde{L})$ of Maslov index 2 which do not intersect the exceptional divisor. In other words, all these discs are contained in biholomorphic open subsets of $\widetilde{Y}$ and $Y$. In particular, the corresponding counts $n_\beta^{\widetilde{Y}}$ and $n_{\pi'_\ast \beta}^{Y}$  
	  are the same. 
	   
\end{proof}	
Together with Theorem \ref{thm: main}, we have Theorem \ref{thm:winintro} in the introduction. 
\begin{remark}
	With the admissible special Lagrangian fibration in $Y\setminus D$, Tu \cite{T4} constructed the family Floer mirror of the total space of the admissible special Lagrangian fibration.
	The superpotential for these admissible fibres satisfying the wall-crossing formula and thus is a well-defined function on the mirror as explained by Yuan \cite{Y5}. 
\end{remark}

 Motivated by Gross-Hacking-Keel \cite[Construction 2.27]{GHK}, we may define the theta functions $\theta_{[D_i]}(u)$ by
  \begin{equation}\label{eqn:thetadinbeta}
      \theta_{[D_i]}(u)=\sum_{\beta} n^{Y}_{\beta}(u)T^{\omega(\beta)}z^{\partial\beta},
  \end{equation} 
  where $\beta$ runs over all relative classes in $H_2(Y,L_u)$ with $\mu(\beta)=2$ and the (unique) disc components of their holomorphic representatives intersect $D_i$ once. Then Theorem \ref{thm: main} implies that \eqref{eqn:thetadinbeta} coincides with the original definition in \cite{GHK}. 
  
  The superpotential computed here is expected to be the mirror of $Y$. In fact, we can prove that the quantum period theorem holds under the additional assumption that $Y$ is Fano. (This does not directly lead to the full close-string mirror symmetry statement which concerns all Gromov-Witten invariants.) Let $\mathcal{M}_{0,1}(Y,A)$ be the moduli space of stable maps of genus zero with a marked point and image have class $A$ for $A\in H_2(Y,\mathbb{Z})$. Write $d_A:=A\cdot (-K_Y)$. The virtual dimension of the moduli space $\mathcal{M}_{0,1}(Y,A)$ is $d_A$, and it has a virtual fundamental class $[\mathcal{M}_{0,1}(Y,A)]^{vir}$. Let $ev:\mathcal{M}_{0,1}(Y,A)\rightarrow Y$ be the evaluation map, and $\psi$ the psi-class insertion at the marked point. Then the regularized quantum period of $Y$ is given by 
       $\hat{G}_Y:=\sum_{A\in H_2(Y,\mathbb{Z})} \langle \psi_{d_A-2}\, pt\rangle_{A} z^{A}$,
    where $\langle \psi_{d-2}pt\rangle_{A}$ is the descendant Gromov-Witten invariants 
    \begin{align*}
       \langle \psi_{d_A-2}pt\rangle_{A}=\int_{[\mathcal{M}_{0,1}(Y,A)]^{vir}}\psi^{d_A-2}ev^*{[pt]}.
    \end{align*} We also write $\langle \psi_{d-2}pt\rangle_d=\sum_{A:d_A=d} \langle \psi_{d_A-2}pt\rangle_{d_A}$.     
  The quantum period theorem is a direct consequence of \cite[Theorem 1.12]{M9} (see also \cite[Theorem 1.13]{TY3}) combined with Theorem \ref{thm: main}. 
  \begin{theorem}
  	 Assume that $Y$ is a Fano surface and 
  	 \begin{enumerate}
  	 	\item $Y$ is a non-toric blowup of a toric surface or
  	 	\item each component of the boundary divisor is nef and 
  	 \end{enumerate} then the constant term of $W(L)^d$ is simply $(d)!\langle \psi_{d-2}pt\rangle_d$, where $L$ is any admissible SYZ fibre in $Y$.
 \end{theorem}

\subsection{Further examples} 

Below, we use the new machinery developed in this paper to recover classical examples in earlier works of Auroux, and provide a tropical explanation of the calculations.

\subsubsection*{\textnormal{(i)} $(\mathbb{P}^2,D)$ with $D=C \cup L$} We first look into the example of $Y\cong\mathbb{P}^2$ with the anticanonical divisor given by $D=C \cup L$, where $C$ is a conic, and $L$ is a line $L$ in $Y$. The corresponding mirror potential was first computed in \cite[Proposition 5.5, Proposition 5.8]{A}.
On the other hand, $(Y,D)$ viewed as a Looijenga pair, the associated non-toric blowup $(\widetilde{Y},\widetilde{D})$ and its toric model $(\bar{Y},\bar{D})$ are computed in  \cite[Example 2.9]{HK}, which we now explain. 

 \begin{figure}[h]
	\begin{center}
		\includegraphics[scale=0.5]{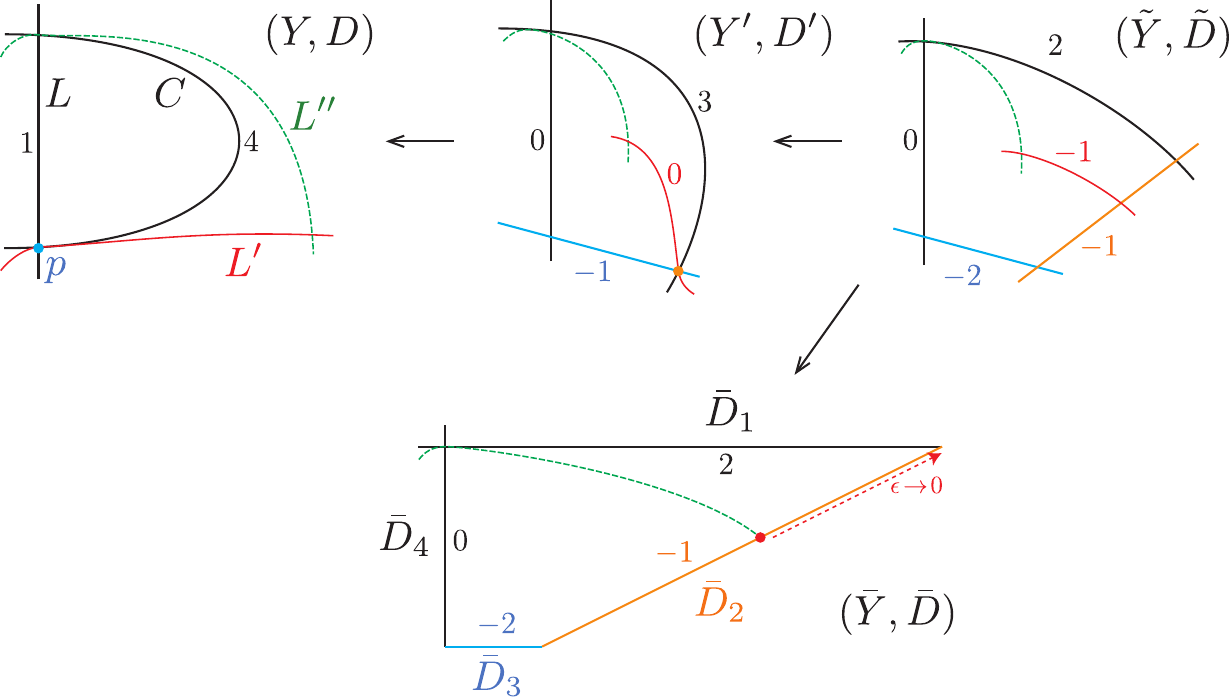}
		\caption{Toric model for (a toric blowup of) the pair $(\mathbb{P}^2, D)$}
		\label{fig:toricmodelp2c}
	\end{center}
\end{figure}

Let $p$ be one of the nodes in $D$, and $L'$ the tangent of $C$ at $p$. We take the blowup $Y'\cong \mathbb{F}_1\rightarrow Y$ of $Y$ at $p$, and write $D'$ for the preimage of $D$. Then $\widetilde{Y}$ is obtained as the blowup of $Y'$ at the intersection of the proper transform of $C$ and the exceptional divisor for $Y' \to Y$. We denote the preimage of $D'$ by $\widetilde{D}$. 
The proper transform of $L'$ is now a $(-1)$-curve in $\widetilde{Y}$, and its blow down gives rise to the Hirzebruch surface $\bar{Y}\cong \mathbb{F}_2$. The push-forward $\bar{D}$ of $\widetilde{D}$ is the toric boundary divisor (see Figure \ref{fig:toricmodelp2c}). One may choose the fan of $\bar{Y}$ whose $1$-cones are spanned by $v_1=(0,-1)$ ,$v_2=(-1,2)$, $v_3=(0,1)$ and $v_4=(1,0)$.	
 We will denote by $\bar{D}_i$ the component of the toric boundary divisor normal to $v_i$.

Given a moment fibre $\bar{L}$ of $\bar{Y}$, one has the basic disc class $\beta_i\in H_2(\bar{Y},\bar{L})$ for each $\bar{D}_i$, which intersects $\bar{D}_i$ exactly once. The superpotential $W_{\mathbb{F}_2}$ for $\mathbb{F}_2$, however, has more contributions than those from $\beta_i$'s. It is well-known by, for e.g., \cite{A, CL, FOOO9}
   \begin{align*}
      W_{\mathbb{F}_2}=t^{\beta_1}y^{-1}+t^{\beta_2}x^{-1}y^2+(1+t^{[\bar{D}_3]})t^{\beta_3}y+t^{\beta_4}x
   \end{align*} where we set $t^{\beta}:=T^{\omega(\beta)}$ in order to keep track of the relative classes easily. Notice the sphere bubble contributions as seen from the term $t^{[\bar{D}_3] + \beta_3}y$.
   
       \begin{figure}[h]
	\begin{center}
		\includegraphics[scale=0.45]{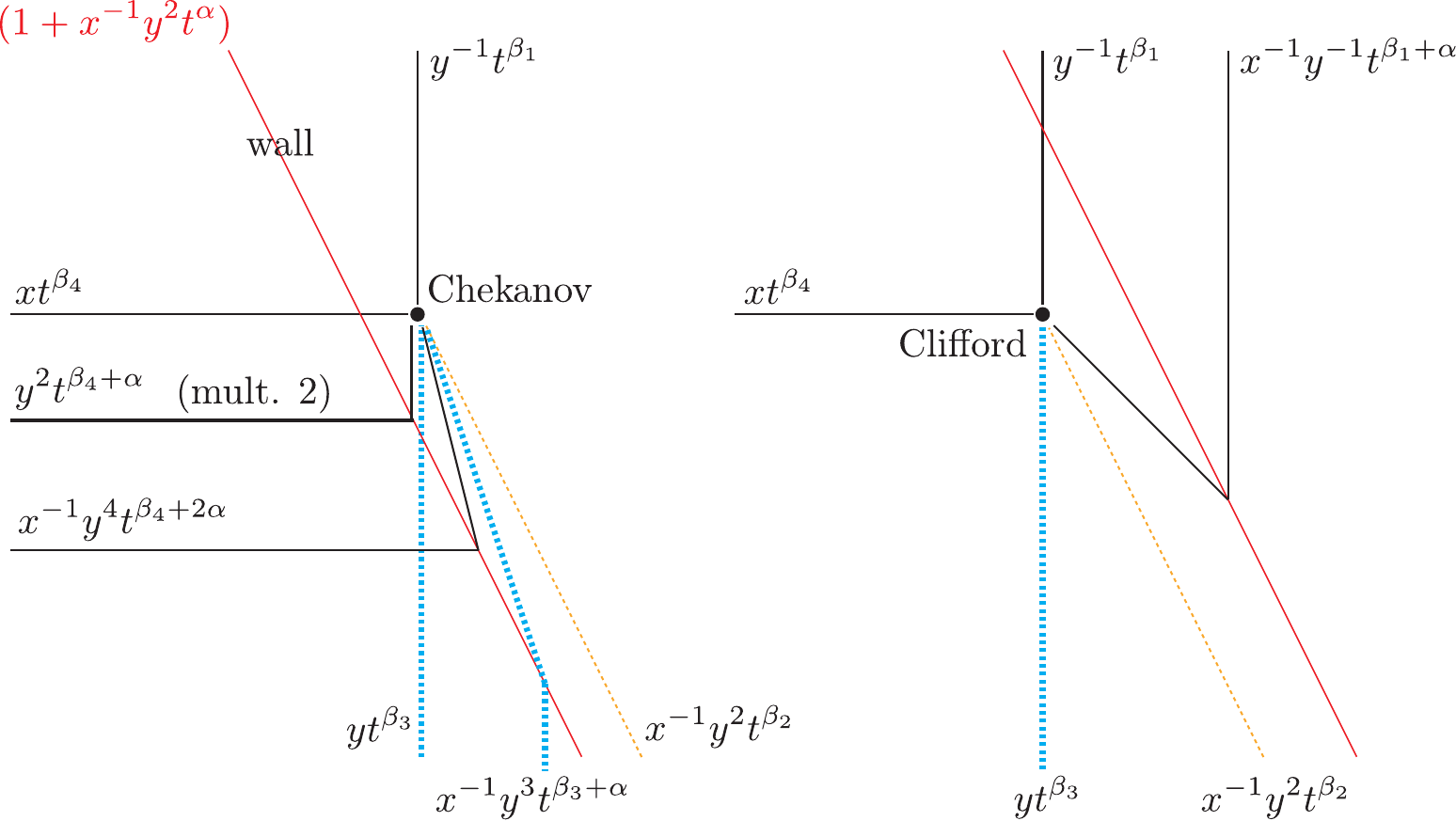}
				\caption{Broken line counts on the two chambers for the toric model of $(\mathbb{P}^2,D)$ in Figure \ref{fig:toricmodelp2c}.\textsuperscript{a} Dashed are the contributions that disappear after blowdowns to $Y$. The broken lines with infinity ends parallel to the negative $y$-axis survive until $Y'$. On $\widetilde{Y}$, one also has $t^{[\bar{D}_3]} y t^{\beta_3} $ and $t^{[\bar{D}_3]} x^{-1} y^3 t^{\beta_3}$ in addition to these, due to sphere-bubbles.}
  \small\textsuperscript{$^{\textnormal{a}}$We use $\beta_i$ to denote the proper transform of the relative class for simplicity.} 
				\label{fig:brlines}
	\end{center}
\end{figure}
    
   Applying Theorem \ref{correspondence: scattering diagram}, the blowup introduces a single ray (wall) for the scattering diagram $\mathfrak{D}^{LF}$ for $(\widetilde{Y},\widetilde{D})$. Let us denote by $\alpha$ the Maslov zero disc class corresponding to the ray.  By Theorem \ref{thm: main}, the superpotential in each chamber can be computed by counting broken lines as illustrated in Figure \ref{fig:brlines}. The sphere bubble contributions are omitted in the figure, which are $t^{[\bar{D}_3]}$ multiplies of the broken lines with infinity ends parallel to the negative $y$-axis. 
   
   During the sequence of blowdowns $\widetilde{Y}\rightarrow Y'\rightarrow Y$, geometry of the complement of the boundary divisor is unchanged, and hence, the scattering diagram remains the same. Finally, we can derive the superpotentials for $(Y',D')$ and $(Y,D)$ with help of Lemma \ref{blow down}. See Figure \ref{fig:brlines} for the precise terms in the potential on each chamber.

Observe that the two superpotentials are exactly those for the Clifford torus and the Chekanov torus in $\mathbb{P}^2$ after a suitable coordinate change. 
The reason that the two chambers support Clifford and Chekanov tori can be explained as follows. When the blowup point $\widetilde{Y} \rightarrow \bar{Y}$ is approaching $\bar{D}_1\cap \bar{D}_2$, then the unique ray in the scattering diagram is shifting upward. At the same time, the conic $C$ degenerates into the union of $L'$ and $L''$, where $L''$ is the tangent of $C$ at the node of $D$ other than $p$. The resulting anticanonical divisor $L\cup L'\cup L''$ is the toric boundary of $Y\cong \mathbb{P}^2$.

	\begin{remark}
	In the above procedure, we do not need to know the precise contributions of the sphere bubble configurations in $\mathbb{F}_2$. For our purpose, it suffices to start with
	  \begin{align*}
	      W_{\mathbb{F}_2}=t^{\beta_1}y^{-1}+t^{\beta_2}x^{-1}y^2+(1+\sum_k a_kt^{k[\bar{D}_3]})t^{\beta_3}y+t^{\beta_4}x,
	  \end{align*} 
	  with some unknown $a_k\in \mathbb{Q}$. Any term involving $t^{[\bar{D}_3]}$ eventually vanishes in the superpotential for $(Y,D)$ since the corresponding discs turn out to have higher Maslov indices in $Y$. Once we computed the superpotential for $(\mathbb{P}^2,C\cup L)$, we can derive the superpotential of $(\mathbb{F}_2,\bar{D})$ reversing the procedures in the previous example. In particular, we see $a_1=1$ and all other $a_k$ vanish.  
\end{remark}

\subsubsection*{\textnormal{(ii)} The third Hirzebruch surface $\mathbb{F}_3$ (non-Fano)} 
	Auroux  \cite[3.2]{A5} (see Figure 4 therein) constructed special Lagrangian fibrations on $Y'\setminus D'$, where we the pair $(Y', D')$ is the same as in the previous example. There are two chambers in the base of the fibration such that fibres over the points in the same chambers have the same superpotential (with suitable correction of symplectic flux). 
	
	Fibres over one chambers are Hamiltonian isotopic to Clifford tori, and the other ones are Hamiltonian isotopic to Chekanov tori. Auroux further showed that on the Chekanov side, the superpotential for $(Y',D')$ matches the superpotential for the moment fibres of $\mathbb{F}_3$. Therefore, the latter can be read off from broken lines in (left) Figure \ref{fig:brlines}. The difference from the previous example of $Y$ is that we need to take into account two more broken lines whose infinity ends are parallel to the negative $y$-axis. Hence, in our choice of coordinates, the potential for $\mathbb{F}_3$ is given as 
	\begin{equation*}
	\begin{array}{ll}
	& x t^{\beta_4} + 2 y^2 t^{\beta_4 + \alpha} +x^{-1} y^4 t^{\beta_4 + 2 \alpha} + y t^{\beta_3} + x^{-1} y^3 t^{\beta_3 + \alpha } + y^{-1} t^{\beta_1} \\
	&= y t^{\beta_3}+  x t^{\beta_4} ( 1 + x^{-1} y^2 t^{\alpha})^2   + (y t^{-\beta_1})^{-1}  (1 + x^{-1} y^4 t^{\beta_3 - \beta_1 + \alpha} ).
	\end{array}
	\end{equation*}
It is not very difficult to match the above with equation (3.7) of \cite{A5} up to some corrections by flux (powers of $t$).

\subsubsection*{\textnormal{(iii)} Cubic surface} The mirror superpotentials of del Pezzo surfaces were written down by Galkin-Usnich \cite{GU}. 
In particular, there are 21 terms in the superpotential for a cubic surface. Therefore, Sheridan \cite{S12} conjectured that there are $21$ discs with boundaries on certain Lagrangian in a cubic surface, as the open analogue of the classical result of $21$ lines in a cubic surface. The conjecture was verified by Pascaleff-Tonkonog \cite{PT} and Vanugopalan-Woodward \cite{VW}. Here we provide a different method: consider the cubic surface $Y$ as the non-toric blowup at 6 points on $\bar{Y}=\mathbb{P}^2$. We will choose the blowup loci such that there are two points on each toric boundary divisors close to a corner and exactly two points in a small neighborhood of each corner. The corresponding GPS scattering diagram has a chamber in the middle and the $21$ broken lines with ends in the middle chamber are listed in Figure \ref{fig:cubic surface}. Then we recover the superpotential of a cubic surface from Theorem \ref{thm: main}.
\begin{figure}[h]
	\begin{center}
		\includegraphics[scale=0.5]{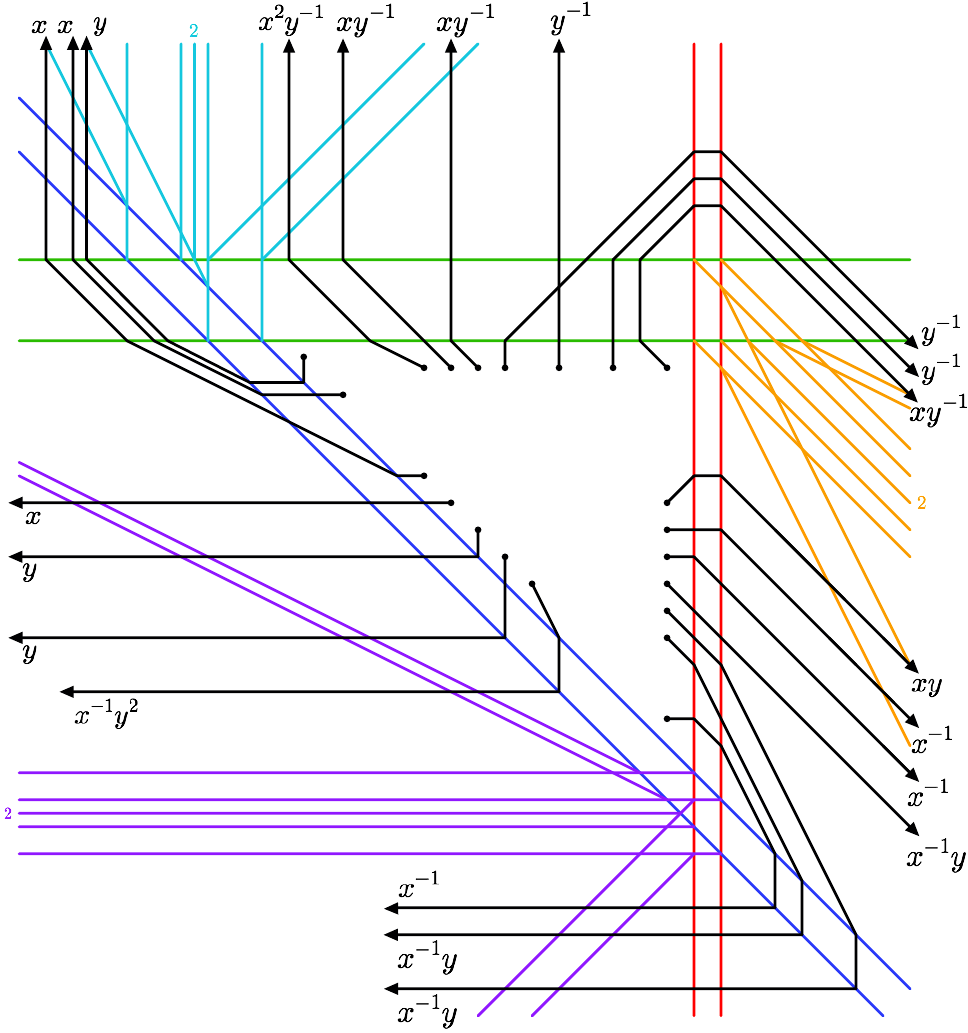}
		\caption{The $21$ broken lines in cubic surface with ends in the middle chamber.}
		\label{fig:cubic surface}
	\end{center}
\end{figure}
	
%

\section{Application to Mirror Symmetry for Rank $2$ Cluster Varieties}\label{sec:apptocluster}

In this section, we will apply Theorem \ref{correspondence: scattering diagram} to rank 2 cluster varieties. 
Here `rank 2' refers to the skew-symmetric form in fixed data (defined in the next section) are of rank 2. 
In this section, we are going to show that the quotients of the $\cA$ cluster varieties are mirror to the fibers of the Langlands dual $\cX$ cluster varieties. 

There have been many works on the mirror symmetry for cluster varieties in the literature.
From the algebro-geometric perspective, Keel-Yu \cite{keelyu} made use of Berkovich non-archimedean methods to construct mirror and compared their mirror algebra with the cluster algebra in \cite{GHKK}. 
Note that Keel-Yu assumed that the skew-symmetric forms are non-degenerate, while our skew-symmetric forms are of rank 2, i.e. we allow those which are degenerate. 
We further work on the general skew-symmetrizable case instead of skew-symmetric case. 
Mandel \cite{travis} showed that theta bases for cluster varieties are determined by descendant log Gromov-Witten invariants of the symplectic leaves of the Langlands dual cluster variety.
From the symplectic perspective, Gammage-Le \cite{ben} constructed a symplectic manifold from a given truncated cluster variety such that the homological mirror symmetry holds. Kim-Lau-Zheng \cite{KLZ} calculated the Lagrangian Floer potential for cluster varieties of type A, which supposedly provide their Landau-Ginzburg mirror.
The following discussion would be the first to start with the $\cA$ side and obtain mirrors in the Langlands dual dual of the $\cX$ side.  

\subsection{Cluster varieties} \label{sec:clusterdata}

Let us begin by defining the \emph{fixed data} $\Gamma$ for a pair of cluster varieties without frozen variables $(\cA,\cX)$, which consists of the following:
\begin{itemize} 
\setlength\itemsep{0em}
    \item a finite set $I$ with $|I|=n$; 
    \item a lattice $N$ of rank $ |I|=n$;
    \item a saturated sublattice $ N_{\mathrm{uf}} \subseteq N$ of rank $n$;
    \item a skew-symmetric bilinear form $\{ \cdot , \cdot \}: N\times N \to \Q$;
    \item a sublattice $N^{\circ}\subseteq N$ of finite index satisfying
    $\lbrace N_{\text{uf}}, N^{\circ}\rbrace \subset \Z  $ and $  \lbrace N,N_{\mathrm{uf}}\cap N^{\circ}\rbrace \subset \Z$;
    \item a tuple of positive integers $(d_i:i\in I)$ with greatest common divisor 1;
    \item $M=\Hom(N,\Z)$ and $M^{\circ}=\Hom(N^{\circ},\Z)$.
\end{itemize}

Given this fixed data, a $\emph{seed data}$ for this fixed data is $\seed := ( e_i \in N \mid i \in I)$,
where  $\{ e_i \}$ is a basis for $N$.
The basis for $M^{\circ}$ would then be $f_i = \frac{1}{d_i} e_i^*$ . 
Define the bilinear map $[\cdot, \cdot]:N\times N \to \Q$ as $[e_i,e_j]=\{e_i, d_je_j\}$, and $\langle \cdot, \cdot \rangle$ denotes the canonical pairing given by evaluation.

Define the map
\begin{align*}
	p^* &:N \rightarrow M^{\circ} \\
	n &\mapsto (N^{\circ} \ni n' \mapsto \{ n, n'\}). 
\end{align*}
We will then define $v_i = p^*(e_i) \in M^{\circ}$.
One can then associate the seed tori 
\[\cA_{\seed} = T_{N^{\circ}} = \Spec \Bbbk [M^{\circ}], \quad
\quad \cX_{\seed} = T_{M} = \Spec \Bbbk [N].\]
We will denote the coordinates as $X_i = z^{e_i}$ and $A_i = z^{f_i}$ and they are called the {cluster variables}. 
Similar to the definition of cluster algebras, there is a procedure, called \emph{mutation}, to produce a new seed data $\mu(\seed)$ from a given seed $\seed$ which is stated in \cite[Equation 2.3]{ghk_bir}. 
The essence is that we will obtain new seed tori $\cA_{\mu(\seed)}$, $\cX_{\mu(\seed)}$ from the mutated seed. 
We will consider the oriented rooted tree $\mathfrak{T}$ with $n$ outgoing edges from each vertex, labeled by the elements of $I$. Let $v$ be the root of the tree and then attach the seed $\seed$ to the vertex $v$. By associating each path in the tree with a sequence of mutation then we can then attach a seed to each vertex of $\mathfrak{T}$. 

Between the tori, there are birational maps, called $\mathcal{A}$-mutation and $\mathcal{X}$-mutation,
$\mu_{\cX}: \cX_{\seed} \dashrightarrow \cX_{\mu(\seed)}$, $\mu_{\cA}: \cA_{\seed} \dashrightarrow \cA_{\mu(\seed)}$ which are defined by pull-back of functions
\begin{align}
	\mu_{\cX,k}^* z^n &= z^n (1+z^{e_k})^{-[n , e_k]}, \label{eqn:Amut} \\
	\mu_{\cA,k}^* z^m &= z^m (1+z^{v_k})^{- \langle d_k e_k, m \rangle}, 
\end{align}
for $n \in N$, $m \in M^{\circ}$.
Note that those birational maps are basically the mutations of cluster variables at direction $k$ as in Fomin and Zelevinsky \cite{cluster1}. 

Let $\cA$ be an union of tori glued by $\cA$-mutations.
While one can define the cluster varieties as gluing of tori, cluster varieties are defined up to certain ambiguity in codimension 2 as follows:
\begin{defn} \cite{CMN}
	A smooth scheme $V$ is a cluster variety of type $\cA$ if there is a birational map $ \mu: V \dashrightarrow \cA$ which is an isomorphism outside codimension two subsets of the domain and range.
	We define a {\it{cluster variety of type $\cX$}} analogously. 
\end{defn}
By definition, the ring of regular functions of $\cA$ cluster varieties are upper $\cA$ cluster algebras $\mathrm{up}(\cA)$. 
Similarly the ring of regular functions of the $\cX$ cluster varieties are the (upper) $\cX$ cluster algebras $\mathrm{up}(\cX)$.
Note that by \cite[Theorem 3.14]{ghk_bir}, the canonical maps
\[
\cA \rightarrow \mathrm{Spec} (\mathrm{up} (\cA))
\]
are open immersions.

For the later discussion, we will note the following structure of the cluster $\cA$ and $\cX$ cluster varieties. 
Let $K = \ker p^* $\footnote{
Again as there is no frozen variables, $\ker p_2^* = \ker p^*$ if one follows the notation in \cite{ghk_bir}.}.
Then the inclusion $K \subseteq N$ gives a map $\cX_{\seed} \rightarrow T_{K^*}$. The map is compatible with mutation maps so that there is a canonical map 
\begin{align} \label{eq:Xfamily}
	\lambda: \cX \rightarrow T_{K^*}.
\end{align}
Let $\ell$ be the rank of the $p^*$ map. Then $\lambda$ is a flat family of $\ell$-dimensional schemes. The $\ell=2$ case is indicated in \cite[Section 5]{ghk_bir}. 

Similarly we consider the inclusion $K^{\circ} := K \cap N^{\circ} \hookrightarrow N^{\circ} $ induces a map of tori $T_{K^{\circ}} \rightarrow \cA_{\seed}$. This gives a torus action $T_{K^{\circ}} $ on $\cA_{\seed}$. 
The action is compatible with the mutations so that there is a canonical action of $T_{K^{\circ}}$ on $\cA$.

\subsubsection{Langlands dual}\label{subsubsec:Langlandsdual}

Consider a fixed data $\Gamma$.
The Langlands dual cluster varieties are defined with the Langlands dual fixed data ${}^L\Gamma$ as stated \cite[appendix A]{GHKK}. 
The fixed data ${}^L\Gamma$ is defined by having ${}^Ld_i= d_i^{-1}D$, where $D = \mathrm{lcm} (d_1, \dots d_n)$. The Langlands dual lattice is ${}^LN=N^{\circ}$, sublattice $({}^LN)^{\circ}= D \cdot N$, and the skew-symmetric form on ${}^LN$ as $\{ \cdot, \cdot\} = D^{-1} \{ \cdot, \cdot\}$. A seed data ${}^L \seed = ({}^Le_i =d_ie_i)$ gives a basis for ${}^LN$ and so the dual basis for ${}^L M= M^{\circ}$ is $\{ {}^L e_i^* = \frac{1}{d_i} e_i^*\}$. 
The basis for ${}^LM^{\circ} = \frac{1}{D} M$ can then be given by ${}^Lf_i = \frac{1}{{}^Ld_i} {}^L e_i^* = \frac{d_i}{D} \cdot \frac{1}{d_i} e_i^* = \frac{1}{D}  e_i^*$. One can similarly define the ${}^Lp^*$ map. Note that now ${}^L v_i = {}^Lp^*({}^Le_i) = \frac{1}{D} p^* ( d_i e_i) = \frac{d_i}{D}v_i$. 

Repeat the same procedure and associate the tori
\[{}^L \cA_{\seed} = T_{({}^L N^{\circ})} = \Spec \Bbbk [{}^L M^{\circ}], \quad
\quad {}^L \cX_{\seed} = T_{({}^L M)} = \Spec \Bbbk [{}^L N].\]

Again there are birational maps between the tori,
$\mu_{{}^L \cX}: {}^L \cX_{\seed} \dashrightarrow {}^L \cX_{\mu(\seed)}$, $\mu_{{}^L \cA}: {}^L \cA_{\seed} \dashrightarrow {}^L \cA_{\mu(\seed)}$ which are defined by pull-back of functions
\begin{align*}
\mu_{{}^L \cX,k}^* z^{{}^L n} &= z^{{}^L n} (1+z^{{}^Le_k})^{-{}^L[{}^Ln , {}^Le_k]}, \\
\mu_{{}^L \cA,k}^* z^{{}^Lm} &= z^{{}^Lm} (1+z^{{}^Lv_k})^{- \langle {}^Ld_k {}^Le_k, {}^Lm \rangle}, 
\end{align*}
for ${}^Ln \in {}^LN$, ${}^Lm \in {}^LM^{\circ}$.
Let us unfold the mutation maps and express it in terms of the fixed data $\Gamma$: 
\begin{align}
\mu_{{}^L \cX,k}^* z^{{}^L n} &= z^{{}^L n} (1+z^{d_k e_k})^{-{}^L[{}^Ln , d_k e_k]}, \label{eqn:LangX}\\
\mu_{{}^L \cA,k}^* z^{{}^Lm} &= z^{{}^Lm} (1+z^{{}^Lv_k})^{-\langle De_k, {}^Lm \rangle}.
\end{align}

One can then define the Langlands dual cluster varieties ${}^L\cA$ and ${}^L\cX$ by repeating the same gluing construction. 

The mirror conjecture is that $\cA$ and $ {}^L\cX$ are mirror to each other.
We will give a rough digression here. Recall the $\cA$ and ${}^L \cX$ mutation maps as in \eqref{eqn:Amut} and \eqref{eqn:LangX}:
\begin{align*}
\mu_{\cA,k}^* z^m &= z^m (1+z^{v_k})^{- \langle d_k e_k, m \rangle}, \\
\mu_{({}^L \cX),k}^* z^{{}^L n} &= z^{{}^L n} (1+z^{d_k e_k})^{-{}^L[{}^Ln , d_k e_k]}.
\end{align*}
While the notation may be confusing, we recall that ${}^L X_i = z^{d_ie_i}$. 
Let us illustrate the duality with a more explicit example. 
Consider a fixed data with a dimension 2 lattice $N$, skew-form
$\begin{pmatrix}
	0 & 1\\
	-1 & 0
\end{pmatrix},$
and $d_1=k$, $d_2=l$, where $k$ and $l$ are coprime. Then the mutation function of $\mu^*_{\cA}$ are
\[
1+A_1^{-k} \, \text{ and } \ 1+A_2^{l}.  
\]
On the other hand, the mutation functions for the $\cX$ Langlands dual $\mu_{{}^L\cX}$ are
\[
(1+{}^LX_1)^l \, \text{ and } (1+{}^LX_2)^k. 
\]
Note there are interchanging of exponent vectors. 
We are going to see the geometric reason later. 

\subsubsection{Cluster scattering diagrams} \label{sec:clusterD}
Gross-Hacking-Keel-Kontsevich \cite{GHKK} constructed bases for cluster algebras by using ideas in the Gross-Hacking-Keel mirror construction. 
The gluing construction of cluster varieties can be seen as gluing tori associated to each chamber of the scattering diagrams and then the collections of theta functions give the bases for cluster algebras. 

We start by first defining the $\cA_{\mathrm{prin}}$ cluster scattering diagrams.   
The fixed data for cluster algebras with principal coefficients can be obtained by `doubling' the original fixed data, i.e. considering the lattice $\widetilde{N} = N \oplus M^{\circ}$ as described in \cite[Construction 2.11]{ghk_bir} and notation as in \cite{CMN}. 
The initial scattering diagram is of the form
\[
\fD_{in, \seed}^{\cA_{\mathrm{prin}}} := \left( (e_i, 0)^{\perp} , 1+z^{ \tilde{p}^*\left( (e_i,0) \right)}\right). 
\]
The $\fD^{\cA_{\mathrm{prin}}}$ scattering diagrams are the consistent scattering diagrams containing $\fD_{in, \seed}^{\cA_{\mathrm{prin}}}$
 \footnote{There is an unfortunate sign switch between the setups in GPS \cite{GHK} and cluster \cite{GHKK} scattering diagrams. It would be resulted in the flipping scattered walls but not the initial walls of the diagram. Also the resulting algebras given by the collection of theta functions would be the same. }. 
Theta functions are then defined on the cluster scattering diagrams which are similar to the one in \eqref{eqn:thetadinbeta}. One can refer to \cite[Section 3]{GHKK} for the precise definition. 

Let $\mathrm{can}(V)$ be the vector spaces generated by the theta functions for $V=\cA, \cX$, and let $\mathrm{mid}(V)$ be the algebra generated by those theta functions which are polynomials (see \cite[Definition 7.2]{GHKK} for more precise definition). 
Note that $\mathrm{can}(V)$ is not an algebra in general. 
Then it is said that the full Fock–Goncharov conjecture holds when $\mathrm{mid}(V)=\mathrm{up} (V) = \mathrm{can} (V)$.
The existence of positive polytope \cite[Proposition 8.17]{GHKK} give algebra structures to $\mathrm{can}(V)$.
With the enough global monomial condition satisfying,  $\mathrm{can}(V)$ is finitely generated $\Bbbk$-algebras.
and the Fock-Goncharov conjecture holds. 
 
Note that even though the walls of $\fD^{\cA_{\mathrm{prin}}}$ rely on a choice of the seed $\seed$; however, Gross-Hacking-Keel-Kontsevich \cite{GHKK} showed that the construction is mutation invariant, i.e., the algebras generated by the collections of theta functions are the same.

The `$\cA$ and $\cX$ scattering diagrams' are the diagrams such that we can compute the corresponding theta functions. 
The $\cA$ scattering diagrams is the image of $\fD^{\cA_{\mathrm{prin}}}$ under the projection map $\rho^T: {}^L\cX_{\mathrm{prin}}(\Z^T) = M^{\circ} \times N \rightarrow  {}^L\cX(\Z^T) =M^{\circ}$, $(m,n) \mapsto m$.
On the other hand, the $\cX$ theta functions are defined from broken lines which lie in the slice of the scattering diagram $\fD^{\cA_{\mathrm{prin}}}$ intersected by the subspace
\begin{align} \label{eq:slices}
\left\{ \left.(m,n) \in \widetilde{M}^\circ_\R \right| m=p^*(n) \right\} \subseteq \widetilde{M}^\circ_\R, 
\end{align}
which we will call the diagrams on the slices as the $\cX$ scattering diagrams $\fD^{\cX}_{\seed}$.

On the other hand, the scattering diagrams for the Langlands dual cluster algebras are defined in the same way. The initial walls of the ${}^L \cX$ scattering diagrams will be stated in later discussion.  The ${}^L \cX$ scattering diagrams can be obtained by the slicing, similar to \eqref{eq:slices}, in the ${}^L \cA_{\mathrm{prin}}$ scattering diagram. 
One can also see the initial wall functions by using \eqref{eqn:LangX}:
\begin{align} \label{eq:LangXdiag}
\bigg( \fd = \{ ( p^*({}^Ln) , {}^L n) \mid {}^L p^*({}^Ln)  \in d_ie_i^{\perp}  \}  , \, (1+z^{d_i e_i})^{\mathrm{ind} v_i} \bigg) ,
\end{align}
where $\mathrm{ind} v_i$ denotes the index for $v_i$ in ${}^LM = M^{\circ}$. With a change of variables, we can see the scattering diagrams lie in $ {}^LN_{\R}= N^{\circ}_\R$.

\subsubsection{Toric models for cluster varieties}

In this section, we are going to describe the relation between the cluster varieties and the blowups of toric varieties \cite{ghk_bir}. 
Let $\seed$ be a seed and consider the fans
\begin{align}
\Sigma_{\seed, \cA} &:= \{ 0 \} \cup \{ \R_{\geq 0} d_i e_i | i \in I \} \subseteq N^{\circ}, \label{eqn:Afan} \\
\Sigma_{\seed, \cX} &:= \{ 0 \} \cup \{- \R_{\geq 0} d_i v_i | i \in I \} \subseteq M. 
\end{align}
We can then consider the toric varieties $\TV_{\seed, \cA}$, $\TV_{\seed, \cX}$ associated to the fans respectively. 
Let $D_i$, $i \in I$, be the toric divisor corresponding to the one-dimensional ray of the fan. 
Then we consider the closed subschemes
\begin{align*}
Z_{\cA, i} &: = D_i \cap \overline{V} (1+z^{v_i}) \subseteq \TV_{\seed, \cA}, \\
Z_{\cX, i} &: = D_i \cap \overline{V} \left( (1+z^{e_i})^{\mathrm{ind}d_i v_i} \right) \subseteq \TV_{\seed, \cX},
\end{align*}
where $\mathrm{ind} d_i v_i$ denotes the greatest degree of divisibility of $d_i v_i$ in $M$. 
We will denote $(\widetilde{\TV}_{\seed, \cA}, D)$ (and $(\widetilde{\TV}_{\seed, \cX}, D)$) be the blowups of $\TV_{\seed, \cA}$ along $Z_{\cA, i} $ over all $i$ (the blow up of $\TV_{\seed, \cX}$ along $Z_{\cX, i}$ over all $i$ respectively), where $D$ is the proper transform of the toric boundaries. 
As the construction for the $\cA$ and $\cX$ are parallel to each other, we will denote $\cA$ or $\cX$ as $V$.
Define $U_{\seed, V} = \widetilde{\TV}_{\seed, V} \setminus D$. 
We will further consider $U'_{\seed, V} \subset V$, the union of the tori $V_{\seed}$ and $V_{\mu_i(\seed)}$, for all $ i \in I$. 

Note that the construction of $U_{\seed, V} $
may appear to depend on a choice of seed data. Gross-Hacking-Keel
\cite[Lemma 3.8 (1)]{ghk_bir} showed that for a seed $\seed_w$ and a mutated seed $\mu_i(\seed_w)$, then one has $U_{\seed_w, V}$ and $U_{\mu(\seed_{w}), V}$ are isomorphic outside codimension two for $V = \cA_{\mathrm{prin}}$, $\cX$. 
For the $\cA$ case, they are still  isomorphic outside codimension two if the seed $\seed_w$ is coprime (\cite[Definition 3.7]{ghk_bir}).

\subsection{Mirror symmetry for rank 2 cluster varieties}

Gross-Hacking-Keel \cite{ghk_bir} indicated the Looijenga pairs constructed in \cite{GHK} agree with the cluster $\cX$ varieties in the rank $2$ case.
In this section, we will investigate the $T_{K^{\circ}}$ action on $\cA$ in the rank $2$ case.

Fix a seed $\seed = (e_i)$. We will assume without loss of generality that $p^*(e_i) \neq 0$, and for  $i \neq j$, $p^*(e_i)$ and $p^*(e_j)$ are not scalar multiples of each other. 
We will consider the lattice quotient by the kernel of $p^*$ map. If $p^*(e_i) = p^*(e_j)$, we can consider a lower rank lattice and then increase the degree of the blow-ups. The case when $p^*(e_i) = \alpha p^*(e_j)$ for some $\alpha \neq 1$ are similar but require more subtle combinatorics.

Consider a choice of splitting of $N \cong N/K \oplus K$ and this induces a dual splitting $M = K^{\perp} \oplus K^*$. 
This gives us $M^{\circ} = (K^{\circ})^{\perp} \oplus (K^{\circ})^*$.
Then we define the rank two lattices:  $\bar{N} = N/K$, $\bar{N}^{\circ} = N^{\circ} / K^{\circ}$,  $\bar{M}= M/ K^{*} = K^{\perp}$ and $\bar{M}^{\circ}= M^{\circ}/ (K^{\circ})^* $.
We further consider the descended $p^*$ map
\begin{align*}
\bar{p}^* : \bar{N} \rightarrow \bar{M}^{\circ}.
\end{align*}
Let $\bar{e}_k \in \bar{N}$, $\bar{v}_k \in \bar{M}^{\circ}$ be the image of $e_k$, $v_k$ respectively in the quotient spaces. 
By construction, $\bar{v}_k = \bar{p}^* (\bar{e}_k)$ for all $k$. 
Therefore, we still have $\langle \bar{e}_k , \bar{v}_k \rangle = 0$. 

Recall the inclusion $K^{\circ} \hookrightarrow N^{\circ}$ gives us the $T_{K^{\circ}}$ action on $\cA_{\seed}$. 
Note that the torus action commutes with the mutation of seed data, i.e. for $k \in I$, the following diagram commutes
\[ \begin{tikzcd}
T_{K^{\circ}} \arrow{r} \arrow{d}{=} & \cA_{\seed} \arrow[dashrightarrow]{d}{\mu_k} \\
T_{K^{\circ}} \arrow{r}& \cA_{\mu_k(\seed)}
\end{tikzcd}
\]

Consider the tori $\cA_{k, \seed} = \mathrm{Spec} \Bbbk [\bar{M}^{\circ}] \cong \cA_{\seed} / T_{K^{\circ}} $. 
Since the $T_{K^{\circ}}$ action on $\cA_{\seed}$ commutes with mutation, we can define the birational maps
\begin{align} \label{eq:quotientmut}
	\bar{\mu}_k: \cA_{k', \seed} \dashrightarrow \cA_{k', \mu_k(\seed)} 
\end{align}
via
\begin{align*}
	\bar{\mu}_k^* (z^{\bar{m}}) = z^{\bar{m}} \left( 1+ z^{\bar{v}_k}\right) ^{-\langle d_k \bar{e}_k , \bar{m} \rangle},
\end{align*}
for $\bar{m} \in \bar{M}^{\circ}$.
Hence, by \cite[Proposition 2.4]{ghk_bir} we can define the scheme $\cA_k$ as the scheme by gluing all $\cA_{k, \mu_k(\seed)} $ via the mutation map $\bar{\mu}_k$.

Recall we assume that for $i \neq j$, $p^*(e_i)$ and $p^*(e_j)$ are not scalar multiple of each other in the beginning of this section. 
Note that this is not an extra condition. 
If $p^*(e_i)$ and $p^*(e_j)$ are scalar multiple of each other for some $i \neq j$, then $e_i -\alpha e_j \in K$ for some scalar $\alpha$, i.e. $\bar{e}_i = \alpha \bar{e}_j$. Hence by \eqref{eq:quotientmut} mutation at index $i$ would be
\begin{align} \label{cluster transf}
\bar{\mu}_i^* (z^{\bar{m}}) = z^{\bar{m}} \left( 1+ z^{\alpha\bar{v}_j}\right) ^{\alpha\langle \bar{e}_j , \bar{m} \rangle}.
\end{align}
Hence $\bar{\mu}_i,\bar{\mu}_j$ have the same mutating directions, in the sense that the function factor on the right hand side of \eqref{cluster transf} and that of the $\bar{\mu}_j^*$ can be polynomials of the same variable.
We can then reduce the rank of the cocharacter lattice in the fixed data $\Gamma$ and combine the mutations by multiplying them together.
One may have noticed that the mutation formulas have changed; however, they still satisfied the cluster or generalized cluster structure as in \cite{origgeneralized}. Further, their scattering diagrams and theta functions are developed in \cite{generalized}. 

Next we would describe the scheme $\cA_k$ as blowup of toric surfaces by following \cite[Construction 3.4]{ghk_bir}. 
Define $\overline{\Sigma}_{\seed, \cA_k} $ as the quotient of the fan $\Sigma_{\seed, \cA}$ in \eqref{eqn:Afan} as
\[
\overline{\Sigma}_{\seed, \cA_k} = \{ 0\} \cup \{ \R_{\geq 0} d_i \bar{e}_i\} \subseteq \bar{N}^{\circ }.
\]
Note that $ \R_{\geq 0} d_i \bar{e}_i \neq  \R_{\geq 0} d_i \bar{e}_j$ for $i\neq j$ from our assumption. 
Let $\mathrm{TV}_{\overline{\Sigma}_{\seed, \cA_k}}$ be the toric variety associated to the fan $\overline{\Sigma}_{\seed, \cA_k}$. 
We denote $\bar{D}_i$ to be the toric boundary divisor of $\mathrm{TV}_{\overline{\Sigma}}$ corresponding to the ray generated by $d_i\bar{e}_i\in \overline{\Sigma}$.  
Let $\bar{v}_i$ be the image of $v_i$.
Note that $\bar{v}_i \neq 0$ as it is in the image of the $p^*$ map. 

For $i=1, \dots, n$, consider the closed subscheme
\[
\bar{Z}_i = \bar{D}_i \cap \bar{V} \left( (1+z^{\bar{v}_i})^{\mathrm{ind} \, d_i \bar{e}_i} \right),
\]
where $\mathrm{ind}\, d_i \bar{e}_i$ denotes the greatest degree of divisibility of $d_i \bar{e}_i$ in $\bar{N}^{\circ}$. 
Note that even though $d_i e_i$ are primitive in $N^{\circ}$, however, they may not be primitive in the quotient lattice $\bar{N}^{\circ}$. 
We will then consider the blow up $\widetilde{\mathrm{TV}}_{\overline{\Sigma}_{\seed, \cA_k}}$ at $\bar{Z}_i$'s.
Let $\widetilde{D}_i$ be the proper transform of the toric boundaries $\bar{D}_i$ and define $U_{\seed, \cA_k} :=  \widetilde{\mathrm{TV}}_{\overline{\Sigma}_{\seed, \cA_k}} \setminus \cup_i \widetilde{D}_i$. 
We further define $U_{\seed, \cA_k}'$ as the union of the tori $ \cA_{k, \seed} $ and $ \cA_{k, \mu_k( \seed)}$ for $i \in I$. 
Note that the condition $p^*(e_i)$ and $p^*(e_j)$ not scalar multiple of each other for $i \neq j$, so we can apply both \cite[Lemma 3.5(i)]{ghk_bir} and \cite[Lemma 3.6]{ghk_bir}.
Hence by \cite[Lemma 3.5(i)]{ghk_bir}, $U_{\seed, \cA_k}$ and $U_{\seed, \cA_k}'$ are isomorphic outside a codimension two set. 
We also have $U_{\seed, \cA_k}$ and $U_{\mu(\seed), \cA_k}$ isomorphic outside a codimension two set by \cite[Lemma 3.6]{ghk_bir}.

\subsubsection{The GPS scattering diagrams associated to rank 2 cluster $\cA$ varieties}

We can now run the machinery discussed in Section \ref{subsec:GPS}. 
First we would like to associate $\widetilde{\mathrm{TV}}_{\overline{\Sigma}_{\seed, \cA_k}}$ with a scattering diagram which the walls carry enumerative interpretation.   
Denote the fan $\overline{\Sigma}=\overline{\Sigma}_{\seed, \cA_k}$.
Let $\bar{\Sigma}'$ be a complete fan containing $\overline{\Sigma}$ and $\TV_{\bar{\Sigma}'}$ be the corresponding toric surface. 
Then we have $\TV_{\bar{\Sigma}} \hookrightarrow  \TV_{\bar{\Sigma}'} $. 
We denote the Gross-Pandharipande-Siebert scattering diagram for the blow up $\widetilde{\TV}_{\bar{\Sigma}'}$ of $\TV_{\bar{\Sigma}'}$ at the proper transform of $Z_i$ by $\mathfrak{D}^{GPS}_{\overline{\Sigma},\seed}$. 
Note that the additional divisors associated to $\bar{\Sigma}' \setminus \bar{\Sigma}$ do not intersect the $Z_i$'s. 
Hence the $\mathbb{A}^1$-curve countings in the wall-functions of $\mathfrak{D}^{GPS}_{\overline{\Sigma},\seed}$ are independent of the choice of $ \bar{\Sigma}' \supseteq \bar{\Sigma}$.
We would then switch back to denote $\bar{\Sigma}$ as the complete fan containing $\overline{\Sigma}_{\seed, \cA_k} $. 

We are going to define the initial scattering diagram associated to $ \mathrm{TV}_{\overline{\Sigma}_{\seed, \cA_k}}$. 
First note that all the initial walls will pass through the origin. 
Denote $\bar{v}_i = (\mathrm{ind } \, \bar{v}_i ) q_i$, where $q_i$ is primitive in $\bar{N}^{\circ}$.
The reason is the blow-up points are at the zero locus of the equations $1+ z^{\bar{v}_i}$. Hence we have $z^{ \bar{v}_i} = -1$ which then means $|z^{q_i}| = 1$. If one uses $z_1,z_2$ as the toric coordinates on $(\mathbb{C}^*)^2\subseteq TV_{\bar{\Sigma}_{s,\mathcal{A}_k}}$, then there exist $a,b\in \mathbb{Z}$ such that $z^{q_i}=z_1^az_2^b$. Therefore, under the $\mbox{Log}$-map, the wall induced by the non-toric blow up at any point on $\{|z^{q_i}=1|\}\cap \bar{D}_i$ is a line passing through the origin. 

The wall functions from one point blowups are divided into two cases: $d_i \bar{e}_i$ is primitive or not primitive in $\bar{N}^{\circ}$.
If $d_i \bar{e}_i$ is primitive, we can apply the discussion in Section \ref{subsec:GPS} to define the corresponding wall
\[
\left( \R \cdot d_i \bar{e}_i, 1+z^{d_i\bar{e}_i} \right).
\]
In Section \ref{subsec:GPS}, the associated wall functions are of the form $1+t^{-[E_{ij}]}z^{m_i}$. Here we put $t^{-[E_{ij}]}=1$ in the cluster setting. 

If $d_i \bar{e}_i$ is not primitive, i.e. $\mathrm{ind } \, d_i \bar{e}_i >1$, the corresponding blow-up would be a  $(\mathrm{ind } \, d_i \bar{e}_i )$-orbifold blowup of the toric variety along the points $\bar{D}_i \cap V(1+z^{d_i\bar{e}_i})$. 
We can apply \cite[Theorem 5.6]{GPS} to associate the wall
\[
\left( \R \cdot q_i, 1+z^{(\mathrm{ind } \, d_i \bar{e}_i ) q_i} \right)=
\left( \R \cdot d_i \bar{e}_i, 1+z^{d_i\bar{e}_i} \right)
\]
is again in the same form. 

Roughly speaking, the construction in Section \ref{subsec:GPS} is done by associating each point of blowing-up with an initial wall.  
Let us work over $\C$ or any algebraically closed characteristic 0 fields.
For each $i$, the subscheme $\bar{D}_i \cap \bar{V} \left( 1+z^{\bar{v}_i} \right)$ consists of $\mathrm{ind} \bar{v}_i$ many distinct points. From the discussion above, to each point in the subscheme $\bar{D}_i \cap \bar{V} \left( 1+z^{\bar{v}_i} \right)$, we associate the wall
\[
\left( \R \cdot d_i \bar{e}_i, 1+z^{d_i\bar{e}_i} \right).
\]
Hence the initial walls associated to the blowup loci $Z_i$ are
\[
\left( \R \cdot d_i \bar{e}_i, (1+z^{d_i\bar{e}_i})^{\mathrm{ind} \bar{v}_i} \right).
\]

Therefore, the initial scattering diagram associated to $ \widetilde{\mathrm{TV}}_{\overline{\Sigma}_{\seed, \cA}}$ would be
\begin{align} \label{eq:GPSscatt}
\fD^{{GPS}}_{\mathrm{in}, \overline{\Sigma}, \seed} =\left\{ \left(\R \cdot d_i \bar{e}_i, (1+z^{d_i\bar{e}_i})^{\mathrm{ind} \bar{v}_i}  \right) \mid i = 1 \dots n \right\}.
\end{align}
Let $\fD^{{GPS}}_{\overline{\Sigma},\seed} $ be the consistent scattering diagram containing $\fD_{\mathrm{in}, \seed}$. 
As described in \eqref{eq:gps}, the wall functions of  $\fD^{{GPS}}_{\overline{\Sigma},\seed} $ would be expressed in terms of log Gromov-Witten invariants. 

\subsubsection*{Matching the cluster scattering diagrams}

Now we would like to obtain the scattering diagrams \eqref{eq:GPSscatt} from the cluster viewpoint. 

Recall when the $p^*$ map is not injective, there is a morphism $\lambda: \cX \rightarrow T_{K^*}$ as stated in \ref{eq:Xfamily}. 
Let $X=U'_{\cX}$ denotes the open subset of $\cX$ obtained by gluing the seed tori $\cX_{\seed}$ and $\cX_{\mu_{k}(\seed)}$, for all $k$. 
The morphism $\lambda$ descends to $\lambda: X \rightarrow T_{K^*}$.
In the discussion above, we see that there is a flat family of surface $\mathcal{Y} \rightarrow T_K^*$ obtained by blowing up toric varieties such that $\mathcal{Y} \setminus
\mathcal{D}$ are isomorphic to $U'_{\cX}$ up to codimension two, where $\mathcal{D}$ is the proper transform of the toric boundary under the blowup. 
Gross-Hacking-Keel \cite{ghk_bir} indicated that the family $\mathcal{Y} \rightarrow T_K^*$ is the universal family of Looijenga pairs constructed in \cite{GHK}. 
For $\phi \in T_K^*$, let $\cX_\phi$ denote the fibre over $\phi \in T_{K^*}$. 
The $\cX_{\phi}$ scattering diagram $\fD^{X_{\phi}}_{\seed}$ defined in \cite[Section 2.2.3]{CMN} can be described as the quotient of the $\fD^{\cX}_{\seed}$ diagram in $\bar{N} \otimes \R$. This construction is well-defined as laid out in \cite[Section 2.2.3]{CMN}. Besides, Zhou \cite{yang} also gave a description of quotient scattering diagrams. 

We will explicitly write down $\fD^{{}^L\cX_{e}}_{\seed}$, where $e$ is the identity of $T_{{}^LK^*}$. 
Recall ${}^LN = N^{\circ}$. 
With a choice of the splitting $N^{\circ} = \bar{N}^{\circ} \oplus K^{\circ}$ as in the earlier section, there is the inclusion
\begin{align*}
	K^{\circ} &\hookrightarrow \bar{N}^{\circ} \oplus K^{\circ} \\
	k &\mapsto (0,k). 
\end{align*}
Hence similar to the $\cA$ scattering diagrams defined from the $\cA_{\mathrm{prin}}$ scattering diagrams, 
we can consider a projection of the $\fD^{{}^L\cX}$ scattering diagram descended from $N^{\circ}_{\R} \rightarrow \overline{N}^{\circ}_{\R}$. 
Note that in the two-dimensional lattice $N^{\circ}$, the normal space to $(d_i \bar{e}_i )^{\perp}$ is the simply one-dimensional ray $\R d_i \bar{e_i}$. 
Combining with \eqref{eq:LangXdiag}, the initial ${}^L\cX_{e}$ scattering diagrams are of the form
\begin{align*}
\big( \R  \cdot d_i \bar{e_i} , \, (1+z^{d_i \bar{e}_i})^{\mathrm{ind} \bar{v}_i} \big). 
\end{align*}
The ${}^L\cX_{e}$ scattering diagrams are the consistent scattering diagrams containing the initial diagram. 
Similar to the discussion in cluster scattering diagram, if there exist positive polytopes and the enough global monomials conditions are satisfied, the vector space $\mathrm{can} ({}^L\cX_e)$ are finitely generated algebra. 
Then we have $\mathrm{mid} ({}^L\cX_e) = \mathrm{can} ({}^L\cX_e) = H^0 ({}^L\cX_e, \mathcal{O} )$.

Note that the initial ${}^L\cX_{e}$ scattering diagram is exactly the same as the Gross-Pandharipande-Siebert scattering diagram in \eqref{eq:GPSscatt}. 
In other words, we have 
 \begin{align}
    \mathfrak{D}^{GPS}_{\overline{\Sigma},\seed}\cong \mathfrak{D}^{{}^L\cX_{e}}_{{}^L\seed},
 \end{align} for every seed $\seed$. Since  $\mathfrak{D}^{{}^L\cX_{e}}_{{}^L\seed}$ and $\mathfrak{D}^{{}^L\cX_e}_{\mu({}^L\seed)}$ differ by mutation by \cite[Theorem 1.24]{GHKK}, we have   $\mathfrak{D}^{GPS}_{\overline{\Sigma},\seed}$ and  $\mathfrak{D}^{GPS}_{\overline{\Sigma},\mu(\seed)}$ are also differed by a mutation.
 
Let us recall the ideas of the Gross-Hacking-Keel \cite{GHK} mirror construction. 
Once the canonical scattering diagram in Section \ref{sec: can} is constructed, the mirrors are constructed as the spectrums of the algebras generated by the collections of theta functions. 
Further the algebras given by the canonical scattering diagrams and the Gross-Pandharipande-Siebert scattering diagrams are identified in \cite{GHK}. 
Now we see that the latter coincides with the ${}^L\cX_e$ scattering diagrams. 
Hence the mirror spaces are ${}^L\cX_e$. 

Now we are ready to combine Floer theory into this story. 
Theorem \ref{correspondence: scattering diagram} tells us that the limiting Lagrangian Floer scattering diagrams agree with the Gross-Pandharipande-Siebert scattering diagrams. Hence we showed that quotients of the $\cA$ cluster varieties are mirror to the fibers at $e\in T_{K^*}$ of the $\cX$ cluster varieties in the sense that the Floer theory of the $\cA$ cluster varieties determined the complex structure of the fibres at $e\in T_{K^*}$ of the $\cX$ cluster varieties.

\begin{theorem} \label{thm:cluster}
	Consider a cluster fixed data $\Gamma$ without frozen direction where the skew-symmetric form is of rank 2.  
	We further fix a seed data $\seed$. 
	Assume none of the $p^*(e_i)$ are scalar multiples of each other.
	Consider the quotient map $N^{\circ} \rightarrow N^{\circ}/K^{\circ}$ and let $\bar{e}_i$ be the image of $e_i$ under this quotient map. 
	Define 	\[
	\overline{\Sigma}_{\seed, \cA_k} = \{ 0\} \cup \{ \R_{\geq } d_i \bar{e}_i\} \subseteq \bar{N}^{\circ }.
	\]
	For $i=1, \dots , n$, let $v_i = p^*(e_i)$ and let $\bar{v}_i$ be the image of $v_i$ under the quotient map $M^{\circ} \rightarrow M^{\circ} / (K^{\circ})^*$. 
	Define
	\[
	\bar{Z}_i = \bar{D}_i \cap \bar{V} \left( (1+z^{\bar{v}_i})^{\mathrm{ind} \, d_i \bar{e}_i} \right),
	\]
	where $\mathrm{ind}\, d_i \bar{e}_i$ denotes the greatest degree of divisibility of $d_i \bar{e}_i$ in $\bar{N}^{\circ}$. 
	
	Let $\widetilde{\mathrm{TV}}_{\overline{\Sigma}_{\seed, \cA_k}}$ be the surface obtained by blowing up the subschemes $\bar{Z}_i \subseteq \mathrm{TV}(\overline{\Sigma}_{\seed, \cA_k})$. 
	Let $\widetilde{D}_i$ be the proper transform of the toric boundaries $\bar{D}_i$ and define $U_{\seed, \cA_k} :=  \widetilde{\mathrm{TV}}_{\overline{\Sigma}_{\seed, \cA}} \setminus \cup_i \widetilde{D}_i$. 
	Then $U_{\seed, \cA_k}$ is isomorphic to $\cA / T_{K^\circ}$ outside a set of codimension at least two, where  $\cA / T_{K^\circ}$ is the quotient of the $\cA$ variety under the $T_{K^\circ}$-action. 
	
	Then the corresponding GPS scattering diagram of  $\widetilde{\mathrm{TV}}_{\overline{\Sigma}_{\seed, \cA_k}}$ is the  ${}^L\cX_{e, \seed}$ scattering diagram, where ${}^L\cX_{e}$ is the fiber of the Langlands dual $\cX$ family ${}^L\cX \rightarrow T_{({}^LK)^*}$ at $e \in T_{({}^LK)^*}$. 
\end{theorem}

\begin{remark}
	If $\mathrm{can}({}^L \cX_e)$ are finitely generated algebras, we obtain the mirror algebras of $\cA / T_{K^{\circ}}$ are rings of regular functions of ${}^L\cX_{e}$. 
\end{remark}

\begin{remark}
We use the algebraic geometric enumerative invariants (from Gross-Siebert-Pandharipande) of $\widetilde{\mathrm{TV}}_{\overline{\Sigma}_{\seed, \cA_k}}$ to recover the $\mathcal{X}$-scattering diagram of ${}^L\mathcal{X}_{e,\seed}$ in Theorem \ref{thm:cluster}. 
If we further assume $d_i \bar{e}_i$ are primitive in $N^{\circ}/ K^{\circ}$, i.e. excluding the case about orbifold blowup, we can recover ${}^L\mathcal{X}_{e,\seed}$ from the open Gromov-Witten invariants of $\widetilde{\mathrm{TV}}_{\overline{\Sigma}_{\seed, \cA_k}}$ by Theorem \ref{thm: main}. 
\end{remark}

\subsection{Example}

Let us consider an example with $\rank \ N = 3$ and the skew symmetric form is of the form 
\begin{equation*}
\begin{pmatrix}
0 & 1 & -1 \\
-1 & 0 & 1 \\
1 & -1 & 0
\end{pmatrix},
\end{equation*}
and $d_i=1$ for all $i$. Note that since $d_i=1$ for all $i$, the fixed data would be the same for the Langlands dual side. 
Consider the initial seed $\seed =\{ e_i\}$, where $e_i$ are the standard basis of $\Z^3$. Then we have
\begin{align*}
	\Sigma_{\seed, \cA} &= \{0\} \cup \R_{\geq 0} (1,0,0)\cup \R_{\geq 0} (0,1,0)\cup \R_{\geq 0} (0,0,1), \\
	\Sigma_{\seed, \cX} &= \{0\} \cup \R_{\geq 0} (0,-1,1)\cup \R_{\geq 0} (1,0,-1)\cup \R_{\geq 0} (-1,1,0).
\end{align*}

Let $\mathrm{TV}_{\seed, \cA}$ be the toric variety associated to the fan $\Sigma_{\seed, \cA}$. 
One can see immediately that $\mathrm{TV}_{\seed, \cA} = \bA^3 \setminus \{$codimension 2 toric strata$\}$. 
In this example, we have $K = \ker p^*= \langle e_1 + e_2 + e_3\rangle \subseteq N$.
We then consider $K^{\circ} = K \cap N^{\circ} \subset N^{\circ}$. 
In this example $K^{\circ} = K$.  
The $T_{K^{\circ}}$ action is of the form
\[
\alpha \cdot (A_1, A_2, A_3) = (\alpha A_1,\alpha A_2, \alpha A_3),
\] 
where $\alpha \in \C^*$. 
One can see that this is the standard torus action to obtain $\PP^2$ from $\bA^3$. 
In this case, $\mathrm{TV}_{\seed, \cA} / T_{K^{\circ}}$ will give us $\PP^2 \setminus \{$3 points$\}$. 
Alternately, consider the quotient map $N^{\circ} \rightarrow N^{\circ}/K^{\circ}$. Then the imagine of the fan $\Sigma_{\seed, \cA}$ is of the form
\[
\overline{\Sigma}_{\seed, \cA} = \{0\} \cup \R_{\geq 0} \bar{e}_1\cup \R_{\geq 0} \bar{e}_2 \cup \R_{\geq 0} \cdot(-\bar{e}_1-\bar{e}_2).
\]
Hence the toric variety is the projective space $\PP^2$ up to codimension 2. 
The space we considered would be $\PP^2$ blow up at the 3 points. 
The corresponding initial scattering  diagram would be
\begin{figure}[H]
	\centering
	\begin{tikzpicture}
	\draw	(0,-1) -- (0,1)node[left] {$1+z^{\bar{e}_2}$};
	\draw	(-1,0) -- (1,0)node[right] {$1+z^{\bar{e}_1}$};
	\draw	(1,1) -- (-1,-1)node[below] {$1+z^{-\bar{e}_1-\bar{e}_2}$};
	\end{tikzpicture}
\end{figure}
It is easy to see these are the initial walls of the ${}^L\cX_e$ scattering diagram of the dual seed of $\seed$.

\section{Mirror Symmetry}\label{sec:msdp5cal}

In this section, we focus on the concrete example of the del Pezzo surface of degree $5$, and show that the Landau-Ginzburg model constructed by disc counting fulfills closed-string mirror symmetry. For more general examples beyond this, the calculation of critical points becomes too complicated to carry out by hands. We speculate that there is a more effective argument that avoids a direct calculation. The main theorem in this section is the following.

\begin{theorem}\label{thm:maindp5}
Let $\widetilde{Y}$ be a del Pezzo surface of degree $5$ represented by the non-toric blowup of a toric surface $\bar{Y }$ as in Figure \ref{fig:dp5base} (note that $\widetilde{Y}$ with the symplectic form from Lemma \ref{Kahler form} is not monotone), and let $(\check{Y}, W)$ be its LG mirror constructed from the SYZ fibration on $\widetilde{Y}\setminus \widetilde{D}$ (see Section \ref{sec:constsyznon-toric}). Then the quantum cohomology ring of $\widetilde{Y}$ is isomorphic to the Jacobian ideal ring of $W$.
\end{theorem}

We will give a concrete description of the LG mirror $(\check{Y},W)$ in \ref{subsec:lgmdp5}. The proof of Theorem \ref{thm:maindp5} is essentially to show that $W$ only has on $\check{Y}$ as many non-degenerate critical points as the rank of the cohomology of $\widetilde{Y}$, since $QH(\widetilde{Y})$ is known to be semi-simple \cite{BM}. This will be done in \ref{eqn:winonechamlim}.

Recall from Section \ref{subsec:LGmirrorY} that once we realize the surface as a non-toric blowup of a toric surface, we automatically obtain a LG model whose chamber structure is recorded in the scattering diagram $\mathfrak{D}^{LF}$. However, this LG model is not enough to capture the full mirror. 
In fact, it is crucial to enlarge it to include nodal fibres, as the mirror potential would miss a few geometrically meaningful critical points otherwise. Recall that in the local form of the SYZ fibration on $X$ \eqref{eqn:locfibbu}, the nodal fibre is given as $|x| = |x(q_i)|$ and $\mu_{S^1} = \epsilon_i$ where $q_i$ is the blowup point, and $\epsilon_i$ is the value of $\mu_{S^1}$ at the isolated $S^1$-fixed point $\tilde{q}_i$ in the exceptional divisor $E_i$. See Figure \ref{fig:loc_nontor}. We remark that even if the loop $|x| = |x(q_i)|$ in the reduction is slightly deformed, it still gives an immersed Lagrangian $S^2$ as long as it passes through $x(q_i)$. This way we have a small freedom in choosing the valuation of this Lagrangian (along the direction perpendicular to the divisor $\{x=0\}$).

\begin{figure}[h]
	\begin{center}
		\includegraphics[scale=0.5]{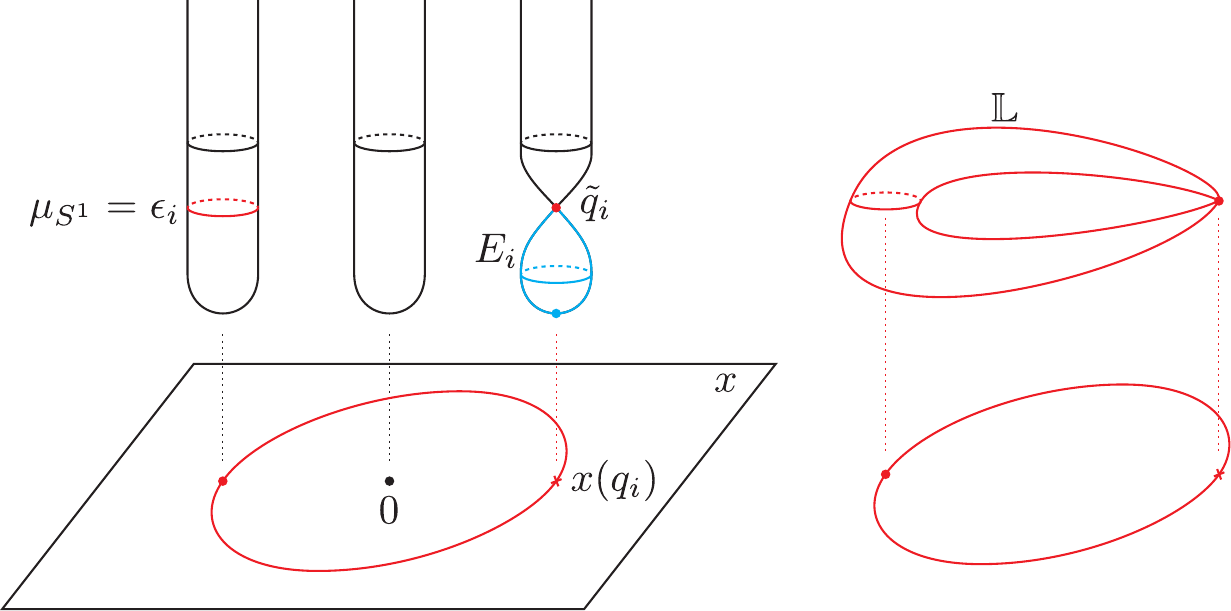}
		\caption{The nodal fibre in the local SYZ fibration on a non-toric blowup}
		\label{fig:loc_nontor}
	\end{center}
\end{figure}

\subsection{Degree 5 del Pezzo surface and its LG mirror}\label{subsec:lgmdp5}

We realize the degree $5$ del Pezzo surface $\widetilde{Y}$ as the blowup of $\mathbb{P}^2$ at two of the torus-fixed points and the other two on the interior of the toric boundary. The toric blowups correspond to chopping off two corners of the moment polytope of $\mathbb{P}^2$ as shown in Figure \ref{fig:dp5base}, and the resulting surface $\bar{Y}$ is a (toric) del Pezzo surface of degree $7$. We set $(a',0)$ and $(0,b')$ to be the locations of the latter two blowup points in the moment polytope of $\mathbb{P}^2$. 

The non-toric blowup in this case produces $\mathfrak{D}^{LF}$ which coincides with the $A_2$-scattering diagram. It has five chambers that support different mirror potentials, which agree with each other after suitable cluster transforms.
For instance, the shaded chamber in Figure \ref{fig:dp5base} (that contains the origin)\footnote{To be more precise, the genuine scattering diagram should be drawn on the $\log$-base.} carries the potential given by
\begin{equation}\label{eqn:winonecham}
W = z_1 + z_2 + (T^a + T^{c \pm \epsilon_1}) \frac{1}{z_1} + (T^b + T^{c \pm\epsilon_2} ) \frac{1}{z_2} + T^{a \pm \epsilon_1} \frac{z_2}{z_1} + T^{b \pm \epsilon_2} \frac{z_1}{z_2} + T^c \frac{1}{z_1 z_2}
\end{equation}
For computational simplicity, we will consider the limiting situation in which $\epsilon_1= \epsilon_2=0$ from now on. 
On the adjacent chambers, the potential takes the form of
$$W_{up}=z_1'  + (1+T^a) z_2' + z_1' z_2' + (T^a + T^c  ) \frac{1}{z_1'}  + T^{b} \frac{1}{z_2'} + T^a \frac{z_2'}{z_1'} + T^{c} \frac{1}{z_1' z_2'},$$
$$W_{right}=(1+T^b)z_1'' + z_2''  +z_1'' z_2'' + T^{a} \frac{1}{z_1''} + (T^b + T^c ) \frac{1}{z_2''} +T^b\frac{z_1''}{z_2''} + T^{c} \frac{1}{z_1 z_2},$$
 where the corresponding coordinate changes are given by
 \begin{equation}\label{eqn:coordchangesdp5}
 z_1=z_1'=z_1''(1+z_2'')^{-1}, \quad z_2=z_2'(1+z_1')^{-1} = z_2'' 
 \end{equation}
These coordinates cover the region
\begin{equation}\label{eqn:regiondp5m}
\begin{array}{c}
0<val(z_1) <a', \,\, 0<val(z_2) <b' \\
0<val(z_1')<a', \,\, b'<val(z_2') < b \\
a'<val(z_1'')<a, \,\, 0<val(z_2'')<b' 
\end{array} 
\end{equation}
(the same inequalities also describe corresponding chambers drawn in Figure \ref{fig:dp5base}). The computation for the other two chambers can be done similarly, and is omitted. By the consistency of $\mathfrak{D}^{LF}$, the five local charts are glued without ambiguity.   
  In order to glue these local LG models, one needs to perturb slightly a complex structure to make local patches overlapped, and use the same coordinate transition as given in \eqref{eqn:coordchangesdp5}. Notice that \eqref{eqn:coordchangesdp5} preserves the $T$-adic valuation.

\begin{figure}[h]
	\begin{center}
		\includegraphics[scale=0.5]{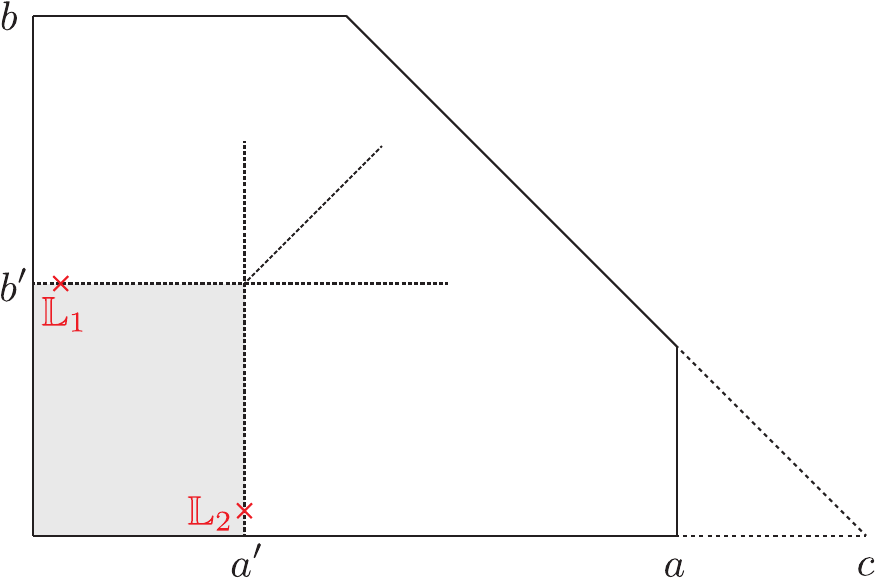}
		\caption{The degree $5$ del Pezzo surface as a non-toric blowup}
		\label{fig:dp5base}
	\end{center}
\end{figure}

In what follows, we enlarge the above LG model to obtain the full mirror $(\check{Y},W)$ by adding two more charts, each isomorphic to $(\Lambda_+)^2$.
Observe that the fibration comes with two immersed Lagrangians $\mathbb{L}_1$ and $\mathbb{L}_2$. We denote by $U_i$ the immersed generators of the Floer complex of $\mathbb{L}_i$ whose associated smoothing produces the torus fibre sitting in the shaded chamber in Figure \ref{fig:dp5base}. Let $V_i$ denote the immersed generator complementary to $U_i$.
For $u_i,v_i \in \Lambda_+^1$, one can take a linear combination $b_i = u_i U_i + v_i V_i$ for $\mathbb{L}_i$, and boundary-deform $\mathbb{L}_i$ into a new Lagrangian brane $(\mathbb{L}_i,b_i)$. One should think of $(\mathbb{L}_i,b_i)$ as a family of objects in the Fukaya category parametrized by the weak Maurer-Cartan space of $\mathbb{L}_i$. It is shown in \cite[3.1]{HKL} that $b_i=(u_i,v_i)$ solves the weak Maurer-Cartan equation for $\mathbb{L}_i$ using reflection symmetry.

Observe that we have a freedom to choose $\mathbb{L}_i$ by adjusting the size of the base circle drawn in Figure \ref{fig:loc_nontor} (since the circle does not have to be concentric about $0$). By suitably choosing its size, one may assume that $(\mathbb{L}_i,b_i=(u_i,v_i))$ is  quasi-isomorphic to the Lagrangian torus fibre represented 
by $z_1=T^{x_1} \rho_1, z_2 = T^{x_2} \rho_2$ (with $x_i=val(z_i)$ and $val(\rho_i)=0$) where $(u_1,v_1)$ and $(z_1,z_2)$ are related by the coordinate change
\begin{equation}\label{eqn:uv1cc1}
u_1 v_1 -1 = \rho_1,\quad u_1 = T^{x_2 - b''} \rho_2.
\end{equation}
Here, $b''$ is the symplectic area of the holomorphic disc bounded by $\mathbb{L}_1$ whose smoothing at the $U_1$-corner lies in the same class as the $z_2$-disc. While it can be chosen to be very close to $b''$, a further adjustment may be needed in order to have a genuine torus fibre satisfying $val(u_1) =x_2 -b'' >0$ (note that we do have such a flexibility in the construction of $\mathbb{L}_1$). One can rewrite \eqref{eqn:uv1cc1} as
\begin{equation}\label{eqn:uv1cc2}
u_1 v_1 -1 = T^{-val(z_1)} z_1,\quad u_1 = T^{-b''} z_2.
\end{equation}
Analogously defined coordinates $(u_2,v_2)$ on the weak Maurer-Cartan deformation space of $\mathbb{L}_2$ are related with others by
\begin{equation}\label{eqn:uv2cc2}
u_2 v_2 -1 =   T^{-val(z_2)} z_2,\quad u_2  = T^{-a''} z_1.
\end{equation}

Notice that there exists $(u_1,v_1)$ satisfying \eqref{eqn:uv1cc2} as long as $ val(T^{val(z_1)} +1) > val(z_2) - b''$ and $val(z_2) > b''$ regardless of whether or not $(z_1,z_2)$ corresponds to an actual torus fibre, and the similar is true for $(u_2,v_2)$.
For instance, $(z_1,z_2)$ with $val(z_1)=0$ or $val(z_2)=0$ can represent a geometric object that is supported over an immersed Lagrangian. Obviously, torus fibres themselves cannot have this feature. In other words, the chart $(\Lambda_+)^2$ from $(u_i,v_i)$ non-trivially enlarges the original region parametrized $(z_1,z_2)$ that supports torus fibres. 

One can argue similarly to find relations between this $(\Lambda_+)^2$-chart with the chambers adjacent to  the shaded one in Figure \ref{fig:dp5base}, which are compatible with  \eqref{eqn:uv1cc2} and \eqref{eqn:uv2cc2} (see \cite[3.2]{HKL} for more details).
We conclude that $\mathbb{L}_i$ strictly extends the LG model obtained from torus fibres, i.e. the LG model constructed by gluing \eqref{eqn:winonecham} with (local) ones from the other four chambers in Figure \ref{fig:dp5base}. In summary, the domain of the mirror LG model $\check{Y}$ are glued from the five open subsets given as in \eqref{eqn:regiondp5m} (two of them were omitted in the equation) together with two $(\Lambda_+)^2$-charts parametrized by $(u_1,v_1)$ and $(u_2,v_2)$ via the transition maps given in \eqref{eqn:uv1cc2} and \eqref{eqn:uv2cc2}, respectively.

\begin{remark}\label{rmk:cubicrel}
If two blowup points are close enough, then $\mathbb{L}_1$ and $\mathbb{L}_2$ can be made quasi-isomorphic to each other after small perturbation (shifting the corresponding base circles and adjusting their sizes as before). Assuming the convergence, we may put $T=1$ to work over $\C$. Then one can derive from \eqref{eqn:uv1cc2} and \eqref{eqn:uv2cc2} the following coordinate transition between complex coordinates $(u_1,v_1)$ and $(u_2,v_2)$ 
$$u_1 v_1 - 1 =  u_2, \quad  u_2 v_2 -1 =  u_1.$$
 Setting $x= u_1, y=-v_1, z=v_2$, we have $-xy-1 = u_2 =( x+1) z^{-1}$,
$$ xyz  +x+z+1 =0,$$
which describes the $A_2$ cluster variety as an affine hypersurface. 
\end{remark}

\subsection{Critical points of the mirror potential}\label{subsec:critdp57}

Let us now compute the critical points of the potential 
\begin{equation}\label{eqn:winonechamlim}
W = z_1 + z_2 + (A+C) \frac{1}{z_1} + (B+C) \frac{1}{z_2} + A \frac{z_2}{z_1} + B\frac{z_1}{z_2} + C \frac{1}{z_1 z_2},
\end{equation}
where we put $A:=T^a$, $B:=T^b$ and $C:=T^c$ for simplicity.
Elementary computation tells us that the critical points of $W$ should satisfy 
\begin{equation}\label{eqn:partder12}
\begin{array}{c}
 z_1^2 z_2- (A+C) z_2 - A z_2^2 + B z_1^2 - C  =0,\\
 z_1 z_2^2 - (B+C) z_1 - B z_1^2 + A z_2^2 - C =0.
 \end{array}
\end{equation}
Subtracting the bottom multiplied by $z_1$ from the above multiplied by $z_2$, we obtain
$$  (z_1 + z_2 +1) (z_1(C + Bz_1) -z_2(C + A z_2)) =0.$$
We then proceed by dividing into two cases depending on which of the two factors vanishes.

\vspace{0.3cm}

\noindent{\bf Case 1:} If $z_1 + z_2 +1 =0$, then the sum of two equations in \eqref{eqn:partder12} implies $ z_1 z_2 + B z_1 + A z_2 + C = 0$, and hence  
$$  z_1^2+  (A-B +1) z_1 +A -C =0.$$
Thus we obtain two solutions
$$ z_1= \frac{-A+B-1 \pm \sqrt{1  -2A - 2B  - 2AB +A^2 + B^2 + 4C}}{2}$$
$$ z_2= \frac{A-B-1 \mp \sqrt{1  -2A - 2B  - 2AB +A^2 + B^2 + 4C}}{2}$$
Using the expansion
$$\sqrt{1  -2A - 2B  - 2AB +A^2 + B^2 + 4C} = 1-A -B + T^{higher},$$
one of the solutions takes the form of
$$(z_1,z_2) = (-1 +B + T^{higher}, -B + T^{higher} ),$$
but $val(z_1)=0$ is not valid if we restrict to torus fibres only. However, transferring to $(u_1,v_1)$-coordinate charts via
$$u_1 v_1 -1   = z_1,\quad u_1 =  T^{-b'} z_2,$$
we have a legitimate Lagrangian brane supported over the immersed Lagrangian $\mathbb{L}_1$, since $val(u_1)>0$ and $val(v_1)>0$.
Likewise, the other solution
$$(z_1,z_2) = (-A + T^{higher}, -1 + A +T^{higher})$$
does not make sense in the original coordinates from torus fibres, but making use of
$$u_2 v_2 -1  = z_2,\quad u_2 =   T^{-a'} z_1,$$
one can replace it by the Lagrangian brane supported over $\mathbb{L}_2$.

\vspace{0.3cm}

\noindent{\bf Case 2:}
 Let us now look into the case of $z_1 (C+Bz_1) - z_2 (C+A z_2) = 0$.
We introduce an extra variable $\lambda$ to keep the symmetry of the equation,
\begin{equation}\label{eqn:introlambdaab}
(Bz_1 + C) =\lambda z_2, \quad (Az_2 + C) =\lambda z_1
\end{equation}
which makes sense since neither of $z_1, z_2$ is zero.
From the sum of two equations in \eqref{eqn:partder12}, we have
$$(z_1 z_2 + 1-\lambda) (z_1 + z_2) = 0.$$
If $z_1+z_2=0$, then $(z_1,z_2)=(\frac{2C}{A-B}, \frac{2C}{B-A})$, whose valuation lives outside the allowed region, nor can it be replaced by appropriate boundary deformation of $\mathbb{L}_i$. Thus this critical point is not geometric in our SYZ setting. 
We speculate that there is \emph{no} nontrivial object in the Fukaya category which represents this point, as otherwise its Floer cohomology with any of torus fibres would vanish.

Therefore geometrically meaningful is to examine the case $\lambda = z_1 z_2 +1$. Making use of \eqref{eqn:introlambdaab}, one can easily find that the critical points are given as
\begin{equation}\label{eqn:critz1z255}
 (z_1,z_2) = \left( \frac{1}{\lambda^2 - AB} (\lambda C + AC), \frac{1}{\lambda^2 - AB} ( \lambda C + BC)\right)
\end{equation}
where $\lambda$ satisfies the degree-5 polynomial equation
$$ \lambda^5 -\lambda^4 - 2 AB \lambda^3 +(2AB-C^2)\lambda^2 + (A^2 B^2 - C^2 (A+B)) \lambda - AB C^2 - A^2 B^2=0.$$
Hence, for generic $A,B$ and $C$, we obtain $5$ mutually distinct critical points. Given the two critical points in the previous case, it only remains to check that the five points are represented by torus fibres.
  Observe that even if the valuation of a critical point escapes the chamber that supports our $W$ \eqref{eqn:winonechamlim}, one can interpret it as the critical point from some other chamber (and vice versa) by applying the valuation-preserving coordinate change \eqref{eqn:coordchangesdp5} as long as its valuation lies in the moment polytope. Note that the Jacobian ideal ring is not affected by coordinate changes.
Therefore it is enough to work with $W$ after enlarging its domain to include the other four chambers and two more $(\Lambda_+)^2$-charts. 

Taking valuations of both sides of $\lambda = z_1 z_2 +1$, 
\begin{equation}\label{eqn:valz1z2lambda}
val(z_1) + val (z_2) = val(\lambda)
\end{equation}
which we will use quite often below. We subdivide the argument into the following four cases.

\begin{itemize}
\item[(i)] If $val(z_1) \leq val (C) - val (B)$ and $val(z_2) \leq val (C) - val (A)$,  then \eqref{eqn:introlambdaab} leads to
$$val(B) + val(z_1) = val(\lambda) + val(z_2)$$
$$val(A) + val (z_2) = val(\lambda) + val(z_1)$$
so that
$$  val(z_1) - val(z_2) = \frac{1}{2} (val(A)-val(B)).$$
On the other hand, combining with \eqref{eqn:valz1z2lambda} we see that
$$  val(z_1) + val(z_2)= val(\lambda)  = \frac{1}{2} (val(A) + val (B)).$$
Hence $val (z_1) = \frac{1}{2} val(A)$ and $val(z_2) = \frac{1}{2} val (B)$, which lie inside the allowed region.

\item[(ii)] If $val(z_1) \leq val (C) - val (B)$ and $val(z_2) > val (C) - val (A)$, we have 
$$val(B) + val(z_1) = val(\lambda) + val(z_2)$$
$$val(C) = val(\lambda) + val(z_1),$$

With \eqref{eqn:valz1z2lambda}, the above implies
$$  val(z_1)=\frac{1}{2} val(C) - \frac{1}{4}val(B), \quad   val(z_2) = \frac{1}{2} val(B).$$
It is not difficult to see that the corresponding point lie inside the region.

\item[(iii)] The case where $val(z_1) > val (C) - val (B)$ and $val(z_2) \leq val (C) - val (A)$ is symmetric to (ii), and we omit.

\item[(iv)] Finally, if $val(z_1)> val (C) - val (B)$ and $val(z_2) > val (C) - val (A)$, then  
$$val(C) = val(\lambda) + val(z_2)$$
$$val(C) = val(\lambda) + val(z_1)$$
Using \eqref{eqn:valz1z2lambda}, we find that 
$$val(z_1) = val(z_2) = \frac{1}{3} val(C).$$
\end{itemize}
Thus all the five critical points given by \eqref{eqn:critz1z255} admit Lagrangian torus fibres with suitable $\Lambda_U$-holonomies. Note, however, that the above calculation does not tell us exactly the chamber which contains a given critical torus fibre. In fact, such a chamber can vary if we change the K\"{a}hler parameter $A,B$ and $C$.

\subsubsection*{Non-degeneracy of the critical points of $W$ and the mirror symmetry for the degree $5$ del Pezzo}

It is well-known that the quantum cohomology of the degree $5$ del Pezzo is semi-simple, and that its rank is $7$ (see \cite{BM}). Therefore, it suffices to show that all the seven critical points of the potential $W$ \eqref{eqn:winonechamlim} are non-degenerate to have Theorem \ref{thm:maindp5}.
The following is an elementary calculation proving their non-degeneracy, which reflects (and is essentially equivalent to) the fact that the choice of generic $A$, $B$ and $C$ in Section \ref{subsec:critdp57} is to achieve the maximal number of critical points of $W$.

Let $f:=\partial_{z_1} W$ and $g:= \partial_{z_2} W$. The non-degeneracy of a critical point of $W$ can be reformulated as the transversality of the intersection between $\{f=0\}$ and $\{g=0\}$. More concretely, we need to check whether or not $df_p$ and $dg_p$ is linearly independent for each intersection point $p$ (i.e. a critical point of $W$). Consider two linear combinations $\lambda_1 f + \lambda_2 g$ and $\eta_1 f + \eta_2 g$ with $\lambda_i$ and $\eta_i$ generic enough in the sense that $\det \left(\begin{array}{cc} \lambda_1 & \lambda_2 \\ \eta_1 & \eta_2 \end{array}\right) $ does not vanish at any $ p \in \{f=0\} \cap \{g=0\}$. Then it is easy to see that the transversality between $\{f=0\}$ and $\{g=0\}$ at $p$ is equivalent to that between $\{\lambda_1 f + \lambda_2 g =0\}$ and $\{ \lambda_1 f + \lambda_2 g\}$ at $p$. 
The strategy is to find such  linear combinations which factor into lower degree polynomials so that checking transversality becomes a bit simpler.

Going back to our situation, we may begin with
\begin{equation}\label{eqn:fandg2}
\begin{array}{c}
f:= z_1^2 z_2- (A+C) z_2 - A z_2^2 + B z_1^2 - C ,\\
 g:= z_1 z_2^2 - (B+C) z_1 - B z_1^2 + A z_2^2 - C 
\end{array}
\end{equation}
as in \eqref{eqn:partder12}. Recall that
$$z_1 f - z_2 g = (z_1 + z_2 +1) (z_1(C + Bz_1) -z_2(C + A z_2)).$$
On the other hand, we have
$$ f+ g = (z_1 + z_2)(z_1 z_2 -C) - Bz_1 - Az_2 - 2C = 0.$$
Set $\alpha:= z_1 + z_2 +1$, $\beta:= z_1(C + Bz_1) -z_2(C + A z_2)$ and $\gamma:= (z_1 + z_2)(z_1 z_2 -C) - Bz_1 - Az_2 - 2C$.
By the above discussion, a sufficient condition for $W$ being Morse is the transversality between $\{\alpha=0\}$ and $\{ \gamma=0\}$ and that between $\{\beta =0 \}$ and $\{ \gamma=0\}$. The former can be checked straightforwardly, whereas the latter involves higher order terms. 
Similarly to before, let us introduce new variables $\lambda=\frac{B z_1 + C}{z_2}$ and $\lambda'=\frac{A z_2 + C}{z_1}$. In these variables, one has
$$\beta = (\lambda - \lambda') z_1 z_2 $$
One can freely replace $\beta$ by $\lambda - \lambda'$, and $\lambda$ can be taken as a coordinate of $\{\beta=0\}$. $(\lambda (z_1,z_2),\lambda' (z_1,z_2) )$ is singular along some affine plane, which finitely many critical points of $W$ can easily avoid by choosing generic parameters.
 
Observe that $\gamma|_{\{\beta=0\}}$ can be rewritten as
$$(z_1 + z_2) (z_1 z_2 - C) -\lambda z_1 - \lambda' z_2  = \gamma = (z_1 + z_2)(z_1 z_2 + 1 -\lambda)$$
with $z_1 = C\frac{\lambda + A}{\lambda^2 - AB}$ and $z_2 =  C\frac{ \lambda + B}{\lambda^2 - AB}$.
Denote the last two factors by $\gamma_1$ and $\gamma_2$, respectively.
Finally, the question of $W$ being Morse boils down to ask if $0$ is a critical value of $\gamma_i|_{\{\beta=0\}}$ for $i=0,1$.
From our earlier discussion in Section \ref{subsec:critdp57}, we can actually ignore $\gamma_1|_{\{\beta=0\}}$ since $0$ is not a critical value for a critical point lying in the domain of definition of $W$. On the other hand, $(\lambda^2 - AB)^2 \gamma_2|_{\{\beta=0\}}$ is explicitly given as
$$  -\lambda^5 +\lambda^4 + 2 AB \lambda^3 +(C^2-2AB)\lambda^2 + (C^2(A+B)-A^2 B^2 ) \lambda + AB C^2 + A^2 B^2,$$
which has five zeroes for generic $A$, $B$ and $C$, and hence $0$ is not a critical for such $A$, $B$ and $C$.

\end{document}